\documentclass[12pt,twoside]{amsart}
\usepackage{amsmath,amsfonts,amssymb,amsthm,enumitem,comment,dirtytalk}

\newtheorem{theorem}{Theorem}[section]

\newtheorem{proposition}[theorem]{Proposition}

\newtheorem{corollary}[theorem]{Corollary}

\newtheorem{lemma}[theorem]{Lemma}

\theoremstyle{definition}
\newtheorem{definition}[theorem]{Definition}
\newtheorem{example}[theorem]{Example}
\newtheorem{notation}[theorem]{Notation}
\newtheorem{remark}[theorem]{Remark}

\title[On rank one log del Pezzo surfaces in char $\neq 2,3$]{On rank one log del Pezzo surfaces in characteristic different from two and three}
\author{Justin Lacini}
\date{}
\address{Department of Mathematics, University of California, San Diego, USA}
\email{jlacini@ku.edu}
\curraddr{Department of Mathematics, University of Kansas, 1450 Jayhawk Blvd. , Lawrence, KS 66045}

\begin{document}

\begin{abstract}
We classify all log del Pezzo surfaces of Picard number one defined over algebraically closed fields of characteristic different from two and three. We also
discuss some consequences of the classification. For example, we show that log del Pezzo surfaces of Picard number one defined over algebraically closed fields
of characteristic higher than five admit a log resolution that lifts to characteristic zero over a smooth base.
\end{abstract}

\footnote{2010 Mathematics Subject Classification. 14J26 and 14J45.}
\footnote{Keywords. Log del Pezzo surface, Minimal Model Program, positive characteristic, classification.}

\maketitle

\tableofcontents

\section{Introduction}

The canonical divisor plays a central role in the classification of algebraic varieties. For example, the Minimal Model Program predicts that every smooth projective variety can be 
\say{built} from the following three families:

\begin{enumerate}
\item Fano varieties, on which the canonical class is anti-ample.
\item Varieties with canonical class that is trivial (for example Calabi-Yau or abelian varieties).
\item Canonical models, on which the canonical class is ample.
\end{enumerate}

The aim of this paper is to classify all klt Fano surfaces (called \emph{log del Pezzo surfaces}) 
of Picard number one over algebraically closed fields of characteristic different from two and three.
More precisely:

\begin{theorem}[Classification of rank one log del Pezzo surfaces, Theorem \ref{total_classification}]\label{introclassification}
Let $S$ be a rank one log del Pezzo surface defined over an algebraically closed field of characteristic different from two and three.
If $S$ is smooth then $S=\mathbb{P}^2$; if $S$ is Gorenstein and singular, then $S$ is one of the surfaces described in Theorem \ref{gorenstein_classification};
otherwise, it belongs to one of the families LDP1 to LDP22 described in Section \ref{sectionclassification}.
\end{theorem}

Theorem \ref{introclassification} has several immediate consequences. For instance, it is well-known that rank one log del Pezzo surfaces defined over the complex numbers have at most
four singular points (see \cite[Corollary 1.8.1]{mckernan} and \cite{belousov}). This result, known as the Bogomolov bound, fails in characteristic two, as Keel and M\textsuperscript{c}Kernan 
exhibited examples of rank one log del Pezzo surfaces with arbitrarily many singular points (\cite[Chapter 9]{mckernan}). Nevertheless, in light of Theorem \ref{introclassification}
we have the following:

\begin{corollary}[Bogomolov bound]
Let $S$ be a rank one log del Pezzo surface defined over an algebraically closed field of characteristic different from two, three and five. Then $S$ has at most four singular points.
\end{corollary}

By slightly modifying Keel and M\textsuperscript{c}Kernan's example, in characteristic three one gets a rank one log del Pezzo surface with seven singularities
(see Example \ref{char3}). As a consequence of the classification in Theorem \ref{introclassification}, we show that there is a counterexample to the
Bogomolov bound in characteristic five.

\begin{theorem}[Example \ref{char5}]\label{introchar5}
There exists a rank one log del Pezzo surface $S$ defined over any algebraically closed field of characteristic five and a curve $C\subseteq S$ such that
\begin{enumerate}
\item $-(K_S+C)$ is ample.
\item $K_S+C$ is dlt.
\item $S$ has five singular points.
\end{enumerate}
\end{theorem}

There has been recently a great deal of interest in extending the classical results of the Minimal Model Program to algebraically closed fields of positive characteristic 
(see for instance \cite{haconxu}). 
It is therefore natural to ask which properties of complex log del Pezzo surfaces carry to positive characteristic.
Cascini, Tanaka and Witaszek \cite{cascini} have proved that in large characteristic all log del Pezzo surfaces either admit a log resolution that lifts to characteristic zero over a smooth base, or
are globally F-regular. Our desire to determine the exact characteristic has been one of the main motivations behind Theorem \ref{introclassification}. 

\begin{theorem}[Lifting to characteristic zero, Theorem \ref{lift_theorem}]\label{introlift}
Let $S$ be a rank one log del Pezzo surface defined over an algebraically closed field of characteristic $p>5$ and consider the minimal resolution $\pi:\tilde{S}\rightarrow S$.
Then $(\tilde{S}, \operatorname{Ex}(\pi))$ lifts to characteristic zero over a smooth base.
\end{theorem}

Notice that Theorem \ref{introlift} is sharp, in light of Theorem \ref{introchar5} and of the Bogomolov bound. Furthermore, Theorem \ref{introlift} implies that Kodaira's vanishing theorem
holds for rank one log del Pezzo surfaces in characteristic $p>5$ (see Theorem \ref{kodairavanishing}). 
This behavior in dimension two should be contrasted with the failure of Kodaira's vanishing theorem
in higher dimension, even for smooth Fano varieties (see \cite{totaro}).

We give now a brief sketch of the proof of Theorem \ref{introclassification}. Let $S$ be a log del Pezzo surface of Picard number one. 
A natural approach is to try to \say{simplify} the singularities of $S$ by extracting an exceptional divisor $E_1$ of its minimal resolution. Let $f:T_1 \rightarrow S$ be the extraction of $E_1$.
Since $S$ has Picard number one, $T_1$ has Picard number two, and therefore the closed cone of curves of $T_1$ is generated by two rays. 
One of the two rays is generated by the class of $E_1$. We may then play the two-ray game on $T_1$ by contracting the other ray and we
let $\pi:T_1 \rightarrow S_1$ be the contraction. There is a priori no reason why $S_1$ would be any \say{simpler} that $S$, and in fact this is not the case for most choices of $E_1$.
The truly remarkable fact, which makes the classification possible at all, is that by choosing $E_1$ to be the divisor with the worst singularity (as measured by the discrepancy), then $S_1$ is 
indeed simpler than $S$, and one may even classify the possibilities for the contraction $\pi_1$. 

The idea is then to continue this sequence of extractions and contractions, which produces a sequence of increasingly simpler surfaces $S_1$, $S_2$ and so on. 
This process, which was first introduced by Keel and M\textsuperscript{c}Kernan \cite{mckernan}, is called the \emph{hunt}. The hunt is very efficient and usually terminates within
three steps, yielding either a Gorenstein log del Pezzo surface of Picard number one, a cone over the rational normal curve of degree $n$, or a Mori fiber space. 
One may therefore recover $S$ by classifying all such surfaces and the contractions $\pi_i$ that appear during the hunt.

Keel and M\textsuperscript{c}Kernan introduced the hunt in order to prove that the smooth locus of log del Pezzo surfaces of rank one is uniruled. Their proof is divided in two cases, based
on the following notion:

\begin{definition}
Let $(X,\Delta)$ be a $\mathbb{Q}$-factorial projective log pair. A \emph{special tiger} for $K_X + \Delta$ is an effective $\mathbb{Q}$-divisor $\alpha$ such that
$K_X + \Delta + \alpha$ is numerically trivial, but not klt. If there is a special tiger, a \emph{tiger} is any divisor $E$ with discrepancy 
at most $-1$.
\end{definition}

For surfaces that admit a tiger they presented a short proof by using deformation theory. To complete their argument, however, they classified all complex log del Pezzo surfaces with no tigers
(more precisely, they constructed a family of surfaces that contains all those that have no tigers). By pushing these methods a bit further, they actually classified all simply connected rank one
log del Pezzo surfaces, \emph{with the exception of a bounded family}.

In this paper we use Keel and M\textsuperscript{c}Kernan's 
ideas to complete the classification over the complex numbers, and to extend it to algebraically closed fields of characteristic $p>3$. In positive characteristic
a whole set of additional difficulties appear. First off, one cannot use topological arguments in order to simplify the classification, such as reducing to the simply connected case.
Furthermore, one cannot use the Bogomolov bound as we do not \emph{a priori} know whether it holds in characteristic $p>5$ or not. Other issues are that the classification of rank one
Gorenstein log del Pezzo surfaces was not available in positive characteristic (to the best of the author's knowledge), 
and that the proof of \cite[Lemma 22.2]{mckernan} on the existence of complements
does not \emph{a priori} carry through in positive characteristic, as it uses the Kawamata-Viehweg vanishing theorem.

The rest of the paper is organized as follows. In Appendix \ref{surfacesings} we describe in detail klt, dlt and lc surface singularities and techniques to deal with them efficiently.
In Appendix \ref{gorenstein} we classify Gorenstein log del Pezzo surfaces in characteristic different from two and three. In Section \ref{sectionhunt} 
we recall the main results concerning the hunt from \cite{mckernan}, in an effort to make the presentation as self-contained as possible. We also develop the hunt in the level case,
which will play a role in classifying log del Pezzo surfaces with tigers. In Section \ref{sectionnotigers} we start analysing
the hunt for log del Pezzo surfaces that do not have tigers, and in Section \ref{sectiontigers} we deal instead with the case in which there are tigers. 
As a byproduct, we classify all pairs $(S,C)$ such that $S$ is a rank one log del Pezzo surface and $C\subseteq S$ is a curve such that $K_S+C$ is anti-nef.
In Section \ref{sectionclassification} we summarize our findings and list all rank one log del Pezzo surfaces. We conclude by considering liftability to characteristic zero
and other applications in Section \ref{applications}.

\ \

\noindent \textbf{Acknowledgements:} I am indebted to my PhD advisor Prof. James M\textsuperscript{c}Kernan for suggesting me the problem and for carefully reading through the many preliminary drafts of this paper. I would also like to thank Fabio Bernasconi for helpful comments on the exposition and Masaru Nagaoka
for pointing out some missing cases in the classification on an earlier version. The author has been supported
by NSF research grants no: 1265263 and no: 1802460 and by a grant from the Simons Foundation \#409187.

\section{Notation}

A \textbf{pair} $(X,\Delta)$ is given by a normal projective variety $X$ and a Weil $\mathbb{Q}$-divisor $\Delta$ such that $K_X+\Delta$ is $\mathbb{Q}$-Cartier.
If $\Delta$ is effective, then $(X,\Delta)$ is a \textbf{log pair}. We say that $\Delta$ is a \textbf{boundary}
if $\Delta=\sum_i c_i D_i$ where $D_i$ are irreducible distinct Weil divisors and $0\leqslant a_i\leqslant 1$.
We say that a birational morphism $f:Y\rightarrow X$ is a \textbf{log resolution} of $(X,\Delta)$ if $Y$ is smooth and $f_* ^{-1}(\Delta)+\operatorname{Ex}(f)$ is a simple normal crossings divisor.

Let $f:Y\rightarrow X$ be any birational morphism with $Y$ normal.
We can write $K_Y+f_* ^{-1}\Delta = f^*(K_X+\Delta)+\sum_i a_i E_i$ where $E_i$ are $f$-exceptional divisors. 
The numbers $a_i=a_i(E_i; X,\Delta)$ are called \textbf{the discrepancy} of $E_i$ with respect to the pair $(X,\Delta)$. We define \textbf{the coefficient} 
of $E_i$ with respect to the pair $(X,\Delta)$ to be $e(E_i; X,\Delta)=-a_i$. This is just the coefficient with which $E_i$ appears in the divisor $\Gamma$ defined by
$K_Y+\Gamma = f^*(K_X+\Delta)$. We call $\Gamma$ the \textbf{log pullback} of $\Delta$.

We say that $(X,\Delta)$ is \textbf{log canonical} (or simply lc) if $a_i\geqslant -1$ for every $E_i$ and every $f$. 
We say that $(X,\Delta)$ is \textbf{Kawamata log terminal} (or simply klt) if $\lfloor \Delta \rfloor=0$ and $a_i>-1$ for every
$E_i$ and every $f$. It is in fact sufficient to check the above definitions on any given log resolution of $(X,\Delta)$. 
Finally, if $\operatorname{dim}(X)=2$ we say that $(X,\Delta)$ is \textbf{divisorially log terminal} (or simply dlt) if there is a log resolution $f:Y\rightarrow (X,\Delta)$ such that $a(E_i; X,\Delta)>-1$
for all $f$-exceptional divisors. This is equivalent to requiring that there is a closed subset $Z\subseteq X$ such that $(X\setminus Z, \Delta_{|_{X\setminus Z}})$ has simple normal crossings and, if $E$ is an irreducible divisor over $X$ with center contained in $Z$, then $a(E;X,\Delta)>-1$.

In this paper we will mainly use the above definitions in the case of surfaces. If $S$ is a normal surface we indicate by $\tilde{S}$ its minimal resolution. If $C\subseteq S$ is an
effective divisor, then we indicate by $\tilde{C}\subseteq \tilde{S}$ its strict transform and by $g(C)$ its arithmetic genus. We indicate by $\mathbb{F}_n$ the Hirzebruch surface
 $\mathbb{P}(\mathcal{O}\oplus \mathcal{O}(n))$ over $\mathbb{P}_k ^1$. We denote by $\overline{\mathbb{F}}_n$ the surface obtained by contracting the unique $(-n)$ curve
 of $\mathbb{F}_n$. A \textbf{log del Pezzo surface} is a projective surface $S$ with only klt singularities such that $-K_S$ is ample. 
 We often call the Picard number of a log del Pezzo surface simply as \textbf{rank}. For example, $\overline{\mathbb{F}}_n$ is a log del Pezzo surface of rank one.

\section{The hunt}\label{sectionhunt}

Let $S$ be a rank one log del Pezzo surface defined over an algebraically closed field $k$ of any characteristic.
In this section we describe how to produce a sequence of progressively simpler log del Pezzo surfaces. This process is called the \say{hunt}, and will be the main tool in our analysis. 
A hunt step consists of a $K$-positive extraction and a $K$-negative contraction. The reason that the hunt is so useful is that it is possible to classify its birational transformations.
Therefore, if we reach a log del Pezzo surface that we fully understand, then we may get useful information about $S$ by reversing the process. 

The hunt and its properties are described in \cite[Chapter 8]{mckernan}. In addition to the hunt for flush pairs, which is described in \cite{mckernan}, we will also need the hunt for level pairs
(see Definition \ref{defflush} for the definition of flush and level). Throughout this section we adopt the following practice: we will state all the relevant results for which a proof can be found in
\cite{mckernan} and provide the reference within that paper, whereas we will prove all the further results that we need, 
even if this only consists in making a minor change to a proof from \cite{mckernan}. 
We hope in this way to keep the exposition as self-contained and detailed as possible, while avoiding repetitions with \cite{mckernan}.

\subsection{Flush and level}

Here we develop the general theory that will be used during the hunt. In this subsection $\Delta$ is a boundary with support $D$.

\begin{definition}
Let $(X,\Delta)$ be a log pair and suppose that $\Delta=\sum_i a_i D_i$ is a boundary (we assume that $D_i$ are all irreducible and distinct).
We define $m=m(\Delta)$ be the minimum of the non-zero $a_i$, with the convention that if $\Delta$ is empty then $m(\Delta)=0$.
\end{definition}

\begin{definition}\label{defflush}
Let $(X,\Delta)$ be a log pair and suppose that $\Delta$ is a boundary. Let $E$ be an exceptional divisor over $X$. We say that the pair $(X,\Delta)$ is:
\begin{enumerate}
\item flush (respectively level) at $E$ if $e(E; X,\Delta)<m$ (respectively $e(E; X,\Delta)\leqslant m$).
\item flush (respectively level) if $e(E; X,\Delta)<m$ (respectively \\ $e(E; X,\Delta)~\leqslant~m$) for all exceptional divisors $E$.
\end{enumerate}
\end{definition}

\begin{lemma}\label{flushdescent}
Let $X$ be a $\mathbb{Q}$-factorial variety and let $\Delta$ be a boundary. Suppose that $f:Y\rightarrow X$ is a birational
morphism with irreducible divisorial exceptional locus $E$. Assume that $e=e(E; X, \Delta)\geqslant 0$. Let $\Gamma=eE+f_* ^{-1}\Delta$ be the log pullback of $\Delta$.
Assume $\Gamma$ is a boundary. Then:

\begin{enumerate}
\item If $(X,\Delta)$ is level and $E$ has maximal coefficient for $(X,\Delta)$ then $K_Y + \Gamma$ is level.

\item If $(Y,\Gamma)$ is flush (respectively level), then $(X,\Delta)$ is not flush (respectively level) for a divisor $F$ if and only $F=E$, 
and the coefficient of $E$ in $\Gamma$ is at least as large as (respectively strictly larger than) the coefficient of some non exceptional component of $\Gamma$.

\end{enumerate}

\end{lemma}

\begin{proof}
Both statements immediately follow from the definitions.
\end{proof}

The following lemma provides a useful criterion to check the flush and level properties. 

\begin{lemma}\label{flushresolution}
Let $S$ be the germ of a klt surface at $p$ in the \'{e}tale topology. 

\begin{enumerate}
\item If $(S,D)$ has normal crossings (in particular, $p$ is smooth) then $(S,\Delta)$ is level. If furthermore
$\lfloor \Delta \rfloor = 0$ then $(S,\Delta)$ is flush.

\item Assume $(S,D)$ does not have normal crossings (that is either $p$ is singular, or $D$ has worse than a simple node). Let $f:T\rightarrow S$ be a log resolution
of $(S,\Delta)$ and assume that $e(F; S, \Delta)\leqslant 1$ for every $f$-exceptional divisor $F$. Then for any exceptional divisor $V$ there is an $f$-exceptional divisor $F$
such that $e(V; S,\Delta)\leqslant e(F; S, \Delta)$. In particular, $(S,\Delta)$ is flush (respectively level) if and only if it is flush (respectively level) at all exceptional divisors
of $f$.
\end{enumerate}
\end{lemma}

\begin{proof}
This is \cite[Lemma 8.3.2]{mckernan}. We only remark that it is enough to check the property of being flush and level in an \'{e}tale neighborhood, so that the same proof
as in \cite{mckernan} applies. 
\end{proof}

Now we remark that in the case of surfaces it suffices to check the property of being divisorially log terminal or log canonical for divisors in the minimal resolution. 
This will play an important role throughout the paper.

\begin{lemma}\label{minimalresolution}
Let $S$ be the germ of a klt surface at $p$ in the \'{e}tale topology. Assume that $p$ is singular. 
\begin{enumerate}
\item If the coefficient of every exceptional divisor of the minimal resolution $\pi: \tilde{S}\rightarrow S$ for $(S,D)$ is strictly less than one, then $(S,D)$ is dlt.

\item If the coefficient of every exceptional divisor of the minimal resolution $\pi: \tilde{S}\rightarrow S$ for $(S,D)$ is less or equal than one, then $(S,D)$ is lc, except 
in the following cases:
\begin{enumerate}
\item There is exactly one irreducible component $E$ of $\tilde{S}$ over $p$, and $D$ is simply tangent to $E$.

\item There is exactly one irreducible component $E$ of $\tilde{S}$ over $p$, and $D$ has two branches meeting transversally on $E$, with each branch meeting
$E$ transversally as well.

\item There are exactly two irreducible components $E_1$ and $E_2$ of $\tilde{S}$ over $p$, and $D$ has one branch meeting each component transversally at $E_1\cap E_2$.
\end{enumerate}
\end{enumerate}
\end{lemma}

\begin{proof}
Part $(1)$ of the statement is \cite[Lemma 8.3.3]{mckernan}. We will go over its proof to get part $(2)$ of the statement.

Suppose $(S,D)$ is not log canonical. $D$ is not empty, since $S$ is klt by assumption. Let $E$ be the reduced exceptional locus of $\pi$. By Lemma \ref{flushresolution},
$\pi$ is not a log resolution for $(S,D)$. By the classification of surface singularities in Appendix \ref{surfacesings} (see in particular Lemma \ref{kltsing}, 
Lemma \ref{dltsing} and Lemma \ref{lcnotdltsing}), $E$ has simple normal crossings and $\tilde{D}$ does not have simple normal crossings with $E$ by Lemma \ref{flushresolution}.
As the computation of the coefficients is purely numerical (see the remark after Lemma \ref{lcnotdltsing}), 
we may replace $\tilde{D}$ by a disjoint union of irreducible curves, each meeting $E$ transversally and such that at least two
curves meet the same irreducible component of $E$. Let $D'$ be the pushforward of the new configuration. Clearly $\pi$ is a log resolution of $(S,D')$ by construction.

Suppose for the moment that $(S,D')$ is not log canonical at $p$. Then, by Lemma \ref{flushresolution}, there exists a component of $E$ with coefficient strictly larger than one
for $(S,D')$, contradiction. Therefore assume that $(S,D')$ is log canonical at $p$. 
By Lemma \ref{dltsing}, $(S,D')$ is not dlt. Also, by Lemma \ref{lcnotdltsing}, $E$ has at most two components over $p$. The only ways this can happen are listed in part $(2)$ of the statement.
\end{proof}

In order to capture the cases in Lemma \ref{minimalresolution} we introduce the following definition.

\begin{definition}
We say a pair $(S,D)$ is almost log canonical if the coefficient of every exceptional divisor of the minimal resolution with respect to $(S,D)$ is less or equal than one. Equivalently,
either $(S,D)$ is log canonical at singular points $p$ or there is an \'{e}tale neighborhood of $p$ such that one of the cases $(a)$-$(c)$ in the previous lemma holds. 
\end{definition}

The next results describe the geometric consequences of being flush and level. 

\begin{lemma}
Let $S$ be the germ of a klt surface at $p$ in the \'{e}tale topology.
Let $f:T\rightarrow S$ extract the irreducible divisor $E$. Assume that $e(E; S,\Delta)\geqslant 0$. Let $\Gamma$ be the log pullback of $\Delta$ and assume $\Gamma$ is
also a boundary.
Let $\pi:\tilde{S}\rightarrow S$ be the minimal resolution.

\begin{enumerate}
\item If $p$ is singular, and $(S,\Delta)$ is flush at every $\pi$-exceptional divisor $F$ then $(S,D)$ is dlt and $(S,\Delta)$ is flush.

\item If $(S,\Delta)$ is dlt and $(S,\Delta)$ is level at every $\pi$-exceptional divisor $F$, then $(S,s\Delta)$ is flush for every $s>1$ such that $s\Delta$ is a boundary.

\item Suppose $p$ is singular. If $f$ is a $K_T$-negative contraction and $(T,\Gamma)$ is flush at every exceptional divisor of the minimal resolution of $T$, then $(S,\Delta)$ is flush.

\item Suppose $p$ is singular and $(S,D)$ is log canonical. Then for any exceptional divisor $V$ there is some $\pi$-exceptional divisor $F$ with 
$e(F; S, \Delta)\geqslant e(V; S,\Delta)$.
\end{enumerate}

\end{lemma}

\begin{proof}
This is \cite[Lemma 8.3.5]{mckernan}.
\end{proof}

We will also need a slightly different version of the previous lemma.

\begin{lemma}\label{flushdescent2}
Let $S$ be the germ of a klt surface at $p$ in the \'{e}tale topology.
Let $f:T\rightarrow S$ extract the irreducible divisor $E$. Assume that $e(E; S,\Delta)\geqslant 0$. Let $\Gamma$ be the log pullback of $\Delta$ and assume $\Gamma$ is 
also a boundary.
Let $\pi:\tilde{S}\rightarrow S$ be the minimal resolution.

\begin{enumerate}
\item If $p$ is a singular point, and $(S,\Delta)$ is level at every $\pi$-exceptional divisor $F$ then $(S,D)$ is almost log canonical at $p$. If furthermore $(S,D)$ is log canonical
then $(S,\Delta)$ is level.

\item Suppose $(S,D+D')$ is log canonical. If there exists $0<c\leqslant 1$ such that $(S,D+cD')$ is level at every $\pi$-exceptional divisor $F$, then $(S,D+sD')$ is level for
any $c\leqslant s\leqslant 1$.

\item Suppose $p=f(E)$ is singular. If $f$ is a $K_T$-negative contraction and $(T,\Gamma)$ is level at every exceptional divisor of the minimal resolution of $T$, then 
$(S, D)$ is almost log canonical at $p$. 
\end{enumerate}

\end{lemma}

\begin{proof}
As usual, we go through the proof of \cite{mckernan} and make the necessary changes. We start with $(1)$. Let $\lambda$ be the smallest coefficient of $\Delta$. 
Let $F$ be a $\pi$-exceptional irreducible divisor and consider the function $f(t)=e(F; S, tD)$. Now $f(0)\geqslant 0$ since $F$ is $\pi$-exceptional. 
We have that $f(\lambda)=e(F; S, \lambda D)\leqslant e(F; S, \Delta)$ since $\lambda D\leqslant \Delta$.
Also, since $(S,\Delta)$ is level at $F$, we have that $e(F; S,\Delta)\leqslant \lambda$. Putting all this together we have
\[
  f(\lambda)\leqslant e(F; S, \Delta)\leqslant \lambda
\]

As $f$ is an affine function, $f(t)\leqslant t$ for all $t\geqslant \lambda$. In particular $e(F; S, D)\leqslant 1$, so that $(S,D)$ is almost log canonical. Suppose now that $(X,D)$ is log canonical.
Then $\pi$ is \'etale locally a log resolution of $(X,D)$, so $(X,\Delta)$ is level by Lemma \ref{flushresolution}. This proves $(1)$.

Now we prove $(2)$. Since $(S, D+D')$ is log canonical, by Lemma \ref{flushresolution} it suffices to check that $(S,D+sD')$ is level at every $\pi$-exceptional divisor $F$.
Consider $f(t)=e(F; S, D+tD')$. Clearly $f(1)\leqslant 1$ since $(S,D+D')$ is log canonical. Also, $f(c)\leqslant c$ since $(S,D+cD')$ is level at $F$. Hence $f(t)\leqslant t$ for
every $c\leqslant t\leqslant 1$, giving the result. 

Finally we prove $(3)$. 
The pair $(S,\Delta)$ is level at each exceptional divisor of the minimal resolution of $S$ above $p$,
since every such divisor also appears in the minimal resolution of $T$. Therefore $(S,D)$ is almost log canonical at $p$ by $(1)$. 
\end{proof}

The next lemma shows that being flush or level gives strong control on the singularities of the boundary divisor at smooth points of $S$. 

\begin{lemma}\label{flushmult}
Suppose $p$ is smooth and the pair $(S,\Delta)$ is flush. Let $\Delta = \sum_i a_i D_i$ and $m=m(\Delta)>0$. Then
\begin{enumerate}
\item If $M_i$ is the multiplicity of $D_i$ at $p$, then $\sum_i a_i M_i - 1<m$. In particular $m<1/(M-1)$ where $M$ is the multiplicity of $D$.

\item If $D$ has a node of genus at least two at $p$, and the coefficient of the two branches of $D$ at $p$ are $a\geqslant b$, then $2a+b<2$.

\item If $p$ is a cusp of $D$ and $a$ is the coefficient of the branch of $D$ at $p$, then $a<4/5$.

\item If $m\geqslant 4/5$ then $D$ has normal crossings.
\end{enumerate}

If $(S,\Delta)$ is level at $p$ instead then $(1)-(4)$ hold by switching \say{$<$} and \say{$\leqslant$}.
\end{lemma}
\begin{proof}
This is \cite[Lemma 8.3.7]{mckernan}.
\end{proof}

\subsection{The first two hunt steps in the flush case}\label{flushhunt}

In this subsection we describe the setting and the outcome of the hunt in the flush case. This is the main type of hunt we run in this paper, and it is especially useful in classifying
log del Pezzo surfaces without tigers (see below for the definition). The underlying idea is that this hunt preserves flushness, which in turn controls the geometry involved. Let's start
with the following fundamental definition.

\begin{definition}
Let $(X,\Delta)$ be a log pair. Let $f:Y\rightarrow X$ be a birational morphism. We say $X$ has a tiger in $Y$ if there exists an effective $\mathbb{Q}$-Cartier divisor $\alpha$ such that
\begin{enumerate}
\item $K_X+\Delta+\alpha$ is numerically trivial.
\item If $\Gamma$ is the log pullback of $\Delta+\alpha$ in $Y$, then there is a divisor $E$ of coefficient at least one in $\Gamma$.
\end{enumerate}
Any such divisor $E$ is called tiger.
\end{definition}

The following geometric situations will be common, and we name them according to \cite{mckernan}.

\begin{definition}
Let $A$ and $B$ be two rational curves on a klt surface $S$ such that $K_S + A + B$ is divisorially log terminal at any singular point of $S$.
We say that $(S,A+B)$ is a
\begin{enumerate}
\item banana, if $A$ and $B$ meet in exactly two points, and there transversally.
\item fence, if $A$ and $B$ meet at exactly one point, and there transversally.
\item tacnode, if $A$ and $B$ meet at most at two points, there is one point $q\in A\cap B$ such that $A+B$ has a node
of genus $g\geqslant 2$ at $q$, and if there is a second point of intersection, then $A$ and $B$ meet there transversally.
\end{enumerate}
\end{definition}

\noindent Now we describe the scaling of the hunt in the flush case.

\begin{lemma}\label{hunt_transformation}

Let $(S,\Delta)$ be a log pair such that $S$ is a rank one log del Pezzo surface. Let $f: T \rightarrow S$
be an extraction of relative Picard number one of an irreducible divisor $E$ of the minimal resolution. The cone of curves of $T$ has 
two edges, one generated by the class of $E$, and let the other be generated by the class of $R$. Let $x=f(E)$, $\Gamma$ such that
$K_T + \Gamma = f^* (K_S + \Delta)$ and $\Gamma_{\epsilon} = \Gamma + \epsilon E$, where $0<\epsilon \ll 1$. Assume that
$-(K_S + \Delta)$ is ample, then

\begin{enumerate}
\item R is $K_T$-negative and contractible, hence there is a rational curve $\Sigma$ that generates the same ray. 
Let $\pi$ be the contraction morphism.
\item $K_T + \Gamma_\epsilon$ is anti-ample.
\item $\Gamma_\epsilon$ is $E$ negative.
\item There is a unique rational number $\lambda$ such that with $\Gamma' = \lambda \Gamma_\epsilon$, $K_T + \Gamma'$ is
R trivial. $\lambda > 1$.
\item $K_T + \Gamma'$ is $E$ negative.
\item $\pi$ is either birational or a $\mathbb{P}^1$-fibration (called \say{net}).
\item If $\pi : T\rightarrow S_1$ is birational, and $\Delta_1 = \pi (\Gamma ')$, then $K_{S_1} + \Delta_1$ is anti-ample
and $S_1$ is a rank one log del Pezzo.
\item If $(S,\Delta)$ does not have a tiger in a surface $Y$ that dominates $T$, then neither do $(T,\Gamma')$ and $(S_1,\Delta_1)$.
\end{enumerate}

\end{lemma}

\begin{proof}
See \cite[Definition-Lemma 8.2.5]{mckernan}.
\end{proof}

\begin{definition}\label{huntchoice}
We call the above transformations $(f,\pi)$ a hunt step (in the flush case) for $(S,\Delta)$ if $e(E; S , \Delta)$ is maximal among exceptional divisors of the minimal
resolution of $S$. There might be multiple choices. If $x$ is a chain singularity we allow any choice $E$ that is not a $(-2)$ curve (this is always
possible unless $\Delta=0$ and $e(E;S)=0$). If $x$ is a non chain singularity we require $E$ to be the central curve (which has maximal coefficient by Lemma \ref{nonchainsing}). 
If two points have the same coefficient, we can pick one at our choice, but unless stated otherwise our choice will be a chain singularity that
allows us to extract the curve with lowest self-intersection.
\end{definition}

We fix now the notation that we will always use when running the hunt in the flush case.

\begin{notation}\label{notation}
We always start from a surface without boundary, so $\Delta_0 = \emptyset$. We index by $(f_i, \pi_{i+1})$ the next
hunt step for $(S_i, \Delta_i)$. Define:
\begin{itemize}
\item $x_i = f_i (E_{i+1}) \in S_i$ 
\item $q_{i+1} = \pi_{i+1} (\Sigma_{i+1})\in S_{i+1}$. 
\item $\Gamma_{i+1}$ to be the log pullback of $\Delta_i$:  $K_{T_{i+1}} + \Gamma_{i+1} = f_i ^* (K_{S_i} + \Delta_i)$. 
\item $\Delta_{i+1} = \pi_{i+1} (\Gamma_{i+1} ')$; it satisfies $K_{T_{i+1}} + \Gamma_{i+1}' = \pi_{i+1}^* (K_{S_{i+1}} + \Delta_{i+1})$. 
\item $A_1=\pi_1 (E_1)\subset S_1$ and $B_2=\pi_2 (E_2)\subset S_2$.
\item $A_2$ the strict transform of $A_1$ on $S_2$.
\end{itemize}
Let $a_1$, $b_2$ be the coefficients of $A_1$, $B_2$ in $\Delta_1$, $\Delta_2$ (which
are also the coefficients of $E_1$, $E_2$ in $\Gamma_1 '$, $\Gamma_2 '$) and $a_2$ the coefficient of $A_2$ in $\Delta_2$.
We remark that $a_1$, $b_2<a_2$ by the flush condition and the previous scaling. Let $e_i$ be the coefficient of $E_{i+1}$ in $(S_{i+1}, \Gamma_{i+1})$.
This is also the coefficient of the pair $(S_i, \Delta_i)$. Finally, we indicate by $\overline{\Sigma}_i$ the image of $\Sigma_i$ in $S_0$ or 
$S_1$, depending on the context.
\end{notation}

We are ready to describe the first two hunt steps in the flush case.

\begin{proposition}\label{hunt}
Suppose that $S$ is a rank one log del Pezzo surface that has no tigers in $\tilde{S}$. For the first hunt step: 
$K_{T_1} + E_1$ is divisorially log terminal, $K_{T_1} + \Gamma_1'$ is flush and one of the following holds.

\begin{enumerate}
\item $T_1$ is a net.
\end{enumerate}

\noindent Otherwise $K_{S_1} + a_1 A_1$ is flush and one of the following holds

\begin{enumerate}[resume]
\item $g(A_1)>1$.
\item $g(A_1) = 1$ and $A_1$ has an ordinary node at $q=q_1$.
\item $g(A_1) = 1$ and $A_1$ has ordinary cusp at $q=q_1$.
\item $g(A_1) = 0$ and $K_{S_1} + A_1$ is divisorially log terminal.

\end{enumerate}

\noindent For the second hunt step one of the following holds.

\begin{enumerate}[resume]
\item $T_2$ is a net.
\item $A_1$ is contracted by $\pi_2$, $K_{T_2}+\Gamma_2'$ is flush, $K_{S_1}+A_1$ is divisorially log terminal, $q_2$ is a smooth point of $S_2$,
$B_2$ is singular at $q_2$ with a unibranch singularity, and $K_{S_2} + \Delta_2$ is flush away from $q_2$, but is not level at $q_2$.
$\Sigma_2$ is the only exceptional divisor at which $K_{S_2}+\Delta_2$ fails to be flush.
\item $\Delta_2$ has two components.

\end{enumerate}

Suppose that in this last case that $a_2 + b_2 \geqslant 1$. Then $\overline{\Sigma}_2 \cap \operatorname{Sing}(A_1)=\emptyset$,
$K_{T_2} + \Gamma_2'$ is flush away from $\operatorname{Sing}(A_1)$. $K_{S_2}+\Delta_2$ is flush away from $\pi_2 (\operatorname{Sing}(A_1))$, and at least one of 
$-(K_{S_2}+A_2)$ or $-(K_{S_2} + B_2)$ is ample. Also, one of the following holds:

\begin{enumerate}[resume]
\item $(S_2, A_2+B_2)$ is a fence.
\item $(S_2, A_2+B_2)$ is a banana, $K_{S_2} + B_2$ is plt, and $x_1\in A$.
\item $(S_2, A_2+B_2)$ is a tacnode, with tacnode at $q_2$. $K_{S_2}+B_2$ is plt. If $x_1\in A_1$, 
$A_2\cap B_2 = \{ x_1, q_2 \}$. If $x_1\notin A_1$ then $A_2\cap B_2= \{ q_2\}$.

\end{enumerate}

\end{proposition}

\begin{proof}
See \cite[Proposition 8.4.7]{mckernan}.
\end{proof}

\begin{remark}
In Proposition \ref{hunt} we only need the hypothesis that $S$ has no tigers in $\tilde{S}$ to ensure that the $\Delta_i$ are always boundaries. Therefore the proposition
still holds under any assumptions that guarantees the same. 
\end{remark}

\subsection{The first two hunt steps in the level case}\label{levelhunt}

In classifying log del Pezzo surfaces with tigers we will use a slightly different version of the hunt. In particular, we change the scaling convention so that the $\Delta_i$
remain boundary divisors.

Let $(S,C)$ be a pair such that $S$ is a rank one log del Pezzo surface and $C$ is a reduced curve in $S$. Suppose that $K_S+C$ is anti-nef and that
$(S,C)$ is log canonical.
In particular, $C$ is a tiger of $S$ and $(S,C)$ is level. Let $T_1\rightarrow S$ be the extraction of a divisor of maximal coefficient for $(S,C)$. This divisor may be chosen
to lie in 
the minimal resolution by Lemma \ref{dltsing}, Lemma \ref{lcnotdltsing} and Lemma \ref{flushresolution}. We write $K_{T_1}+(f_0)_* ^{-1} C+e_1E_1=f_0 ^*(K_S+C)$. 

Lemma \ref{hunt_transformation} $(1)$ gives us a $K_T$-negative contraction $\pi_1$. Assume that $\pi_1:T_1 \rightarrow S_1$ is a birational morphism and let $A_1 = \pi_1 (E_1)$.
Assume also that $\pi_1$ does not contract $(f_0)_* ^{-1} C$ and call $C_1$ its image in $S_1$. 
We define $a_1\geqslant e_1$ be such that $K_{T_1}+C+a_1 E_1$ is $\pi_1$-trivial. If $a_1 > 1$, then $K_{S_1}+C_1+A_1$ is log canonical, since 
$K_{T_1}+(f_0)^{-1}C+E_1$ is log canonical
and $\Sigma_1$ has negative coefficient. In particular, $K_{S_1}+C_1+A_1$ is anti-ample and level. In this case we re-define $a_1=1$ and go to the second hunt step. 
If instead $a_1\leqslant 1$, then $K_{S_1}+C_1+a_1A_1$ is anti-nef and $K_{T_1}+(f_0)_* ^{-1} C+E_1=\pi_1 ^*(K_{S_1}+C_1+a_1A_1)$. 
The pair $(S_1, C_1+a_1A_1)$ is level by Lemma \ref{flushdescent}. In particular, $(S_1, C_1+A_1)$ is almost log canonical at singular points by Lemma \ref{flushdescent2}.
This concludes the first hunt step in the level case. For the second hunt step, let
$f_1:T_2\rightarrow S_1$ be the extraction of the exceptional divisor of maximal coefficient $E_2$ in $\tilde{S_1}$ relative to the pair $(S_1, C_1+a_1A_1)$.
We write 
\[
K_{T_2}+(f_1)_* ^{-1} C_1+a_1A_1+e_2E_2=(f_1) ^*(K_{S_1}+C_1+a_1A_1)
\]

Suppose that $\pi_2:T_2\rightarrow S_1$ is a birational morphism and that neither $C_1$ nor $A_1$ get contracted.
Call $C_2$ and $A_2$ the strict transforms of $C$ and $A_1$ in $S_2$. Let $B_2=\pi_2 (E_2)$ and $a_2=a_1$. Define again $b_2\geqslant e_2$ so that 
$K_{T_2}+(f_1)_* ^{-1} (C_1+ a_1 A_1)+ b_2 E_2$ is $\pi_2$-trivial.
As above, if $b_2>1$ then $K_{S_2}+C_2+a_2 A_2+B_2$ is log canonical and anti-ample.
We may therefore assume that $b_2 \leqslant 1$. Then $K_{S_2}+C_2+a_2 A_2+b_2 B_2$ is anti-nef, and level
by Lemma \ref{flushdescent}. It follows that $(S_2, C_2 + A_2 + B_2)$ is almost log canonical at singular points by Lemma \ref{flushdescent2}.
At smooth points, we can control singularities by Lemma \ref{flushmult}. 

\subsection{Classification of the hunt contractions in the flush case}

In this subsection we describe the geometry of the contractions that appear during a hunt in the flush case. The description we give is \'{e}tale local. First, we explain our setting, following
\cite[Chapter 11]{mckernan}.
Let $T$ be a klt projective surface and let $\pi:T\rightarrow S$ be a proper birational contraction of a $K_T$-negative extremal ray. Denote by $\Sigma$ the exceptional divisor
and let $q$ be the image of $\Sigma$. We will be concerned with the \'{e}tale local description of $T$ around $\Sigma$. 
Let $W\subset T$ be a curve with smooth components crossing transversally. Assume $W$ has at most two irreducible components $X$, $Y$ and that $K_T + cX + dY$ is $\pi$-trivial and flush, with $0<c,d<1$. Assume also that $\pi|_W$ is finite (which is to say that $\Sigma$ is not a component of $W$). 
Thus $K_S + cX + dY$ is flush by Lemma \ref{flushdescent}. 
Let $D$ be the image of $W$ in $S$ and let $h:\tilde{T} \rightarrow \tilde{S}$ be the induced map between the minimal resolutions.

Before stating the classification, we need to name some particular geometric configurations.

\begin{definition}
Let $S$ be a surface, and $D$ be a curve in $S$. Choose a point $q\in D$. If $D$ has two smooth branches $X$ and $Y$ meeting to order $g$ at $q$ we say that $D$ has a node of order $g$.
If $g\geqslant 1$ we have to blow $g-1$ times to make $X+Y$ have normal crossings. We call this configuration 0. If we further blow up once at $X\cap Y$ we reach configuration $I$.
Now blow up at $X\cap \Sigma$, where $\Sigma$ is the unique $(-1)$ curve. This is configuration $II$. 
We denote by $(II)$ and a string of $x$ and $y$ (for example $(II,x,x,y)$) a configuration that is reached by starting with configuration $II$ and then blowing up the intersection of
the unique $(-1)$ curve with the branch corresponding to the letter of the string. 
\end{definition}

\begin{definition}
Let $S$ be a surface, and $D$ be a curve in $S$. Choose a point $q\in D$. If $D$ has a unibranch double point at $q$ we say $D$ has a cusp of order $g$. We call $X$ the branch of $D$
at $q$. We need $g$ blowups along $X$ to remove its singularity. We call this configuration $I$. Next, we blow up at $X\cap \Sigma$, where as usual $\Sigma$ is the unique $(-1)$ curve.
This is configuration $II$. Next, we blow up again at $X\cap \Sigma$, reaching configuration $III$. We have now three choices:
\begin{enumerate}
\item If we blow up at the intersection with the unique exceptional $(-3)$ curve, we reach configuration $U$.
\item If we blow up at $X\cap \Sigma$ then we reach configuration $V$.
\item If we blow up at the intersection with the unique $(-2)$ curve, we reach configuration $W$.
\end{enumerate}
After this, at each step we may only choose to blow up one of two points on the last $(-1)$ curve: either the nearest point to $X$, in which case we add
an $n$ to the string of letters we have formed so far, or the farthest point from $X$, in which case we add an $f$.
\end{definition}

\begin{notation}
If $p$ is a chain singularity, we mark with a prime the component touched by $\Sigma$ and we underline the component touched by $D$.
\end{notation}

Going back to our original setting, if $D$ has multiplicity two at $q$, then $q$ is smooth and we have the following classification.

\begin{lemma}\label{multiplicity_two_node}
Suppose $D$ has a node of order $g$. One of the following holds:

\begin{enumerate}
\item $T$ has type $I$ or $0$, and $c+d=1$.
\item $g\geqslant 2$, $T$ has type $II$ and $gd+(g+1)c = g+1$.
\end{enumerate}

$g=1$, $K_T+\Sigma+X+Y$ is log canonical and either

\begin{enumerate}[resume]
\item $T$ has type $(II, x^{r-1})$ with $r\geqslant 1$, there is a unique singularity, an $A_r$ point, $X$ meets $\Sigma$ at a smooth point,
$Y$ meets $\Sigma$ at the $A_r$ point, which has type $(\underline{2}, ..., 2')$ and $c+\frac{d}{r+1} = 1$.

\item $\Sigma$ meets $X$ and $Y$ each at a singular point of $T$, and those are the only singularities along $\Sigma$.
\end{enumerate}

\end{lemma}

\begin{proof}
See \cite[Lemma 11.1.1]{mckernan}.
\end{proof}

\begin{lemma}\label{multiplicity_two_cusp}
Suppose $D$ has a cusp of order $g$. Then either
\begin{enumerate}
\item $T$ has type $I$ and $c=1/2$.
\item $T$ has type $II$ and $c=(g+1)/(2g+1)$.
\item $T$ has type $III$ and $c=(g+1)/(2g+1)$.
\item $g=1$, $T$ has type $u$ and $c=3/4$ or $g=2$ and $c=9/14$.
\item $g=1$, $T$ has type $v$ and $c=5/7$ or $g=2$ and $c=7/11$.
\item $g=1$, $T$ has type $w$ and $c=7/9$.
\item $g=1$, $T$ has type $(u;n)$ and $c=11/14$.
\item $g=1$, $T$ has type $(v;f)$ and $c=10/13$.
\item $g=1$, $T$ has type $(v;f^2)$ and $c=15/19$.
\item $g=1$, $T$ has type $(v;n)$ and $c=3/4$.
\item $g=1$, $T$ has type $(v;n^2)$ and $c=7/9$.
\end{enumerate}
\end{lemma}

\begin{proof}
See \cite[Lemma 11.2.1]{mckernan}.
\end{proof}

If $D$ has multiplicity three, then we only have the following two cases.

\begin{lemma}\label{multiplicity_three_a}
Suppose that $D$ has multiplicity three. If $D$ has two branches at $q$, then one branch has a double point, a simple cusp, 
and the other is smooth. If $X$ is the branch with the cusp, then $\Sigma$ meets $X$ transversally at one smooth point, $\Sigma$ 
contains two singularities $(2)$ and $(3)$, and $Y$ meets the $(-2)$-curve, and is disjoint from $\Sigma$ on $\tilde{T}$.
\end{lemma}
\begin{proof}
See \cite[Lemma 11.3.2]{mckernan}.
\end{proof}

\begin{lemma}\label{multiplicity_three_b}
Suppose that $D$ has multiplicity three. If $D$ is unibranch, then either
\begin{enumerate}
\item $\Sigma$ has singularities $(3,2')$, $(3)$ and meets $X$ transversally at a smooth point
\item $\Sigma$ has singularities $(3)$, $(2)$ and on the minimal resolution $X$ meets $\Sigma$ transversally at the intersection of
$\Sigma$ and the $(-2)$-curve.
\end{enumerate}
\end{lemma}
\begin{proof}
See \cite[Lemma 11.3.3]{mckernan}.
\end{proof}

The last type of contraction we are interested in are fibrations  (which we call \say{nets}, following \cite{mckernan}).
Let $\pi:T \rightarrow C$ be a $\mathbb{P}^1$-fibration of relative Picard number one, with $T$ a normal surface. Let $\tilde{\pi}$ 
be the composition $\tilde{\pi}:\tilde{T}\rightarrow T \rightarrow C$, where $\tilde{T}$ is the minimal resolution of $T$. 
We describe the fiber $F$ of $\tilde{T}$ above $p\in C$ as the sequence $k(-a)+l(-b)+\cdots +m(-c)$, by which we mean a chain 
of curves of self intersection $-a$, $-b$ and so on, with multiplicities $k, l$ and so on.

We recall here the following definition from Appendix \ref{surfacesings}.

\begin{definition}
A non Du Val klt singularity with coefficient strictly less than $1/2$ is called almost Du Val. They all are of the form $(3,A_k)$ for some $k$.
\end{definition}

\begin{lemma}\label{net_fiber_classification}
Assume $T$ is klt, $G$ is a multiple fiber of $\pi$ of multiplicity $m$, and $G$ contains a cyclic singularity, either Du Val or almost Du Val. 
If $e(T)<2/3$ then $G$ is one of the following:

If $(T,G)$ is not dlt at any singular point:
\begin{enumerate}
\item $(2,2',2), m=2$.
\item $(3,2',2,2), m=3$.
\end{enumerate}

If $(T,G)$ is dlt at one singular point, but not dlt:

\begin{enumerate}[resume]
\item $(2', z), m=4$. $z$ is a non chain singularity, with center -2 and branches $(2)$, $(2)$ and $(2, \cdots, 3')$ (or $(3')$).
\item $(2,3',2; 2'), m=4$.
\end{enumerate}

If $(T,G)$ is dlt:

\begin{enumerate}[resume]
\item $(A_k; (k+1)'), k\leqslant 4, m=k+1$. The fiber is $-(k+1)+[k+1](-1)+k(-2)+[k-1](-2)+\cdots +(-2)$.
\item $(2,3'; 2', 3), m=5$. The fiber is $(-2)+2(-3)+5(-1)+3(-2)+(-3)$.
\item $(3,2,2';4',2), m=7$. The fiber is $(-3)+3(-2)+5(-2)+7(-1)+2(-4)+(-2)$.
\item $(4,2'; 3', 2, 2), m=7$. The fiber is $(-4)+4(-2)+3(-1)+7(-3)+2(-2)+(-2)$.
\end{enumerate}
\end{lemma}

\begin{proof}
See \cite[Lemma 11.5.9]{mckernan}.
\end{proof}

\begin{lemma}\label{adhoc}
If $G$ is a multiple fiber of multiplicity three and the coefficient $e(T)<2/3$, then $G$ is one of the fibers of the above classification.
\end{lemma}
\begin{proof}
See \cite[Lemma 11.5.13]{mckernan}.
\end{proof}

\subsection{Useful facts}

Here we collect some results that are somehow unrelated to the previous discussion, but that will be useful later on.

\begin{lemma}\label{sum_ten}
Let $S$ be a rank one log del Pezzo surface and $\tilde{S}$ its minimal resolution. Then $K_{\tilde{S}}^2 + \rho(\tilde{S}) = 10$.
\end{lemma}
\begin{proof}
Run the minimal model program on $\tilde{S}$. The end result $S_{min}$ is a Mori fiber space because $\tilde{S}$ is birational to $S$. 
Then $S_{min}$ is either $\mathbb{P}^2$ or a ruled surface. In this last cast it's a Hirzebruch surface because it's rational. 
In any case, it follows that $K_{S_{min}} ^2 + \rho(S_{min})=10$ and a sequence of smooth blow ups does not change the equality.
\end{proof}

\begin{lemma}\label{global_canonical}
Let $S$ be a rank one log del Pezzo surface. Let $n$ be the number of exceptional components in the minimal resolution. Then
\[
K_S ^2 = 9 - n + \sum_i e_i (-2-E_i ^2)
\]

where $e_i$ is the coefficient of the divisor $E_i$.
If $e(S)<1/2$, $u$ is the number of exceptional components coming from Du Val singualrities, and $n_r$ is the number
of points of type $(3,A_r)$, then the above formula takes the form
\[
 K_S ^2 = 9 - u - \sum_r n_r (r+1)\Big( 1- \frac{1}{2r+3}\Big)
\]
\end{lemma}
\begin{proof}
Obvious.
\end{proof}

\begin{lemma}\label{high_picard}
Let $S$ be a rank one log del Pezzo surface and $\tilde{S}$ its minimal resolution. Then $\rho(\tilde{S})\geqslant 11$ if $S$ has no tigers in $\tilde{S}$.
\end{lemma}
\begin{proof}
This follows from \cite[(10.3)]{mckernan}.
\end{proof}

\begin{lemma}\label{no_more_singularities}
Let $p$ be a klt singularity with $n$ components and coefficient $e<3/5$. Then $n\geqslant \sum e_i(-2-E_i^2)$
and equality holds if and only if $p$ is of type $(4)$.
\end{lemma}
\begin{proof}
Obvious by Proposition \ref{singlist}.
\end{proof}

\begin{lemma}\label{smoothp1}
Suppose $C$ is a smooth rational curve in the smooth locus of a rank one log del Pezzo surface $S$. Then either $S=\mathbb{P}^2$
or $S=\overline{\mathbb{F}}_n$. If $S=\overline{\mathbb{F}}_n$ then $C\in |\sigma_n |$, where $\sigma_n$ is the image of a section of self intersection $n$.
\end{lemma}
\begin{proof}
See \cite[Lemma 13.7]{mckernan}.
\end{proof}

\section{Log del Pezzo surfaces without tigers}\label{sectionnotigers}

We will assume throughout this section that $S_0$ is a rank one log del Pezzo surface with no tigers in $\tilde{S}_0$. Our goal is to explicitly classify all such surfaces. To this end we will 
apply the hunt in the flush case (see Subsection \ref{flushhunt}) 
and carefully analyze each step (all notation is fixed as in (\ref{notation})). We start by studying the case in which $T_1$ is a net, in analogy with \cite[Chapter 14]{mckernan}.

\subsection{$T_1$ a net}\label{section_net}

Here we assume that $T_1$ is a net. If $\operatorname{char}(k)\neq 2,3$, then we show that there is only one possibility for $S_0$ (see Proposition \ref{no_net}).

Recall that $K_{T_1}+a_1 E_1$ is flush, klt and anti-nef by Proposition \ref{hunt}.
We start with the following easy observation.

\begin{lemma}\label{lemma_net}
The $f_0$-exceptional divisor $E_1$ is not contained in a fiber of $\pi_1$. Let $F$ be the general fiber of $\pi_1$ and consider $d=F\cdot E_1$. 
Then $d\geqslant 3$ and $e_0<a_1=2/d$.
\end{lemma}
\begin{proof}
The proof of \cite[Lemma 14.2]{mckernan} applies independently of the characteristic, and we briefly recall it here. Since $E_1$ is $f_0$-exceptional,  $E_1 ^2 <0$ and thus
$E_1$ cannot be a fiber of $\pi_1$. It follows then that the two edges of the cone of curves of $T_1$ are generated by $F$ and $E_1$. Now notice that
$E_1 ^2 <0$ implies that $(K_{T_1}+E_1)\cdot E_1 < (K_{T_1}+e_0 E_1)\cdot E_1$. The right hand side of this inequality is equal to $f_0 ^*(K_{S_0})\cdot E_1  = 0$.
Therefore $(K_{T_1}+E_1)\cdot E_1<0$.

On the other hand, $(K_{T_1}+E_1)\cdot F= -2+d$. The divisor $-(K_{T_1}+E_1)$ is not nef since $S_0$ has no tigers in $\tilde{S}_0$, and therefore $-2+d>0$. Since $d$ is clearly an integer,
we get that $d\geqslant 3$. Finally, in order to find $a_1$ we simply solve $(K_{T_1}+a_1 E_1)\cdot F=0$.
\end{proof}

\begin{lemma}\label{lemma_net_2}
If $\operatorname{char}(k)\neq 2,3$, then $e_0 \geqslant 1/2$.
\end{lemma}
\begin{proof}
Suppose that we have $e_0<1/2$ instead. By the classification of klt singularities of low coefficient (see Proposition \ref{singlist}), all
singularities are either Du Val or of the form $(3,A_r)$, for some $r\geqslant 0$. Suppose there are $n$ non Du
Val singularities. If $n\leqslant 2$ then $S_0$ has a tiger by \cite[Lemma 10.4]{mckernan}, so $n\geqslant 3$.
In particular there are at least two non Du Val points on $T_1$, $S_0$ is not Gorenstein and  $e_0\geqslant 1/3$.
Also, by Lemma \ref{lemma_net}, $d\leqslant 5$. The first hunt step extracts the $(-3)$ curve $E_1$ from a point with maximal $r$.

\textbf{Case 1:}
Suppose for the moment that $r=0$. On $T_1$ we then only have singularities of type $(3)$ or Du Val. 
By Lemma \ref{net_fiber_classification}, any fiber through a $(3)$ point has multiplicity three and contains only one other singularity, an $A_2$ point. 
Since $E_1$ is in the smooth locus of $T_1$, $m$ divides $d$ and $d=3$. 
There are at least two $(3)$ points on $T_1$, so that if we apply Riemann-Hurwitz to the morphism $\pi_1 |_{E_1}\rightarrow \mathbb{P}^1$,
we see that there are exactly two multiple fibers. This implies that the Picard number of $\tilde{T_1}$ is eight. 
But then $\tilde{S_0}$ has Picard number eight too and hence $S_0$ has a tiger by Lemma \ref{high_picard}. 

\textbf{Case 2:}
Hence $r\geqslant 1$ and on the exceptional divisor $E_1$ we only have one singularity, an $A_r$ point. 
Also, $e_0\geqslant 2/5$, which in turn implies that $d\leqslant 4$ by Lemma \ref{lemma_net}.
Note that a fiber can contain at most two singular points by Lemma \ref{net_fiber_classification}, 
so there is at least one non Du Val point which is not in the same fiber as the $A_r$ point. Call this point
$p$ and the respective fiber $F$. The fiber $F$ meets $E_1$ only at smooth points because of the way we defined $p$. This tells us that
the multiplicity of $F$ can't be more than four. The only non Du Val fibers with multiplicity at most four have multiplicity exactly three by Lemma \ref{net_fiber_classification} 
and the fact that $e_0<1/2$. Hence $d$ is a multiple of three, and so $d=3$.

In summary, $F$ has multiplicity exactly three and passes through either just one singularity of type $(3,2,2,2)$, or two singularities, an $A_2$ point and a $(3)$ point.
This means that $p$ is either of type $(3,2,2,2)$ or $(3)$, and is the only non Du Val point on $F$. By the above there is at least another non Du Val point on $T_1$, which we call $q$. Finally, let $G$ be the fiber through the $A_r$ point of $E_1$. 

\textbf{Case 2a:}
Suppose first that $q$ lies on $G$. By Lemma \ref{net_fiber_classification}, the singularities on $G$ are $A_2$ and $(3)$, G has multiplicity three and $K_{T_1}+G$ is dlt. 
This implies that the intersection of $E_1$ with $G$ at $q$ is either one or two, depending on whether they meet at opposite ends of the $A_2$ chain
or at the same end. This means that they necessarily intersect at another point, but since all other points on $E_1$ are smooth, we get that $G\cdot E_1\geqslant 4$, contradiction.

\textbf{Case 2b:}
Hence $q$ lies on a fiber distinct from $F$ and $G$, which we call $H$. Clearly $H$ has multiplicity three. Notice now that one of the points of intersection of $E_1$
and $G$ is ramified for $\pi_1 |_{E_1}$, for either they only meet at the $A_r$ point, which is then ramified, or they meet at another point, which is ramified
since it is a smooth point and $G$ is a multiple fiber. But then applying Riemann-Hurwitz we get a contradiction, since there would be at least three
ramified points, two of which ramified of order three.
\end{proof}

\begin{proposition}\label{no_net}
If $\operatorname{char}(k)\neq 2,3$ and if $T_1$ is a net, then $\tilde{S}_0$ is obtained as follows: choose a smooth rational curve 
$E_1\subseteq \mathbb{P}^1\times \mathbb{P}^1$ which is a triple section for the first projection and has exactly three ramified points $p'$, $q'$, and $r'$.
Suppose that $p'$ and $q'$ are of order two and that $r'$ is of order three. Let $F$, $G$ and $H$ be the corresponding fibers. Blow up above $p'$ three times along $E_1$.
Then blow up twice above $q'$ along $E_1$. Finally, blow up four times above $r'$ along $E_1$.
\end{proposition}
\begin{proof}
Suppose $T_1$ is a net. 
By Lemma \ref{lemma_net_2}, $e(S_0)\geqslant 1/2$. By Lemma \ref{lemma_net}, $d=3$ and $e(S_0)<2/3$.
The spectral value of each singularity along $E_1$ is at most one by Lemma \ref{smallsvalue}.
The classification of Lemma \ref{net_fiber_classification} therefore applies: either the multiple fiber meets $E_1$ at smooth points,
in which case we can use Lemma \ref{adhoc} because $m=3$, or it meets $E_1$ in singular points of spectral value less or equal to one,
which are cyclic Du Val or almost Du Val by \cite[Lemma 8.0.8]{mckernan}. There are four possible cases: $T_1$ has zero, one, two or three singularities
on $E_1$.

\textbf{Case 1:} $T_1$ is smooth along $E_1$. Every multiple fiber has multiplicity three and either contains
just one singularity, a $(3,2,2,2)$ point, or two singularities, a $(3)$ point and an $A_2$ point. In any case, the singular points on a multiple fiber contribute at most
four to the Picard number of the minimal resolution. By Riemann-Hurwitz applied to $\pi_1 |_{E_1}$, there are exactly two singular fibers, and hence
the Picard number of the minimal resolution of $S_0$ is at most ten, which contradicts Lemma \ref{high_picard}.

\textbf{Case 2:} There is just one singular point, which we call $p$, on $E_1$. Let $F$ be the the fiber through it.

\textbf{Case 2a:}
Let us first make the further assumption that the multiplicity
of $F$ is at least three. Then $p$ is ramified of order three and there can be at most one other singular fiber,
necessarily of multiplicity three. As we saw above, the singular points of this other fiber contribute to the Picard number of the minimal resolution by at most four. 
Therefore, the minimal resolution of $T_1$ must have at least five exceptional components above the singularities in $F$ by Lemma \ref{high_picard}.
By Lemma \ref{net_fiber_classification}, the list of possibilities is $(A_4; 5)$ and $m=5$ or $(3,2,2;4,2)$ and $m=7$ or
$(2;z)$ with $z$ a non chain singularity and $m=4$. The first case does not occur: since the spectral value of $p$ is at
most one, $p$ must be the $A_4$ point. Since in this case both $K_{T_1}+F$ and $K_{T_1}+E_1$ are dlt, 
an easy computation shows that the intersection $F\cdot E_1$ can't be three, which contradicts the fact that $d=3$. 
In the second case $p$ must be the $(3,2,2)$ singularity. Again $K_{T_1}+F$ is dlt and $F$
can touch either end of the singularity. A computation shows that $E_1\cdot F=3$ is only achieved when 
both the strict transforms of $E_1$ and $F$ on the minimal resolution touch the $(-3)$ curve, and do not meet on it. But then the singularity of $x_0$
has to be of the form $(k, 3, 2, 2)$. Since $E_1$ must be the $(-k)$ curve, and since its coefficient is the highest due to the definition of the hunt, we
must have $k\geqslant 5$. However this gives $e_0> 2/3$, contradiction.
In the third case $p$ is the $(2)$ point, $K_{T_1} + F$ is dlt at $p$ and $F$ can't meet $E_1$ on any other point since $m=4$. 
Under these hypotheses however  $F\cdot E_1\neq 3$, contradiction.

\textbf{Case 2b:}
We now assume that the multiplicity of $F$ is two. By Lemma \ref{net_fiber_classification}, $F$ contains either only one $A_3$ singularity or two $A_1$ singularities. 
Notice that $F$ can't meet $E_1$ on three different points, so there is at least one ramification point on it. Any other fiber has multiplicity
three and contributes by at most four to the Picard number of the minimal resolution. There must be at least two of them
by Lemma \ref{high_picard}, but then this contradicts Riemann-Hurwitz applied to $\pi_1 |_{E_1}$.

\textbf{Case 3:} Suppose there are exactly two singular points $p$ and $q$ on $E_1$. 
If $p$ and $q$ lie on the same fiber, they must be the only points of $E_1$ on the fiber since $d=3$ and the fiber is
multiple. In particular, exactly one of them is a ramification point for $\pi_1 |_{E_1}$.
Since both $p$ and $q$ are Du Val or almost Du Val, their fiber contributes by at most four to the Picard number of 
the minimal resolution by Lemma \ref{net_fiber_classification}. There are then at least two more singular fibers, each of multiplicity three, contradicting 
Riemann-Hurwitz. Hence $p$ and $q$ lie on distinct fibers $F$ and $G$ respectively. 

\textbf{Case 3a:} We further suppose there is another multiple fiber, which we call $H$, necessarily of multiplicity three. 
Then $F$ and $G$ have multiplicity two by Riemann-Hurwitz. Again by Lemma
\ref{high_picard} and Lemma \ref{net_fiber_classification}, at least one of the singularities on $E_1$, say $p$, is an $A_3$ point. 
Notice that $F$ meets transversally the middle component of $p$ and that $K_{T_1}+E_1$ is dlt. This implies that $F\cdot E_1 = 1$. Therefore
$F$ must intersect $E_1$ on another point, which is necessarily a smooth point of $T_1$ and which
is a ramification point of order two for $\pi_1 |_{E_1}$. Note that $x_0$ can't have type $(2,2,2,k,2,2,2)$ with $k\geqslant 3$ because $e(S_0)<2/3$,
hence $G$ contains two $A_1$ singularities. $H$
passes through a $(3,2,2,2)$ point by considerations on the Picard number. Notice that $G$ is
disjoint from $E_1$ on the minimal resolution above the singular point, since otherwise the $A_1$ point would be a ramification point of order three, contradicting
Riemann-Hurwitz. It follows that $G$ also meets $E_1$ on a point different from $q$, which is necessarily a smooth point of ramification of order two. Finally $E_1$ is a $(-3)$ curve
because $e_0<2/3$. Contracting all the $(-1)$ curves from the minimal resolution of $T_1$ we perform nine blowdowns on
$E_1$. This means that we reach a Hirzebruch surface $\mathbb{F}_n$, the image of $E_1$ has self intersection six and is a triple section.
Write $E_1\equiv aC+ bf$ in $\mathbb{F}_n$, where $C$ is the negative section and $f$ a general fiber. Then $a=3$ and $b=(3n+2)/2$.
By \cite[Chapter V, Proposition 2.20]{hartshorneag}, we must have $b>3n$, which gives $n=0$. Therefore $\mathbb{F}_n$ is $\mathbb{P}^1 \times \mathbb{P}^1$.
By working backwards, we find the surface described in the statement.

\textbf{Case 3b:} Now assume there are only two singular fibers. One of them contributes by at least five to the Picard number of the minimal
resolution by Lemma \ref{high_picard}. But this kind of fiber was ruled out in Case 2a.

\textbf{Case 4:} As the last case, suppose there are three singularities $p$, $q$ and $r$ along $E_1$. Since $x_0$ is a klt singularity with coefficient
less than $2/3$, the singularities on $E_1$ are respectively $(2), (2)$ and $(A_j, 3)$ for some $j$. The two $A_1$ points $p$ and $q$ can't lie on the same
fiber $F$, otherwise $K_{T_1}+F$ is dlt and there is no configuration in which $E_1\cdot F= 3$. So $F$ and $G$ contain each two $A_1$ points and $K_{T_1}+F+G$ is dlt.
By Lemma \ref{net_fiber_classification} an $A_1$ point and $(A_j, 3)$ can't lie on the
same fiber either. Hence $p$, $q$ and $r$ lie on different fibers, say $F$, $G$ and $H$. The fibers through the $A_1$ points can't contain
a non chain singularity, for otherwise $m=4$ and there is no way to get $E_1\cdot F=3$. Since the $(A_j,3)$ point
can lie only in fibers of multiplicity at least three, $\pi_1 |_{E_1}$ is ramified of order three at $r$. 
The contribution of $G$ to the Picard number of the minimal resolution of $T_1$ must be at least five, since there are no
more multiple fibers by Riemann-Hurwitz. Then $j=2$ and $G$ contains a $(4,2)$ point by Lemma \ref{net_fiber_classification}. The coefficient of $(4,2)$ is
$4/7>1/2$, which contradicts the choice of $x_0$ in the hunt.
\end{proof}
\subsection{$g(A_1)>1$}\label{section_big_genus}

In the previous subsection we saw that $T_1$ can't be a net, so $\pi_1$ is a birational contraction. Here we prove:

\begin{proposition}\label{nobiggenus}
If $g(A_1)>1$ and $\operatorname{char}(k)\neq 2$, then $S_0$ is the surface of \cite[Lemma 15.2]{mckernan}.
\end{proposition}

We will compare the local information we get by the geometry of the hunt contraction, and the global information
we get from Lemma \ref{global_canonical} to obtain numerical obstructions. First note that $S_0$ is not Gorenstein, 
for otherwise it would have a tiger by Lemma \cite[Lemma 10.4]{mckernan}.
Hence $1/3\leqslant e_0 < a_1$. Therefore $A_1$ has either a point of multiplicity three or a double point by Lemma \ref{flushmult}.

\begin{lemma}
$A_1$ cannot have multiplicity three.
\end{lemma}
\begin{proof}
Suppose that $A_1$ has multiplicity three. Then $a_1<1/2$ by Lemma \ref{flushmult} and all the singularities in $S_0$
are either Du Val or almost Du Val. The possible configurations for the contraction are given by Lemma 
\ref{multiplicity_three_a} and Lemma \ref{multiplicity_three_b}.

Assume that $\pi_1$ is given by Lemma \ref{multiplicity_three_a}. Then $x_0$ has type $(3,2)$. Say that on $S_0$ we 
have $n$ singularities of type $(3)$, $m$ of type $(3,2)$ and possibly some Du Val ones. We compute
\[
  K_{S_0}^2 = \frac{(K_{S_0}\cdot \overline{\Sigma}_1)^2}{\overline{\Sigma}_1 ^2} = \frac{1}{15\cdot 11}
\]
From the projection formula we get that $K_{S_0}^2 = K_{\tilde{S}}^2 + n/3 + 2m/5$, but this can't happen because
there is a factor eleven in the denominator of $K_{S_0}^2$ and $K_{\tilde{S}}^2$ is an integer.

Now assume that $\pi_1$ is given by the first case of Lemma \ref{multiplicity_three_b} (the second case is numerically the same
as above). Say $x_0$ has type $(3,A_r)$, with $r\geqslant 1$. We have that $K_{S_0} \cdot \overline{\Sigma}_1 = -7/15+(r+1)/(2r+3)$.
By adjunction $\overline{\Sigma}_1 ^2 = -1/15+(r+1)/(2r+3)$. Then, as above,
\[
 K_{S_0} ^2 = \frac{(r-6)^2}{15(2r+3)(13r+12)}
\]
Since $K_{S_0}\cdot \overline{\Sigma}_1 < 0$, we have that $r \leqslant 5$, hence the non Du Val singularities can only be
of type $(3,A_j)$ with $j\leqslant r\leqslant 5$. One checks that for any $r\leqslant 5$ Lemma \ref{global_canonical} cannot hold.
Suppose for example that $r=1$. Then $K_{S_0}^2 = 1/75$, hence there is a factor twenty five in the denominator, contradiction by Lemma \ref{singlist}.
All the other cases are similar, except when $r=5$. In this case $K_{S_0}^2 = 1/(3\cdot 5\cdot 7\cdot 11\cdot 13)$.
For this to be possible, there should be a singularity of type $(3, A_j)$ for $j=0,1,2,4,5$. But then, again by
Lemma \ref{global_canonical} and Lemma \ref{no_more_singularities} we get $K_{S_0}^2 < 0$, contradiction.
\end{proof}

Now we consider double points of genus $g\geqslant 2$. We analyze them via Lemma \ref{multiplicity_two_node} and Lemma \ref{multiplicity_two_cusp}.
Note that configuration $0$ of Lemma \ref{multiplicity_two_node} does not arise since $E_1$ is smooth.
Also note that configurations $I$ and $II$ are numerically the same in the node and cusp cases. In these cases we will only discuss cusps since we only use 
the numerical properties of the configuration. Let's start with configuration $I$, where we also allow $g=1$ because it simplifies some work later on.

\begin{lemma}\label{cusp_one}
Suppose $A_1$ has a double point and $g(A_1)\geqslant 1$. Then configuration $I$ does not arise.
\end{lemma}
\begin{proof}
By Lemma \ref{multiplicity_two_node} we have $a_1=1/2$, so that $1/3\leqslant e_0 < 1/2$. Therefore $x_0$ is of type $(3,A_r)$ for some $r\geqslant 0$.
The local configuration tells us that
\[
  K_{S_0}^2 = \frac{(K_{S_0}\cdot \overline{\Sigma} _1)^2}{\overline{\Sigma} _1 ^2} = \frac{(-1+2e_0)^2}{4e_0 - 1/g}
\]

Lemma \ref{global_canonical} has finitely many solutions thanks to Lemma \ref{no_more_singularities}.
By running a computer program (or by a direct and tedious computation) one sees that the only solutions for Lemma \ref{global_canonical} are $r=0$, $g=1$ and $r=1$, $g=1$.

In the case $r=0$, $g=1$, $A_1$ is in the smooth locus and is a tiger by adjunction. So assume $r=1$, $g=1$. There are five 
solutions compatible with these numbers: in the notation of Lemma \ref{global_canonical}, we have 
\begin{enumerate}
\item $u=6$, $n_0 = 2$, $n_1 = 1$.
\item $u=4$, $n_0 = 5$, $n_1 = 1$.
\item $u=2$, $n_0 = 8$, $n_1 = 1$. 
\item $u=0$, $n_0 = 11$, $n_1 = 1$.
\item $u=0$, $n_1 = 4$, $n_2 = 2$.
\end{enumerate}
Now we will apply the results
of \cite[Chapter 12]{mckernan} to our situation. By \cite[Lemma 12.1]{mckernan} we have that $F=\emptyset$ and 
$M\in |K_{S_1}+A_1|$ is an irreducible rational curve that passes through the only singular point on $A_1$. $M$ also
passes through all the non Du Val points of $S_1$. 
Let $\overline{M}$ be the reduction of $M$.
Since $(K_{S_1}+M)\cdot M<0$, we have by adjunction that $K_{S_1}+\overline{M}$ is dlt and there are exactly two $(3)$ points on $M$.
Hence we are in the case $(1)$. Now consider the transformation $S_1 \dashrightarrow W$ of \cite[Lemma-Definition 12.4]{mckernan}. $W$ is
a Gorenstein log del Pezzo of Picard number one. Let $C$ be the exceptional $(-2)$ curve over the singular point contained in $A_1$. Then, after contracting $M$,
$C$ becomes a $(-1)$ curve passing through the intersection of the exceptional curves of an $A_2$ point. Notice that $K_W^2 = 1$ by Lemma \ref{sum_ten}
since $u=6$ and there is a $A_2$ point. By the discussion in Appendix \ref{reduction}, after blowing up the unique base point $p$ of $|-K_{\tilde{W}}|$ we get an extremal rational elliptic surface.
Notice that we must have $p=A_1\cap C$ since $A_1\equiv C\equiv -K_W$ and the Picard group of $\tilde{W}$ has no torsion. By blowing up $p$, we get an elliptic fibration
which has $I_1$ and $IV$ among its singular fibers, but this does not appear in the classification of Theorem \ref{elliptic_classification}, contradiction.
\end{proof}

\begin{lemma}\label{config2}
Suppose $A_1$ has a double point and $g(A_1)\geqslant 2$. If $\operatorname{char}(k)\neq 2$ then configuration $II$ does not arise, unless $S_0$ is the surface
described in \cite[Lemma 15.2]{mckernan}.
\end{lemma}
\begin{proof}
By Lemma \ref{multiplicity_two_cusp}, we get that $e_0<a_1=(g+1)/(2g+1)\leqslant 3/5$. By Lemma \ref{singlist}, 
either $x_0 = (2,3,2,2)$ and $g=2$, or $x_0=(3,A_g)$.

\textbf{Case 1:}
Suppose that $x_0 = (3,A_g)$ and hence $e_0 < 1/2$. The local configuration tells us that
\[
  K_{S_0}^2 = \frac{(K_{S_0}\cdot \overline{\Sigma} _1)^2}{\overline{\Sigma} _1 ^2} = \frac{(-1+\frac{2g+1}{2g+3})^2}{\frac{6g+1}{2g+3}-1}
\]

By running a computer program we see that the only solutions to Lemma \ref{global_canonical} are 
$g=2$, $u=3$, $n_0=5$, $n_2=1$ and $g=2$, $u=5$, $n_0=2$, $n_2=1$. In any case, we have $g=2$ and $u>0$.
Again, since $a_1=3/5>1/2$, let $F$ and $M$ be as in \cite[Chapter 12]{mckernan}.
$A_1$ is in the smooth locus of $S_1$ by the description of configuration $II$ and by the assumption that $x_0 = (3, A_g)$. Also, $A_1 ^2 = -3 + 9=6$.
It follows by \cite[Lemma 12.1.7]{mckernan} that $M=\emptyset$. Since $(K_{S_1}+F)\cdot F<0$ and since $F$ passes through all the non Du Val points, we must be in the
case $u=5$, $n_0=2$. After extracting the two $(-3)$ curves $E_2$ and $E_3$, $|F|$ is base point free and gives a fibration $T\rightarrow \mathbb{P}^1$ by \cite[12.3]{mckernan}. 
Since the relative Picard number of $T$ is two, there is exactly one reducible fiber with two irreducible components $F_1$ and $F_2$. 
All other fibers lie in the smooth locus, since $E_2$ are $E_3$ are sections. Since all singularities on $T$ are Du Val, 
$F_1$ and $F_2$ must meet at an $A_5$ point at the opposite ends of the chain. The component $F_1$ cannot meet both $E_1$ and $E_2$ since otherwise $F_2$ would
be contractible on $S_1$. Therefore assume $F_i \cdot E_i=1$ for $i=1,2$. By contracting $F_2$, we reach the Hirzebruch surface $\mathbb{F}_3$. Let us denote by a bar
the image of the curves in $\mathbb{F}_3$. With this convention, $\overline{A_1}$ is a double section, $\overline{E}_1$ is the negative section and
 $\overline{E_2}$ is a positive section disjoint from $E_1$. Depending on whether $F_1 \cap A_1 = 2$ or $F_1 \cap A_1=1$, we have that $\overline{A_1}^2 = 6$
 or $\overline{A_1} ^2 = 0$ respectively. The first case cannot happen since numerically $\overline{A_1} \equiv 2E_1 + aF$ and the equation $-12+4a=6$ has no solutions.
The second case cannot happen either since $A_2$ cannot be a fiber.

\textbf{Case 2:}
Suppose now that $x_0 = (2,3,2,2)$, which gives $e_0 = 6/11$. 
For this case, the argument given in \cite[Chapter 15]{mckernan} goes through without any
changes, but we'll show it here too for sake of clarity. Since $K_{S_0}^2 = 1/(11\cdot 13)$, Lemma \ref{global_canonical} implies that there must be a
 $(3,A_5)$ point to compensate for the factor thirteen in the denominator.
The equality in Lemma \ref{global_canonical} is already satisfied with these two singularities, hence the additional singular points can only be $(4)$ points
by Lemma \ref{no_more_singularities}. By \cite[Lemma 12.3]{mckernan} all these additional points need to be on $M$, and since $K_{S_1}+M$
is negative, there can be at most one such point. If $M$ passed through a $(4)$ point, then $M$ would be smooth by adjunction. One could then contract $M$ on $S_1$, contradiction.
This means that $(2,3,2,2)$ and $(3,A_5)$ are the only singularities in $S_0$, so now one can proceed as in \cite[Definition-Lemma 15.2]{mckernan}.
\end{proof}

\begin{lemma}
Suppose $A_1$ has a double point and $g(A_1)\geqslant 2$. Then configuration $III$ does not arise.
\end{lemma}
\begin{proof} 
We will only sketch the proof, as the idea behind it is very simple but the computations are rather tedious. 
By Lemma \ref{multiplicity_two_node}, $a_1=(g+1)/(2g+1)$. We explicitly know the geometry of configuration $III$, so that if we were able to determine the type of singularity
at $x_0$ as well, we could compute

\[
  K_{S_0}^2 = \frac{(K_{S_0}\cdot \overline{\Sigma} _1)^2}{\overline{\Sigma} _1 ^2} 
\]  

One could then use a computer program to see if there are any combinations of singularities that satisfy Lemma \ref{global_canonical}. Note however that the possibilities
for $x_0$ are limited by the fact that $e_0<3/5$ and by Lemma \ref{singlist}. One can therefore repeat the above procedure for each case of the singularity at $x_0$
and conclude. In the following we only mention the cases in which Lemma \ref{global_canonical} has a solution as an example of the strategy we just described. 

The only cases that give solutions to Lemma \ref{global_canonical}
are $x_0 = (4)$, $x_0 = (4,2)$ and $x_0=(3,2)$, all for $g=2$. Suppose first that $x_0 = (4)$. By \cite[Lemma 12.1]{mckernan} we have that $M=\emptyset$. By \cite[12.3]{mckernan}
$F$ only passes through almost Du Val points. The only solution to Lemma \ref{global_canonical} that is compatible with $(K_{S_1}+F)\cdot F<0$, 
is $u=6$, $n_0=2$ and $n_1 = 1$. We may run therefore the same argument of Lemma \ref{config2}, Case 1. 

Suppose then that $x_0 = (4,2)$. Then $K_{S_0}^2 = \frac{2}{13\cdot 35}$. By Lemma \ref{global_canonical} and Lemma \ref{no_more_singularities}
the singular points of $S_0$ are $(4,2)$, $(3,2)$, $(2)$, $(3,A_5)$ and possibly some $(4)$ points.

\textbf{Case 1:} Suppose there is at least one $(4)$ point. The second hunt step extracts an exceptional $(-4)$ curve $E_2$, which lies in the smooth locus of $T_2$.
We also have that $a_2>b_2>e_1=1/2$. In particular, Proposition \ref{hunt} applies. 

\textbf{Case 1a:} Suppose furthermore that $T_2$ is a net. $E_2$ is then a multi-section. If $E_2$ is not a section, then $\Sigma_2\cdot E_2\geqslant 2$. This 
in turn implies that $(K_T+(3/5) A_1 + (1/2)E_2)\cdot \Sigma_2 > 0$, contradiction. If $E_2$ is instead a section, then $T_2$ is smooth, also a contradiction.

\textbf{Case 1b:} Assume that $(S_2, A_2+B_2)$ is a fence. Since $E_2$ is in the smooth locus, and since $B_2^2>0$, we must have that $\Sigma_2$ meets
$A_1$ at an $A_r$ point with $r\geqslant 4$, contradicting the fact that $A_1$ does not contain any such point.

\textbf{Case 1c:} Let $(S_2, A_2+B_2)$ be a tacnode of genus $g$. Since $E_2$ is in the smooth locus and since there is only a $(2)$ point in $A_1$,
the genus of the node is exactly one and we are in either configuration $I$ or $II$. In any case, $B_2^2<0$ since $E_2^2 = -4$, contradiction.

Since $A_1$ can't be contracted, and since $x_1\notin A_1$, the above cases exhaust all the possibilities in Proposition \ref{hunt}. Therefore there cannot be a $(4)$
point in $S_1$. This means that the second hunt step extracts the $(-3)$ curve of the $(3,A_5)$ singularity. We again go over all cases.

\textbf{Case 2a:} Suppose that $T_2$ is a net. Just as in case (1a) we may deduce that $E_2$ is a section. This however contradicts Lemma \ref{net_fiber_classification}.

\textbf{Case 2b:} Assume that $(S_2, A_2+B_2)$ is a fence. Just as in case (1b) this implies that $B_2^2 <0$, contradiction.

\textbf{Case 2c:}  Let $(S_2, A_2+B_2)$ be a tacnode of genus $g$. The only configuration compatible with the geometry at hand is $II$, where $\Sigma_2$ meets $E_2$
at the $A_5$ point. Therefore $g=5$. This implies that $B_2$ is in the smooth locus of $S_2$, $B_2^2 = 6$, and the only singular point of $S_2$ is a $(2)$ point.
This contradicts Lemma \ref{smoothp1}.

The above cases show that $x_0$ is not a $(4)$ point or a $(4,2)$ point.
We show now how one may exclude the case in which $e_0 < 1/2$ (the other cases are involve entirely analogous
computations). Let $x_0$ be of type $(3,A_r)$. Then

\[
  K_{S_0}^2 = \frac{2(g+r+2)^2}{(2g+1)(2r+3)(4gr+4g-1)}
\]

Since there is $(2)$ point and since $g-1 \leqslant r$, in the sum of the singularities of Lemma \ref{global_canonical}
we already have a term which is at least $1+2g(1-\frac{1}{2g+1})$. By Lemma \ref{no_more_singularities}
we get that $8>2g(1-\frac{1}{2g+1})$ and hence $g\leqslant 4$. For the case $g=4$ note that by Lemma \ref{global_canonical}
we have $r=3$, otherwise the right hand side would be negative. By the above we get $K_{S_0}^2 = 2/(7\cdot 9)$, which means
that there would be an $(3,A_2)$ point. This again contradicts Lemma \ref{global_canonical}. Suppose $g=3$. If
$3\leqslant r\leqslant 5$, the $12r+11$ part of the denominator has big primes in its factorization, contradicting Lemma 
\ref{global_canonical}. If $r=2$, then $K_{S_0}^2 = 2/35$, hence the only non Du Val singularities are of type $(3,2)$ and $(3,2,2)$.
However it's easy to see that $2/35 = 9 - u - (8/5)n_1 - (18/7)n_2$ has no solutions in integers. 

Finally, for the case $g=2$, running a computer program shows that there are only solutions for $r=1$:
\begin{enumerate}
\item $u=1$, $n_0=7$, $n_1=2$.
\item $u=3$, $n_0=4$, $n_1=2$.
\item $u=5$, $n_0=1$, $n_1=2$.
\end{enumerate}

$(1)$ and $(2)$ are impossible by adjunction. Therefore we assume $(3)$ and we go to the second hunt step, which extracts the unique $(-3)$ curve.
$T_2$ is not a net, for otherwise $E_2$ would be a smooth section, contradiction. If $\Sigma_2$ meets $A_1$ at a smooth point,
we must go to a Gorenstein log del Pezzo $S_2$ that contains at least two $A_1$ points and such that $K_{S_2}^2 = 3$. One may check that there are no such
surfaces by Theorem \ref{gorenstein_classification}. If $\Sigma_2$ meets $A_1$ at the $(2)$ point, then
both $A_2$ and $B_2$ must lie in the smooth locus and we must go to a Gorenstein log del Pezzo such that $K_{S_2}^2=5$. Therefore, $S_2=S(A_4)$.
Also, since $\pi_2$ is an isomorphism on $A_1$, we have that $(K_{S_2}+A_2)\cdot A_2 = 2$. By the description of configuration $III$, we have
$A_2^2 = 8$, which leads to $K_{S_2}\cdot A_2 = -6$. However this implies that $(K_{S_2}\cdot A_2)^2 \neq K_{S_2}^2 \cdot A_2^2$, contradiction.
\end{proof}

\begin{lemma}
If $\operatorname{char}(k)\neq 2$ and $\pi_1$ has type $U$ or $V$, then $S_0$ has a tiger.
\end{lemma}
\begin{proof}
The proof goes almost exactly along the lines of \cite[Lemma 15.4]{mckernan}, with only minor modifications. In particular,
we have that $F$ passes through at least two singular points (we need $\operatorname{char}(k)\neq 2$ in order to apply \cite[Lemma 15.3]{mckernan}), 
$M$ is empty, and $A_1$ is in the Du Val locus of $S_1$, with at most one singular point on it. 

Suppose first that $A_1$ is in the in the smooth locus of $S_1$. By \cite[Lemma 12.3]{mckernan} all the non Du Val points on $S_1$
are almost Du Val and $F$ does not pass through Du Val points. By \cite[15.4.1]{mckernan}, there can't be two non Du Val points on $F$, contradiction.

So there is exactly one Du Val point on $A_1$ (necessarily in $A\cap F$), and it must be a $(2)$ point by 
\cite[15.4.1]{mckernan}. Again by \cite[15.4.1]{mckernan}, there is exactly one other singular point on $F$, which is almost Du Val.
We can't be in configuration $U$, since otherwise $x_0=(2,3,2,2)$ or $x_0=(2,4,2,2)$.
In the first case, the point $(4,2,2)$ in configuration $U$ has higher coefficient, contradicting the choice of $x_0$ in the hunt.
In the second case, $x_0$ has coefficient $e(x_0) > 2/3$, contradicting Lemma \ref{multiplicity_two_cusp}. So we are in configuration $V$ and $x_0 = (4,2)$. 
Now we compute $(K_{S_1}+A_1)\cdot A_1 = 5/2$ by adjunction. Also, $A_1^2 = 15/2$ by the description of the configuration. This implies that $K_{S_1}\equiv -(2/3)A_1$
and that $K_{S_1}^2 = 10/3$. In particular, the only non Du Val point in $S_1$ is a $(3)$ point. By Lemma \ref{global_canonical} applied to $S_1$ we see that there must be
exactly four components coming from Du Val points outside $A_1$. Now extract the Du Val singularity on $A_1$ and the only non Du Val singularity. Following the argument in
Lemma \ref{config2}, Case 1, we get a contradiction.
\end{proof}

\begin{proof}[Proof of Proposition \ref{nobiggenus}]
This follows from Lemma \ref{multiplicity_two_node}, Lemma \ref{multiplicity_two_cusp} and the lemmas above.
\end{proof}

\subsection{$A_1$ has a simple cusp}\label{section_cusp}

Here we prove:

\begin{proposition}\label{no_cusp_genusone}
If $\operatorname{char}(k)\neq 2,3$ and $A_1$ has a cusp of genus one then $\tilde{S}_0$ is isomorphic to one of the following:
\begin{enumerate}
\item Take the cubic $C$ given by $Z^2X=Y^3$ in $\mathbb{P}^2$, the line $L$ given by $Y=0$ and a line $E$ meeting $C$ at $L\cap C$ and two other distinct points $p$ and $q$.
Blow up four times above $[0,0,1]$ along $C$. This gives the minimal resolution of the Gorenstein log del Pezzo $S(A_4)$. Now blow up twice on $p$ along $E$.
This gives the minimal resolution of the Gorenstein log del Pezzo surface $S(A_1+A_5)$. Next, blow up on the cusp of $C$ four times along $C$. 
\item Take the cubic $C$ given by $Z^2X=Y^3$ in $\mathbb{P}^2$, the line $L$ given by $Y=0$ and a line $E$ meeting $C$ at two points $p$ and $q$ with order two and one respectively.
Blow up three times above $[0,0,1]$ along $C$. This gives the minimal resolution of the Gorenstein log del Pezzo $S(A_1+A_2)$. Now blow up twice above $p$ along $E$. This
gives the minimal resolution of the Gorenstein log del Pezzo surface $S(3A_2)$. Next, blow up on the cusp of $C$ four times along $C$.
\item Take the cubic $C$ given by $Z^2X=Y^3$ in $\mathbb{P}^2$, the line $L$ given by $Y=0$ and a line $E$ meeting $C$ at $L\cap C$ and two other distinct points $p$ and $q$. Blow up three times above $[0,0,1]$ along $C$. This gives the minimal resolution of the Gorenstein log del Pezzo $S(A_1+A_2)$. Now blow up twice above
$p$ along $E$. This gives the minimal resolution of the Gorenstein log del Pezzo $S(2A_1+A_3)$. Now blow up at the cusp of $C$ either five or six times.
\end{enumerate}
\end{proposition}

In the following discussion we always assume that $\operatorname{char}(k)\neq 2,3$ and we discuss the case $\operatorname{char}(k)=5$ separately when needed.
First notice that $K_{T_1} + a_1 \Sigma_1$ and $K_{S_1}+a_1 A_1$ are flush by Proposition \ref{hunt}. 
The geometric configurations are described in Lemma \ref{multiplicity_two_cusp}. As case $I$ has been ruled out in Lemma \ref{cusp_one}, we may assume that $a_1 \geqslant 2/3$.
Clearly $A_1$ is not the in smooth locus of $S_1$, for otherwise it would be a tiger by adjunction. 
By Lemma \ref{multiplicity_two_cusp} we have that $2/3\leqslant a_1 < 4/5$. 
Note that $A_1 \neq \Sigma_2$ because $\Sigma_2$ is smooth. By Lemma \ref{coefficient_boundary} we have that $e_1 \geqslant 1/3$,
so that $a_1 + e_1 \geqslant 1$ and the second hunt step is classified in Proposition \ref{hunt}. Clearly $A_2$ is still a rational curve
with a single cusp, of genus one, by Proposition \ref{hunt}. Also, at least one of $-(K_{S_2}+A_2)$ and $-(K_{S_2}+B_2)$ is ample,
hence $-(K_{S_2}+B_2)$ is ample and $B_2$ is smooth.

\begin{lemma}\label{cusp_fence}
The pair $(S_2,A_2+B_2)$ is not a fence.
\end{lemma}
\begin{proof}
Assume $(S_2, A_2 + B_2)$ is a fence. 

\textbf{Case 1:} Suppose first that $A_2$ is not in the smooth locus of $S_2$. Then, since $K_{S_2}+(2/3) A_2$ is negative, 
we have a birational map to a Gorenstein log del Pezzo surface $W$ by \cite[Lemma-Definition 12.4]{mckernan}. 
As usual, we let $M$ be as in \cite[Chapter 12]{mckernan} (and $F=\emptyset$). 
The rank of $W$ is one more than the difference between the number of 
singular points on $A_2$ and the number of irreducible components of $M$. By \cite[Lemma 12.5]{mckernan}, no irreducible component
of $M$ passes through more than two singular points. Furthermore $0<(K_{S_2}+A_2)\cdot B_2 < 1$, hence $M$ meets $B_2$
at least once, and only at singular points of $S_2$. If the rank of $W$ were more than one, there would be a component of $M$
which passes through two singular points of $A_2$.
This is however absurd, since it would also meet $B_2$ at a singular point, and therefore it would pass through at least three singular points. So the rank
of $W$ is one. Let $G$ be the exceptional divisor adjacent to $A_2$. We have that $A_2$ is a cuspidal rational curve in the smooth locus of $W$, 
$A_2 \in |-K_W|$ and $K_W \cdot G = K_W \cdot B_2 = -1$. If $\operatorname{char}(k)\neq 5$ this contradicts Lemma \ref{cuspgorenstein}.
Suppose then that $\operatorname{char}(k)=5$. In this case $W$ is the Gorenstein log del Pezzo surface described in Example \ref{sa4char5}.
Therefore $W$ has exactly two singularities, $p$ and $q$, each of type $A_4$. Notice that $G$ and $B_2$ meet only once in $W$ 
and are both nodal, with each a node at one of $p$ and $q$.
This, however, is impossible by the description of the transformation $S_2 \dashrightarrow W$.

\textbf{Case 2:} Suppose now that $A_2$ is in the smooth locus of $S_2$. This clearly implies that $S_2$ is Gorenstein. 
Notice that since $K_{S_2}+B_2$ is dlt we can immediately rule out the case $\operatorname{char}(k)=5$ by the explicit description of $S(2A_4)$ in this case 
(Example \ref{sa4char5}).
Suppose therefore that $\operatorname{char}(k)\neq 5$. Again using the fact that $K_{S_2}+B_2$ is dlt we deduce that
$0<B_2 ^2 = -2+\text{deg(Diff}_{B_2} (0) \text{)} - K_{S_2}\cdot B_2$ by adjunction. However $K_{S_2} \equiv - A_2$, hence 
$\text{deg(Diff}_{B_2} (0)\text{)}>1$, and there are at least two singular points on $B_2$. 
By Lemma \ref{cuspgorenstein}, $S_2$ is $S(A_1+A_2)$, and is obtained by taking
a flex cubic in $\mathbb{P}^2$ and the tangent line to its flex, blowing up three times to separate them and the blowing down the
$(-2)$ curves. However then we have tigers: if $a_2\geqslant 5/6$ there is a tiger over the singular point of $A_2$, the $(-1)$ curve
of the resolution of the cusp, and otherwise $B_2$ is a tiger as $K_{S_2}+(5/6)A_2+B_2$ is numerically trivial. 
\end{proof}

\begin{lemma}
The Gorenstein log del Pezzo surface $W$ associated to $S_1$ as in \cite[Lemma-Definition 12.4]{mckernan} has rank at least two.
\end{lemma}
\begin{proof}
Suppose $W$ has rank one. 

\textbf{Case 1:} Suppose also that $A_1$ contains at least two singular points of $S_1$. By \cite[Lemma-Definition 12.4]{mckernan}, the image of $A_1$ in $W$ has a cusp
and meets two $(-1)$ curves. Furthermore these two $(-1)$ curves must meet the image of $A_1$ in $W$ at different points by \cite[Lemma 12.1 (7)]{mckernan}.
But then we obtain a contradiction by Lemma \ref{cuspgorenstein} as in Lemma \ref{cusp_fence} Case 1.

\textbf{Case 2:} Suppose now that $A_1$ contains just one singular point, which implies that $M$ of \cite[Definition-Lemma 12.0]{mckernan} is irreducible by 
\cite[Lemma-Definition 12.4 (1)]{mckernan}. Consider the morphism
$f:Y\rightarrow S_1$ extracting the exceptional divisor $G$ adjacent to $A_1$ and the morphism $\pi : Y\rightarrow W$ contracting $M$.
Define $\Gamma$ by $K_Y + \Gamma = f^* (K_{S_1} + a_1 A_1)$ and $\Gamma' = \lambda (\Gamma+\epsilon G)$ 
such that $K_Y + \Gamma'$ is $\pi$-trivial. Let $\Delta' = \pi(\Gamma')$. Clearly $K_W + \Delta'$ is negative by 
Lemma \ref{hunt_transformation}. Now $A_1$ is in the smooth locus of $W$ and $K_W \cdot G = - A_1 \cdot G = -1$.
If $K_W ^2 \geqslant 2$, then $K_W + G$ is anti-ample because $(K_W + G)\cdot A_1 < 0$, and hence $G$ is a smooth
$(-1)$ curve. If $K_W ^2 =1$, we consider the associated extremal rational elliptic surface. For the moment we would like to show that $G$ is smooth. There are two cases.

\textbf{Case 2a:} Suppose $\operatorname{char}(k)\neq 5$. Then $W=S(E_8)$ by Lemma \ref{cuspgorenstein}. Since
the pullback of $G$ is in $|-K_{\tilde{W}}|$, it follows by the description of the fibers of Theorem \ref{elliptic_classification} that $G$ is smooth. 

\textbf{Case 2b:} Suppose $\operatorname{char}(k)=5$. Since $K_{S_1}+M$ is negative, $M$ is a smooth rational curve. Again by 
Theorem \ref{elliptic_classification} we deduce that either $W=S(E_8)$, in which case we conclude as above, or $W=S(2A_4)$. 
Suppose we have the latter. The fact that the pullback of $G$ is in $|-K_{\tilde{W}}|$ implies that $G$ is a nodal rational curve, with the node passing through an
$A_4$ point. It is easy to see however that since $K_{S_1}+A_1$ is dlt at singular points, and since $M$ is smooth, 
there is no way to get such a geometric configuration. Therefore this case
does not occur either.

In conclusion, $G$ is a smooth $(-1)$ curve. Now, if $K_W+G$ is dlt, $A_1+G$ is a fence and we can proceed as in Lemma \ref{cusp_fence}.

Otherwise, $G$ meets a unique curve $V$ of the minimal resolution since it's a fiber of the associated extremal rational elliptic surface.
Let $h: Q\rightarrow W$ extract $V$. $G$ is a $(-1)$ curve in the smooth locus of $Q$, so we can contract it with
$r:Q\rightarrow W_1$. Notice that $K_{W_1}^2 = K_W ^2 -1$. Scaling again as in Lemma \ref{hunt_transformation} and
repeating the process with $A_1$ and $r(V)$, we can induct on $K_W^2$, by Lemma \ref{hunt_transformation}. Eventually we therefore we reach the case in which $K_W+G$ dlt,
which we can discard as above via Lemma \ref{cusp_fence}.
\end{proof}

\begin{lemma}
$S_1$ is singular along $A_1$ in at least two points and is Du Val outside $A_1$.
\end{lemma}
\begin{proof}
The first part is obvious from the above and second part follows from \cite[Lemma 12.5]{mckernan} (compare with \cite[Lemma 16.3]{mckernan}).
\end{proof}

\begin{lemma}
$S_1$ is not Gorenstein unless $S_0$ is one of the surfaces described in Proposition \ref{no_cusp_genusone}.
\end{lemma}
\begin{proof}
From now on we suppose that $S_1$ is Gorenstein. Let's start by noticing that 
$(K_{S_1}+A_1)\cdot A_1 \geqslant 1$ by adjunction. Since $(K_{S_1}+2/3A_1)\cdot A_1 < 0$, we have that $A_1^2 > 3$.
Note also that $K_{S_1}^2 > (4/9) A_1^2$, hence $K_{S_1}^2 \geqslant 2$. Also, recall that $a_1 < 4/5$ and $e_0 \geqslant e_1 > 0$, so that $E_1 ^2 <-2$.

\textbf{Case 1:} Suppose that $A_1$ passes through three singular points. By the classification of klt singularities, these would
be either $2A_1 + A_n$ or $A_1+A_2+A_k$, with $k\leqslant 4$. By Theorem \ref{gorenstein_classification} the only possibilities for $S_1$ 
are $S_1=S(2A_1+A_3)$ or $S_1=S(3A_1 + D_4)$. In the first case,
$4/5\leqslant e_0<a_1$, contradicting the fact that $a_1 < 4/5$. In the second case $a_1>e_0\geqslant 2/3$.
Since $(K_{S_1}+A_1)\cdot A_1 = 3/2$ we get that $K_{S_1}^2 \geqslant 3$ following the same reasoning at the beginning of the proof of this lemma, contradicting the fact that
$S_1 = S(3A_1 + D_4)$ implies that $K_{S_1} ^2 = 2$.

\textbf{Case 2:} So $A_1$ passes through just two singular points. We start by proving that they can't be both $A_1$ points. Suppose by contradiction
that both of the singular points are $A_1$ points. It immediately follows by adjunction that $(K_{S_1}+A_1)\cdot A_1=1$. Therefore $M\cdot A_1 = 1$, which implies that 
$K_{S_1}+M$ is dlt. Notice that $M$ does not contain other singular points by \cite[Lemma 12.5]{mckernan}.
Extract the $(-2)$ curves of the $A_1$ points and contract $M$ and one of the extracted curves. The other curve becomes a curve with self-intersection zero,
contradicting the fact that the Picard number is one.

Now we prove that the two singular points on $A_1$ can't be a $(2)$ and a $(2,2)$ point respectively. 
In fact, by the classification of Theorem \ref{gorenstein_classification} and the fact that 
$K_{S_1}^2\geqslant 2$, we see that $S_1=S(A_1+A_2)$ and $K_{S_1}^2=6$. 
But then $\tilde{A_1}^2 = K_W^2 > K_Y^2 =K_{S_1}^2 = 6$, and we obtain a contradiction thanks to \cite[Lemma 16.4]{mckernan}.

In summary, so far we have we have shown that if $A_1$ has a $(2)$ point on it, then it also has either an $A_3$ or an $A_5$, again by the
classification of Theorem \ref{gorenstein_classification}. Configurations  $II$, $u$, $w$, $(U;n)$, $(V;f)$, $(V;f^2)$ of Lemma \ref{multiplicity_two_cusp} do not occur,
for otherwise $x_0$ would be a non chain singularity and $e_0\geqslant 4/5$, contradiction.
Now we are left with configurations $V$, $(V;n)$, $(V;n^2)$. We have an improved estimate on $a_1$, namely $a_1 \geqslant 5/7$. 
If $K_{S_1}^2 \geqslant 6$, we are done as in the case $K_{S_1}^2 = 6$ described above. 
Also, there are no Gorenstein log del Pezzo surfaces with $K_{S}^2 = 5$ and at least two singularities, by the list in Theorem \ref{gorenstein_classification}. 
On the other hand, combining $K_{S_1}+5/7A_1 \leqslant 0$ and $(K_{S_1}+A_1)\cdot A_1 \geqslant 5/4$, we get that $K_{S_1}^2 \geqslant 3$. 
Hence $K_{S_1}^2$ is either three or four, and the only possibilities for $S_1$ are $S(2A_1+A_3)$, $S(A_1+A_5)$ and $S(3A_2)$.

Consider $M$ as in \cite[Chapter 12]{mckernan}. Every component of $M$ contains at least two singular points
 by \cite[Lemma 12.2]{mckernan} and the number of components
of $M$ is at most two by \cite[Lemma 12.1]{mckernan}. 
Suppose that $M$ has two components $M_1$ and $M_2$. Consider $W$ of \cite[Lemma - Definition 12.4]{mckernan}. 
Since $K_W ^2 = K_{S_1}^2 + 2 \leqslant 6$, $W$ is either $S(A_4)$ or $S(A_1+A_2)$ by Lemma \ref{cuspgorenstein}. 
Now $M_1+M_2$ is log canonical at $b=M_1\cap M_2$ and dlt away from $b$ by adjunction. It is easy now to see that there are no compatible
configurations with this geometric description. We may therefore assume that $M$ has only one component. Now we analyze the possible contractions $\pi_1$.

\textbf{Configuration $V$:} We have that $E_1^2 = -3$ since $e_0<a_1$. Recalling that $M\equiv K_{S_1}+A_1$ 
one then computes $(K_{S_1}+M)\cdot M = -2/3$ if $S_1$ is either $S(A_1+A_5)$ or
$S(3A_2)$, and $(K_{S_1}+M)\cdot M = 5/4$ if $S_1=S(2A_1+A_3)$. 
Adjunction implies that only the first two cases are possible, and in both of them $M$ passes only through the two singular points in $A_1$.
Furthermore, since $M\cdot A_1 = (K_{S_1}+A_1)\cdot A_1 = 4/3$ we have that $M$ is reduced and meets the same exceptional curves as $A_1$ in $\tilde{S}_1$.
In the case in which $S_1 = S(3A_2)$ extract one of the exceptional divisors $E$ of the $A_2$ points touching $A_1$.
By contracting $M$ we get that $W=S(A_1+A_2)$ and the image of $E$ is a curve of self intersection one which touches the strict transform of $A_1$ in two points with multiplicities
one and two. Working backwards, one gets the description in Lemma \ref{no_cusp_genusone} (2).

Suppose now $S_1=S(A_1+A_5)$ and extract the $(-2)$ curve $E$ above the $A_5$ point which is adjacent to $A_1$. Let $W$ be the surface obtained by contracting $M$. 
We have that $W=S(A_4)$ and the image of $E$ is a zero curve. Working backwards, one may see that $S$ is the surface described in Lemma \ref{no_cusp_genusone} (1).

\textbf{Configuration $(V; n)$:} Suppose for the moment that $E_1 ^2 = -3$. By computing $(K_{S_1}+M)\cdot M$, we see that the only 
choice is $S_1=S(2A_1+A_3)$, in which case we get $-3/4$. Proceeding as in Configuration V, we get description $(3)$ in the case of five blow-ups.

If $E_1^2 = -4$ instead, then we must have $S_1=S(2A_1+A_3)$ again since $e(x_0)<a_1=3/4$. This is excluded by adjunction.

\textbf{Configuration $(V;n^2)$:} Since there is a point of coefficient $2/3$ we necessarily have that $E_1^2 = -4$. If $S_1=S(A_1+A_5)$ or $S_1=S(3A_2)$, one computes
$K_{S_0}^2 = \frac{1}{18\cdot 19}$, which contradicts Lemma \ref{global_canonical}. Finally, if $S=S_1(2A_1+A_3)$, one gets description $(3)$ in the case of six
blow-ups.
\end{proof}

\begin{lemma}\label{no_banana_net}
If $\operatorname{char}(k)\neq 2,3$ then $S_2$ is not a banana or a net.
\end{lemma}
\begin{proof}
The proofs in \cite[Lemma 16.5, Lemma 16.6]{mckernan} carry through. We just point out that the argument in \cite[Lemma 16.5]{mckernan} does not
actually use the Bogomolov bound, and that the reduction to $S(A_1)$, $S(A_1+A_2)$ and $S(A_4)$ still holds without the simply connected
hypothesis thanks to Lemma \ref{gorenstein_classification} and Theorem \ref{elliptic_classification}.
\end{proof}

\begin{proof}[Proof of Proposition \ref{no_cusp_genusone}]
This follows from Lemmas \ref{cusp_fence} - \ref{no_banana_net} and Proposition \ref{hunt}.
\end{proof}
\subsection{$A_1$ has a simple node}\label{section_node}

Throughout this section we suppose that $\operatorname{char}(k)\neq 2,3$, that $A_1$ has a simple node and that $S_0$ doesn't have tigers in $\tilde{S}_0$.

\begin{remark}
The proof of \cite[Proposition 13.5]{mckernan} still works in our setting.
\end{remark}

We fix notation as in \cite[Chapter 17]{mckernan}.

\begin{notation}
Let $C$ and $D$ be the two branches of $A_1$ at the node, and $c,d$ be the points of $T_1$ where the branches meet $\Sigma_1$.
We may assume that the first two blow ups of $h:\tilde{T}_1\rightarrow S_1$ are along $C$. Let $r+1$ be the initial
number of blow ups along $C$, $r\geqslant 1$. Note that $d$ is necessarily singular.
\end{notation}

\begin{lemma}\label{node_facts}
Notation as above. Then:
\begin{enumerate}
\item $K_{T_1}+\Sigma_1+E_1$ is log canonical.
\item $\overline{\Sigma}_1$ has two smooth branches through $x_0$ and meets no other singularities.
\item If $c$ is smooth, then $d$ in an $A_r$ point, $r\geqslant 1$ and $a_1 = (r+1)/(r+2)$.
\item $T_1$ is singular at some point of $E_1\setminus E_1 \cap \Sigma_1$.
\item $A_1$ contains exactly one singularity.
\item $S_0$ has exactly two non Du Val points and $e_0 > 1/2$.
\item $a_1\geqslant 2/3$, and $a_1\geqslant 4/5$, unless we have (3) with $r\leqslant 2$.
\end{enumerate}
\end{lemma}

\begin{proof}
This is \cite[Lemma 17.2]{mckernan}. The same proof applies, using Lemma \ref{nodegorenstein}.
\end{proof}

Let $x_0$ and $y$ be the non Du Val points on $S_0$, and let $z$ be the singular point of $S_1$ contained in $A_1$, of index $s$.

\begin{proposition}\label{node_to_fence}
$(S_2, A_2+B_2)$ is a fence. $g(A_2)=1$ and $B_2$ is smooth.

\begin{enumerate}
\item If $x_1\in A_1$ then $\Sigma_2$ meets $E_2$ at a smooth point, and contains a unique singular point
$(A_t, 3, A_{j-2})$ for some $t$. $K_T + \Sigma_2$ is dlt and $\Sigma_2$ meets the end of the $A_{j-2}$
chain. $(S_2,A_2 + B_2)$ is given by \cite[13.5]{mckernan} and $q_2$ is an $A_{t+1}$ point.

\item If $x_1\notin A_1$ and $A_2$ is not in the smooth locus of $S_2$, then $S_2$ is obtained by starting from
$S(2A_1+A_3)$ (see Lemma \ref{nodegorenstein} for a detailed description of this surface),
picking a $(-1)$ curve $B$, blowing up on the $(-1)$ curve $C\neq B$ at the $A_3$ point once and then contracting $C$. 
Let $A$ be the nodal curve contained in the smooth locus of $S(2A_1+A_3)$.
To obtain $S_1$ blow up on the intersection of $A$ and
$B$ twice along $A$, then contract $B$. To obtain $S_0$ blow up twice along one of the branches of $A_1$ and then
contract $A$. In particular $x_0$ is a chain singularity.

\item If $x_1\notin A_1$ and $A_2$ is in the smooth locus of $S_2$, then one of the following holds:
\begin{enumerate}
\item $a_2 < 6/7$ and  $(S_2, A_2+B_2)$ is given by \cite[13.5.1]{mckernan}, or 
\item $S_0$ is obtained as follows: start with $S(A_1+A_5)$ or $S(3A_2)$, consider a $(-1)$ curve $B$ passing through two of the singularities,
consider a rational nodal curve $A$ in the smooth locus, blow up twice on $A\cap B$ along $B$, three times on the node of $A$ along the same branch,
and then contract all the negative curves with self intersection less than $(-1)$; or 
\item $S_0$ is obtained as follows: start with $S(2A_1 + A_3)$, consider a $(-1)$ curve $B$ passing through two singularities, blow up on $A\cap B$ twice along $B$,
blow up on the node either four or five times along the same branch, the contract as above; or
\item $S_0$ is obtained as follows: start with $S(2A_1+A_3)$, consider a $(-1)$ curve $B$ passing through two singularities, blow up three times on $A\cap B$ along $B$, 
blow up four times on the node along the same branch and then contract as above. 
\end{enumerate}
 In particular, if $S_2 \neq S(A_1+A_2)$ then $x_0$ is a chain singularity. 
\end{enumerate}
\end{proposition}

\begin{proof}
Part $(1)$ follows as in \cite[Proposition 17.3]{mckernan}, therefore from now on we assume that $x_1\notin A_1$. 

\textbf{Case 1:} Suppose also that $A_2+B_2$ has a node of genus $g\geqslant 2$. Then we have that
$z\in \Sigma_2$, $a_1 = 2/3$, $g=2$, $r=1$, $E_1^2 \leqslant -3$ exactly as in \cite[Proposition 17.3]{mckernan}. The point $x_0$
is a chain singularity of type $(2,2,-E_1^2, 2)$, hence $E_1^2=-3$, for otherwise the coefficient would be too large.
Thus $K_{S_2}^2 = A_2 ^ 2 = r+4+g+E_1^2 = 4$. Now notice that $(K_{S_2}+2/3A_2+B_2)\cdot B_2 > 0$ because there are no tigers,
hence $(K_{S_2}+B_2)\cdot B_2> -2/3$. By adjunction there are at least two singular points of $S_2$ on $B_2$. By checking
the list in Theorem \ref{gorenstein_classification}, we see that $S_2=S(2A_1+A_3)$ because it's the only surface with $K_S^2 = 4$ and at least two singularities. 
But then, applying adjunction again we see that $B_2$ has to contain all of these singularities. This implies that $x_1$ is a non chain singularity. Furthermore, since it's not Du Val
by Lemma \ref{node_facts}, $E_2^2 \leqslant -3$ and $e_1>2/3$, contradiction. 

\textbf{Case 2:} So we must go to a fence. If $A_2$ is not in the smooth locus of $S_2$, then the proof of \cite[Proposition 17.3]{mckernan} gives the result.
Suppose therefore that $A_2$ is in the smooth locus. If $S_2 = S(A_1 + A_2)$, the result follows again by the proof of \cite[Proposition 17.3]{mckernan}.
Assume then that $S_2\neq S(A_1+A_2)$.

\textbf{Case 2a:} Let's start with the case in which $x_1$ is a non chain singularity. Since $S_2$ has no tigers
we must have $(K_{S_2}+A_2+B_2)\cdot B_2 > 0$.
It then follows that $B_2$ contains at least two singularities and they can't be both $(2)$ points. Looking at the possibilities of 
\cite[Lemma 13.5]{mckernan} and comparing them with the classification of non chain klt singularities, we see that
$\Sigma_2$ meets $E_2$ at a point of index two, or three if $S_2 = S(2A_1+A_3)$.
If $\Sigma_2$ meets $E_2$ at a $(2)$ point, then $z=(A_k,3)$ and $\tilde{E}_2^2 = -3-k$, with $k\geqslant 0$. Since $S_2\neq S(A_1+A_2)$
we have that $A_2^2 = K_{S_2}^2 \leqslant 4$, which means that $K_{\tilde{S}_1} \cdot \tilde{A}_1 \geqslant -3$ by adjunction and
the description of the configuration. Hence $K_{S_1}\cdot A_1 \geqslant -3 + 1/(2k+3)$. We also have that
$(K_{S_1}+A_1)\cdot A_1 = (2k+2)/(2k+3)$ by adjunction, and therefore by linearity of the intersection product we get

\[
   a_1 < \frac{3-1/(2k+3)}{3-1/(2k+3)+(2k+2)/(2k+3)}=\frac{3k+4}{4k+5}
\]

One checks that the singularity at $x_1$ has always greater coefficient, contradiction.

Let's consider now the case where $\Sigma_2$ meets $E_2$ at a $(3)$ point. If it meets $A_1$ at a $(2)$ point, then 
$E_1$ is a $(-2)$ curve, $\tilde{A}_1^2 = 2$, $K_{S_1}\cdot A_1 = -2$ and $(K_{S_1}+A_1)\cdot A_1 = 1/2$. That gives
us that $a_1< 4/5=e_1$, contradiction. Otherwise $\overline{\Sigma}_2$ meets $A_1$ at an $(A_k, 3, 2)$ point.  By making the same
computations as above, we get that $\tilde{E}_2 ^2 = -3-k$ and

\[
   a_1 < \frac{3-2/(3k+5)}{3-2/(3k+5) + (3k+4)/(3k+5)}=\frac{9k+13}{12k+17}
\]

Again, the singularity at $x_1$ has always greater coefficient, contradiction. 
The only remaining case is that $\Sigma_2$ meets $E_2$ at an $A_2$ point and $A_1$ at a $(A_k, 4)$ point. We do not treat this case, as it may be ruled out by 
analogous computations as above.

\textbf{Case 2b:} Let's consider now the case in which $x_1$ is a chain singularity. Then the configuration is given by $(3)$ of Lemma 
\ref{multiplicity_two_node}, $\Sigma_2$ meets $A_1$ at an $A_k$ point and $b_2+a_2/(k+1)=1$. 
Since $(K_{S_2}+a_2 A_2 + b_2 B_2)\cdot A_2 < 0$, we get that

\[
   a_2 < \frac{K_{S_2}^2 - 1}{K_{S_2}^2 - 1/(k+1)}
\]

We must have $K_{S_2}^2 \geqslant 3$, for otherwise $a_1 < 2/3$, contradicting Lemma \ref{node_facts} (7).

If $K_{S_2}^2 = 3$, then $S_2 = S(A_1+A_5)$ or $S(3A_2)$ by Theorem \ref{gorenstein_classification}. Also, $a_1 \leqslant 3/4$, again by the above inequality and Lemma \ref{node_facts}.
Since $B_2$ is a $(-1)$ curve, we get that $\tilde{E}_2 ^2 = -2 - k$, and hence $k=1$ because $e_1<a_2$.
One computes $e_1 = 3/5$, hence $\tilde{E}_1^2 = -4$ and $r=2$. This is case $(b)$ in Proposition \ref{node_to_fence} (3).

Suppose now that $K_{S_2}^2 = 4$, which implies that $S_2 = S(2A_1+A_3)$. One computes $e_1 = \frac{4k}{4k+3}$.
Setting $e_1 < a_2$ we get that $k=1$ or $k=2$, and that $a_2 < 6/7$ or $a_2 < 9/11$ respectively.
Suppose that $x_0$ is a chain singularity. It follows from Lemma \ref{node_facts} that $a_1 = (r+1)/(r+2)$. Let's start with $k=1$. Then $\tilde{E}_1 ^2 = -1-r$, and 
$e_0 = \frac{2r^2-2}{2r^2+r+1}$. The only solutions to $e_1<e_0<a_2$ are $r=3$ or $r=4$. This leads to case $(c)$ of Proposition \ref{node_to_fence} (3).
In the case with $k=2$, one finds again $r=3$ exactly as above. This is case $(d)$ of Proposition \ref{node_to_fence} (3).

If $x_0$ is a non chain singularity instead there are no solutions, for one may check that if there are many blow ups at the node,
the coefficient is at least $6/7$, and if there are few then the coefficient the coefficient is lower than $e_1$.
\end{proof}

\begin{lemma}\label{nodenonchain}
Suppose that $x_0$ is a non chain singularity. Then $x_1\in A_1$ and
$S_0$ is obtained by blowing up the end of the $A_5$ point in $S(A_1+A_5)$ along the $(-1)$ curve $Y$, 
then blowing up the node twice along one branch 
and then once along the nearest point of the other branch, and finally contracting down all the $K$-positive curves.
\end{lemma}

\begin{proof}
The proof is the same as in \cite[Lemma 17.4]{mckernan}, by using Proposition \ref{node_to_fence} instead of 
\cite[Proposition 17.3]{mckernan}.
\end{proof}

We conclude this section by noting now that the same classification as in \cite[17.7 - 17.14]{mckernan}
carries through, by the previous results.
\subsection{Smooth fences}\label{smooth_fences}

In this subsection we collect some useful facts about abstract smooth fences. We will then apply these results during the hunt in the next sections. 
Throughout the subsection $(S,X+Y)$ will be a smooth fence. By this we mean that $S$ is a rank one log del Pezzo surface and that $X+Y$ is a fence, with
$X$ and $Y$ both smooth rational curves. No assumption about the existence of tigers and about the characteristic of the ground field are made.
We define $\alpha = -(K_S+X)\cdot X$ and $\beta = -(K_S+Y)\cdot Y$. Let's begin with the following elementary but fundamental lemma.

\begin{lemma}\label{fence_tactic} The following hold:
\begin{enumerate}
\item $\alpha = 1$ if and only if $\beta=1$.
\item $\alpha < 1$ if and only if $\beta<1$.
\item If $\beta \neq 1$ then $X^2 = \frac{1-\alpha}{1-\beta}$. 
\item If $\beta \neq 1$ then $(K_S+tX)\cdot X = (t-\beta)\frac{1-\alpha}{1-\beta} - 1$.
\item If $K_S + aX + bY$ is anti-ample and $\alpha,\beta<1$, then $a(1-\alpha)+b(1-\beta) < 1-\alpha \beta$. 
\end{enumerate}
\end{lemma}
\begin{proof}
Since $S$ is of rank one and $X+Y$ is a fence, $K_S + X \equiv -\alpha Y$ and $K_S + Y \equiv -\beta X$. 
Hence $K_S + X + \alpha Y \equiv K_S + Y + \beta X$, from which $(1-\beta) X\equiv (1-\alpha) Y$. Now the results follows easily.
\end{proof}

\begin{lemma}\label{no_minus_one_curves}
If $\alpha > 0$ then $\tilde{Y}^2 \neq -1$.
\end{lemma}
\begin{proof}
Suppose that $Y$ is a $(-1)$ curve. Then $K_S \cdot Y \geqslant -1$, hence $(K_S+X)\cdot Y\geqslant 0$. It follows then that $K_S+X$ is nef,
contradicting the fact that $\alpha > 0$.
\end{proof}

\begin{lemma}\label{fence_net}
Suppose $Y$ passes through exactly two singular points and $X$ passes through at least two singular points. Then $\tilde{Y}^2 \neq 0$.
\end{lemma}
\begin{proof}
Suppose by contradiction that $\tilde{Y}^2 = 0$. Let $f:T\rightarrow S$ be the extraction of the two adjacent curves $L$ and $M$ to $Y$. 
Since $\tilde{Y}^2 = 0$, then we have a contraction
$\pi: T \rightarrow \mathbb{P}^1$. Clearly $Y$ is a fiber and $X$ and the two extracted curves are sections. Since the relative Picard number of the
contraction is two, there is exactly one reducible fiber. However, since $X$ contains at least two singularities of $T$, there is a multiple and irreducible fiber $F$.
This fiber can contain at most two singularities, hence it touches either $L$ or $M$ at a smooth point, contradiction.
\end{proof}

\begin{lemma}\label{fence_two_two}
$X$ and $Y$ can't both contain exactly two singular points.
\end{lemma}
\begin{proof}
Suppose they do. Then $\alpha,\beta > 0$. Without loss of generality we may assume that $Y^2\leqslant 1$. However, this contradicts either Lemma \ref{no_minus_one_curves} or
Lemma \ref{fence_net}.
\end{proof}

\begin{lemma}\label{fence_du_val}
Let $S(X+Y)$ be a fence such that $X$ contains at least two singular points of $S$ and $Y$ contains exactly two singular points, one Du Val and the other a $(n)$ point. 
Then $\tilde{Y}^2 \neq -1$.
\end{lemma}
\begin{proof}
Suppose $\tilde{Y}^2 = -1$. Let $f:T \rightarrow S$ be the extraction of the $(-n)$ curve $E$ on $Y$. Then the strict transform of $Y$ is the curve
that gets contracted in Lemma \ref{hunt_transformation}. Clearly we get a birational morphism $\pi: T \rightarrow S_1$, and the image of $E$
is a smooth rational curve in the smooth locus of $S_1$. However, since there are at least two singularities along $X$, we get a contradiction
by Lemma \ref{smoothp1}.
\end{proof}

\begin{lemma}\label{fence_one}
Let $S(X+Y)$ be a fence such that $X$ contains exactly three singular points and $Y$ contains exactly two singularities. 
Suppose also that one of the two singularities on $Y$, $p$, is an $A_1$ point. Then $\tilde{Y}^2 \neq -1$. 
\end{lemma}
\begin{proof}
Suppose $\tilde{Y}^2 = -1$. Let $f_0:T_1 \rightarrow S$ be the extraction of the divisor $E$ adjacent to $Y$ which does not lie over $p$. As above,
let $\pi_1: T_1 \rightarrow S_1$ be the contraction of $Y$.  We have that $X$ and $E$ meet once with order two on $S_1$. If $E$ is in the smooth locus of $S_1$, we are done
by Lemma \ref{smoothp1}. Otherwise there is a singular point on $E$. Notice that $E$ is either a $(-1)$-curve or a zero curve on $S_1$ since is has self intersection at most $(-2)$ on
$T_1$ and $K_{S_1}\cdot E<0$.
It can't be a $(-1)$ curve, for otherwise it would contract on $S_1$, contradiction. Therefore it is a zero curve.
Let $f_1:T_2\rightarrow S_1$ be the extraction of the adjacent divisor $G$ to $E$. Since $E$ is a zero curve, it is contractible and therefore $T_2$ is a net. Clearly $G$ is a section.
There is at most one singular fiber on $T_2$ since $G$ has at most one singular point. However there are two singular points in $X$, which are not contained in $G$. Since a multiple
fiber can contain at most two singular points, we have a contradiction in this case too. In conclusion, $\tilde{Y}^2 \neq -1$.
\end{proof}

\begin{lemma}\label{fence_summary}
Suppose that $X$ contains at least three singular points and that $Y$ contains exactly two singular points. Then $\beta \leqslant 2/3$.
\end{lemma}
\begin{proof}
Suppose $\beta>2/3$. Then there is an $A_1$ singularity on $Y$ by adjunction. Since $X$ has at least three singular points $\alpha\leqslant 1/2$,
so that $Y^2 = \frac{1-\beta}{1-\alpha} < 1$. Therefore $\tilde{Y}^2$ is either zero or negative one and we can conclude by Lemmas 
\ref{fence_one} and \ref{fence_net}.
\end{proof}

\begin{lemma}\label{small_alpha}
Suppose that $X$ contains exactly three singular points and that $Y$ contains exactly two singular points. Then $\alpha \leqslant 1/6$, and if $\alpha>0$ then 
$\beta<\alpha$.
\end{lemma}
\begin{proof}
Let's start by noticing that if $\alpha \leqslant \beta$, then $Y^2 \leqslant 1$ by Lemma \ref{fence_tactic}. Therefore $Y$ must be a $(-1)$ curve by 
Lemma \ref{fence_net}. In that case we must have $\alpha\leqslant 0$ by Lemma \ref{no_minus_one_curves}.

It remains to see that $\alpha\leqslant 1/6$. Suppose that $\alpha > 1/6$. We must have $\beta<\alpha$ for otherwise $\alpha \leqslant 0$ by the previous remark. 
By adjunction we see that the singularities on $X$ are two $A_1$ points and a point of index $k$,
with $2\leqslant k \leqslant 5$. A simple local computation shows that $X^2 = l/k$, with $k$ and $l$ coprime. Since $\alpha = 1/k$ by adjunction, we also have 
that $X^2 = \frac{(k-1)/k}{1-\beta}$ by Lemma \ref{fence_tactic}. This leads to $\beta=1-\frac{k-1}{l}$. By the above $\beta<\alpha$, therefore $\frac{k-1}{l}>\frac{k-1}{k}$ and $l<k$. 
But then $\beta\leqslant 0$, contradiction by adjunction.
\end{proof}

\begin{lemma}\label{fence_three_two}
Suppose that $X$ contains exactly three singular points and that $Y$ contains exactly two singular points. Then $K_S + (4/5)X + (2/3)Y$ is nef.
\end{lemma}
\begin{proof}
Clearly $\alpha<1$ by adjunction. Suppose by contradiction that $K_S + (4/5)X + (2/3)Y$ is anti-ample.
By Lemma \ref{fence_tactic}, we have that $(4/5)\alpha + (2/3)\beta - \alpha \beta > 7/15$. Notice that $\alpha\leqslant 1/6$ by Lemma \ref{small_alpha} and that
$\beta\leqslant 2/3$ by Lemma \ref{fence_summary}. But then
\[
  7/15 = (4/5)\cdot(1/6) + (2/3)\cdot(2/3) - (1/6)\cdot(2/3)\geqslant (4/5)\alpha + (2/3)\beta - \alpha \beta > 7/15
\]
which is a contradiction.
\end{proof}

\begin{lemma}\label{fence_three_two_one}
If $X$ contains at most one singular point, so does $Y$. 
\end{lemma}
\begin{proof}
Suppose $Y$ contains at least two singular points. Then $\beta\leqslant 1<\alpha$, but this can't happen by Lemma \ref{fence_tactic}.
\end{proof}

Now we characterize smooth fences with that contain exactly one singular point on each branch.

\begin{lemma}\label{fence_one_one}
Suppose that $X$ and $Y$ contain both exactly one singular point and assume $Y^2 \leqslant X^2$. 
Then $\tilde{S}$ is obtained as follows: on the Hirzebruch surface $\mathbb{F}_n$ pick the section with negative self intersection $C$, choose a
disjoint section $D$ (which has necessarily positive self intersection) and choose a fiber $F$; now blow up once at $F\cap C$ or $F\cap D$, 
then blow up once more at the intersection of the exceptional divisor and $F$ and keep blowing up at either end of the $(-1)$ curve. 
To get $S$ we contract all the $K$-positive curves. $X$ will be the strict transform of $D$ and $Y$ will be the strict transform of a fiber different from $F$.
\end{lemma}
\begin{proof}
Notice that $Y$ is a curve with zero self intersection on the minimal resolution since $Y^2 \leqslant 1$ and $Y$ passes exactly through one singular point.
Therefore, after extracting it's adjacent divisor $E$, we get a net $T$. It's clear that $Y$ is a fiber, $X$ and $E$ are sections. 
The divisor $E$ does not lie in the smooth locus of $T$, for otherwise $T$ would be smooth. Therefore $E$ passes through a unique singular point, which lies on the
same fiber as the singularity contained in $X$. Now the result easily follows from running a relative MMP on the minimal resolution of $T$, just as in \cite[11.5.4]{mckernan}.
\end{proof}
\subsection{$A_1$ is smooth}\label{section_smooth}

In this subsection we classify all cases in which $A_1$ is smooth. Recall that $K_{S_1}+a_1 A_1$ is flush and that $K_{S_1}+A_1$ is dlt by
Proposition \ref{hunt}. Since $S_0$ has no tigers in $\tilde{S}_0$, we must have $(K_{S_1}+A_1)\cdot A_1>0$. We deduce then by adjunction that $A_1$ contains at least three singularities.
On the other hand, by the description of klt singularities and by the description of the hunt contractions, $A_1$ contains at most four singularities. Most of this subsection
is devoted to proving the following proposition.

\begin{proposition}\label{must_be_banana}
Suppose $\operatorname{char}(k)\neq 2,3$.
If $A_1$ is smooth, then $A_1$ contains exactly three singularities and $(S_2, A_2+B_2)$ is a smooth banana.
\end{proposition}

Let's start with a preliminary lemma.

\begin{lemma}\label{three_four_singularities}
$\tilde{A}_1 ^2\geqslant -1$, $a_1>2/3$, $e_1>1/2$.
\end{lemma}
\begin{proof}
$\tilde{A}_1^2\geqslant -1$ is clear, since $K_{S_1}\cdot A_1<0$.

Notice that the lemma follows from \cite[Lemma 18.2.4]{mckernan} in the case in which $A_1$ contains three singularities of $S_1$. 
Suppose then that $A_1$ contains four singularities. This implies that $x_0$ is a non chain singularity and that $\Sigma_1$ meets $E_1$ at a smooth point.
Let's start by showing that $a_1>2/3$. If $e_0\neq 1/2$ then $a_1>e_0\geqslant 2/3$ and we are done.
Suppose that $e_0=1/2$. If there was a chain singularity with maximal coefficient we would have chosen that singularity in the hunt by Definition \ref{huntchoice}.
Therefore all chain singularities are either Du Val or almost Du Val. Let the branches of $x_0$ be $(2)$, $(2)$ and $(A_r, 3)$.
In order to create the fourth singular point, there is a $(3,A_k)$ point (with possibly $k=0$) in $S_0$. 
If $k\geqslant 1$ and $\Sigma_1$ meets the end corresponding to a $(-2)$-curve, then $a_1>2/3$.
Suppose therefore that $\Sigma_1$ meets the $(-3)$ curve.
Then $K_{S_0}\cdot \overline{\Sigma}_1 = -1/2 + (k+1)/(2k+3)$ and $\overline{\Sigma}_1^2 = -(k+2)/(2k+3)+r+3/2$.
By Lemma \ref{global_canonical} and the estimate $K_{S_0}^2>0$ we immediately get $r+k\leqslant 4$.
To rule out each one of these cases, one computes:
\[
  K_{S_0}^2 = \frac{(K_{S_0}\cdot \overline{\Sigma}_1)^2}{\overline{\Sigma}_1^2}=\frac{1}{2(2k+3)(4k+4kr+6r+5)}
\]

The only solutions to Lemma \ref{global_canonical} are given by $k=r=0$, in which case there must be three extra $(3,2)$ points and no other singularities.
In this situation, $a_1 = 2/3$ and the $(3)$ point on $A_1$ has coefficient $5/9$. The second hunt step extracts this divisor, which we call $E_2$.

\textbf{Case 1:} Suppose that $T_2$ is a net. Then $E_2$ is a multi-section. $A_1$ is not a fiber since it contains three singular points. Since
$(K_{T_2}+(2/3)A_1 + (5/9)E_2)\cdot F < 0$ for the general fiber, and since $E_2$ cannot be a section, we see that $A_1$ must be a section. 
There are therefore three triple fibers, each containing a $(3,2)$
and an $A_1$ point. This implies that $E_2$ is a triple section, which is ramified at least in three points with order three. This, however, contradicts Riemann-Hurwitz.

\textbf{Case 2:} Suppose that $\pi_2 : T_2\rightarrow S_2$ is a birational contraction. The curve $A_1$ is not contracted since it contains three $(2)$ points. 
$(S_2, A_2 + B_2)$ cannot be a banana, for otherwise $B_2$ would be in the smooth locus of $S_2$, contradicting Lemma \ref{smoothp1}.
$(S_2, A_2 + B_2)$ cannot be a tacnode because the curve $A_1$ only contains $A_1$ singularities and $E_2$ is in the smooth locus of $T_2$.
$(S_2, A_2 + B_2)$ cannot be a smooth fence by Lemma \ref{fence_three_two_one} and cannot be a singular fence either by the description of the singularities in $S_2$.
Therefore we get a contradiction by Proposition \ref{hunt}.

Let's now show that $e_1>1/2$. If there is either a non Du Val point or a point of index at least four on $A_1$,
we conclude by Lemma \ref{coefficient_boundary}. If not, then $E_1$ is not a $(-2)$ curve, for otherwise $e_0=0$. If there is a branch of index three in $x_0$,
we get that $a_1>e_0\geqslant 3/4$ and we conclude again by Lemma \ref{coefficient_boundary}. 
Suppose therefore that the three branches of $x_0$ are all $(2)$ points. If $E_1 ^2 \leqslant -4$ we have that $a_1 > e_0\geqslant 4/5$, so that
$e_1>1/2$ by Lemma \ref{coefficient_boundary}.
If not, $E_1^2 = -3$ and $e_0 = 2/3$. The fourth singular point on $A_1$ can't be $(2)$, for otherwise there would be a tiger, hence it's $(2,2)$.
That means that $\Sigma_1$ meets the first curve of an $(A_k, 3, 2)$ singularity, with $k$ possibly zero.
Now we can compute $K_{S_1}\cdot A_1 = -k$ and $A_1^2 = k + 1/6$, hence $K_{S_1}^2 = 6k^2 / (6k+1)$.
If we had $e_1\leqslant 1/2$, using Lemma \ref{global_canonical} one obtains immediately a contradiction.
\end{proof}

\begin{lemma}
$(S_2, A_2+B_2)$ is not a tacnode.
\end{lemma}
\begin{proof}
Suppose that $S_2$ is a tacnode and define $l$ so that $A_1$ is a $(l-2)$ curve. As in \cite[Lemma 18.3 - 18.5]{mckernan}, one shows that 
\begin{enumerate}
\item $S_1$ is Du Val along $A_1$.
\item $(K_{S_1}+A_1)\cdot A_1 \geqslant l/3\geqslant 1/3$.
\item $(K_{S_1}+(3/4) A_1)\cdot A_1>0$.
\end{enumerate}

By Lemma \ref{multiplicity_two_node} we have that $c+dg/(g+1)=1$, where $d$ is the coefficient of the branch containing the $A_g$ point.
Since $a_2>2/3$ and $b_2>1/2$, the only possibility is that $g=2$ and the singularity lies on the curve $A_1$ (in other words, $c=b_2$ and $d=a_2$).

\textbf{Case 1:} Suppose that $A_1$ contains exactly three singular points. Let the two singular points that are not contained in $\Sigma_1$ 
be an $A_t$ and an $A_m$ point respectively, with $t\geqslant m$. By adjunction $(K_{S_1}+A_1)\cdot A_1 < 2/3$, hence $l=1$. Since $S_1$ is Du Val along $A_1$, we have 
that $K_{S_1}\cdot A_1 = -1$ and $A_1 ^2 = -1/3 + t/(t+1) + m/(m+1)$. However $(K_{S_1}+3/4 A_1)\cdot A_1>0$, hence $-1-1/4+3/4(t/(t+1)+m/(m+1))>0$. 
This implies that $t/(t+1)+m/(m+1)>5/3$. Now recall that $b_2+(2/3) a_2 = 1$. Since $b_2 > a_2 t/(t+1)$ by Lemma \ref{coefficient_boundary} and since $a_2>2/3$, we have that
$t/t(+1)<5/6$, which contradicts the previous inequality.

\textbf{Case 2:} Let's consider the case in which $A_1$ contains four singular points. Then $x_0$ is necessarily a non chain singularity. 
It follows by adjunction that $(K_{S_1}+A_1)\cdot A_1<2/3$ and therefore $l=1$. Also, $e_0=2/3$ since $a_1<3/4$. 
This implies that $E_1^2 = -2$ by inspecting all non chain singularities of coefficient $2/3$. The fourth singularity on $A_1$ is necessarily created by contracting $\Sigma_1$, which meets
$E_1$ at a smooth point and meets a $(3,A_k)$ point at its $(-3)$ curve. But then $K_{S_0}\cdot \overline{\Sigma}_1 \geqslant -1 + 1/3 + 2/3 = 0$, contradiction.
\end{proof}

The next step towards the proof of Proposition \ref{must_be_banana} is to prove that $S_2 (A_2 + B_2)$ can't be a fence. The main tools for this will be the results
in Section \ref{smooth_fences}. Let's start with the following.

\begin{lemma}
Suppose that $S_2(A_2+B_2)$ is a fence and $x_1\notin A_1$. Then $\Sigma_2$ can't meet $E_2$ at an $A_r$ point and $A_1$ at a smooth point.
\end{lemma}
\begin{proof}
Suppose it does. Then $a_2 + b_2 / (r+1)=1$, hence $b_2 < (r+1)/(r+2)$. Since $S_0$ has no tigers in $\tilde{S}_0$, $B_2$ contains at least two singular points. This implies
that $B_2$ contains exactly two singular points and that $x_1$ is a non chain singularity. 
We know from Lemma \ref{three_four_singularities} that $e_1 > 1/2$, which implies that $e_1\geqslant 2/3$ since $x_0$ is a non chain singularity.
Therefore $b_2 > e_1 \geqslant 2/3$. The two singular points on $B_2$
can't be both $A_1$ points because $S_0$ would have a tiger, hence $r\leqslant 4$ by the classification of klt singularities. Note also that $r\neq 1$ 
because $b_2 > 2/3$, so that the two singular points on $B_2$ are a $(2)$ point and a point of index $d$, with $3\leqslant d \leqslant 5$.
But then we conclude by Lemma \ref{fence_summary}.
\end{proof}

\begin{lemma}
Suppose that $S_2(A_2+B_2)$ is a fence and $x_1\notin A_1$. Then $\Sigma_2$ can't meet $E_2$ at a smooth point and $A_1$ at an $A_r$ point.
\end{lemma}
\begin{proof}
Suppose it does. Then $a_2 / (r+1) + b_2 = 1$, hence $a_2 > (r+1)/(r+2) > b_2$. 
On the other hand $(K_{S_1}+A_1)\cdot\overline{\Sigma}_2 = -1 + e_1 + 1/(r+1)>0$,
hence $e_1>r/(r+1)$. If $x_1$ is a non chain singularity, its coefficient is a rational number of the form $m/(m+1)$ for some $m$ 
by \cite[Lemma 8.3.9]{mckernan}, contradicting the inequality $(r+1)/(r+2)>e_1>r/(r+1)$. So $x_1$ is a chain singularity, and 
$A_1$ must contain four singular points by Lemma \ref{fence_two_two}. Let's proceed by cases on $r$.

\textbf{Case 1:} Suppose $r=1$. Then $1/2<e_1 < 2/3$. Since $B_2^2>0$, we necessarily have that $E_2^2 = -3$, hence $\tilde{B}_2$ is a $(-1)$ curve.
It follows by Lemma \ref{no_minus_one_curves} that $\alpha \leqslant 0$. This, on the other hand, implies that the $(2)$ point that gets contracted
by $\Sigma_2$ is a branch of the non chain singularity $x_0$. 
If $b_2<3/5$, then $x_1$ is of the form $(2,3,A_k)$ for $2\leqslant k\leqslant 4$, but this is impossible by Lemma \ref{fence_one}.
So $b_2\geqslant 3/5$, and we also get that $e_0<a_2\leqslant4/5$. This implies that $e_0$ is either $2/3$ or $3/4$.

\textbf{Case 1a:} Suppose $e_0 = 3/4$. Then by $e_0 /2 + b_2 < 1$ we get that $b_2<5/8$. By inspecting all non chain singularities of coefficient $3/4$ we may conclude
that the branches of $x_0$ are $(2)$, $(2)$ and $(2,2)$, for otherwise one of them would have been chosen in the second hunt step. 
Since $\alpha \leqslant 0$, the fourth singularity on $A_1$ has index at least six. But then it would be chosen by the second hunt step,
contradiction. 

\textbf{Case 1b:} Suppose $e_0=2/3$. Then $x_0$ can't have a branch which is either $(A_k, 4)$ or $(A_k,3,2)$, because in both cases they would be chosen
by the second hunt step. Hence $E_1 ^2 = -2$ and $x_0$ has branches $(2)$, $(2,2)$ and $(3)$. To rule out this case, some extra work is needed.
First of all notice that by Lemma \ref{fence_one}  and by the inequality $1/2<e_1<2/3$ the singularities on $B_2$ are $(2,2)$ and either
$(2,2)$ or $(2,2,2)$. These cases correspond to $B_2^2 = 1/3$ or $5/12$. 
In each case we must have $\alpha = 0$, hence the third singularity on $A_2$ is either $(2,2)$ or $(3)$. Suppose this singularity is a $(3)$ point, so that 
it is obtained by contracting $\Sigma_1$ which passes through an $(A_k,4)$ point. Since $K_{S_0}\cdot \overline{\Sigma}_1<0$ we must have that $k\geqslant1$.
If $k\geqslant 2$ we have that $A_2 ^2>2$ and if $k=1$ we have that $A_2^2 = 7/3$. In each case $A_2 ^2 \cdot B_2 ^2 \neq 1$, contradiction. One similarly rules
out the case in which the third singularity on $A_2$ is an $(2,2)$ point and it is obtained by contracting a curve passing through an $(A_k,3,2)$ point. 

\textbf{Case 2:} Suppose now that $r=2$. This implies that $e_1>2/3$ and $e_0\geqslant 3/4$. 
We divide our analysis in cases once again, depending on whether the $A_2$ point belongs to $E_1$ or not.

\textbf{Case 2a:} Suppose that the $A_2$ point is constructed by contracting $\Sigma_1$ which passes through an $(A_k,3,2)$ point. 
Then $a_1 = (3k+3)/(3k+5)$, with $k\geqslant 1$ since $a_1>2/3$. Hence $k\geqslant 2$ and $a_1\geqslant 9/11>4/5$. Now we conclude by Lemma \ref{fence_three_two}.

\textbf{Case 2b:} Suppose that the $A_2$ point belongs to $E_1$ instead. If $e_0\geqslant 4/5$ we are done again by Lemma \ref{fence_three_two}. 
Suppose that $e_0 = 3/4$. Then the two other singularities on $E_1$ must be both $A_1$ points. This implies that $\alpha >0$ and therefore $\beta<\alpha\leqslant 1/6$ 
by Lemma \ref{small_alpha}. We can finally conclude by applying Lemma \ref{fence_tactic}.

\textbf{Case 3:} As the last case, suppose that $r\geqslant 3$. In this instance, we can directly conclude by Lemma \ref{fence_three_two} since $a_2>4/5$ and $b_2>3/4$.
\end{proof}

\begin{lemma}
Suppose that $S_2 (A_2+B_2)$ is a fence and $x_1 \notin A_1$. Then $\Sigma_2$ can't meet both $A_1$ and $E_2$ at singular points.
\end{lemma}
\begin{proof}
Suppose it does. The point $x_1$ must by a non chain singularity by Lemma \ref{fence_three_two_one}.
Also, $A_1$ must have four singular points by Lemma \ref{fence_two_two}, so that $x_0$ is a non chain singularity as well.
The singular points on $B_2$ can't be both $(2)$ points, for otherwise $B_2$ would be a tiger. If $\Sigma_2$ meets $E_2$ at a point of index bigger than
two, then there is a $(2)$ point on $B_2$ and the other point has index at most five by the classification of non-chain klt singularities. 
Therefore we get a contradiction by Lemma \ref{fence_summary}.

If $\Sigma_2$ meets $E_2$ on a $(2)$ point instead, then it meets $A_1$ on an $(3,A_k)$, with $\Sigma_2$ touching the $(-3)$ curve and $A_1$
touching the $(-2)$ curve on the minimal resolution. By imposing $K_{S_1}\cdot \overline{\Sigma}_2 < 0$ and $(K_{S_1}+A_1)\cdot \overline{\Sigma}_2 > 0$,
we get that $(2k+2)/(2k+3)<e_1<(2k+4)/(2k+5)<a_2$, hence $e_1 = (2k+3)/(2k+4)$. Now we conclude by Lemma \ref{fence_three_two}.
\end{proof}

\begin{lemma}
Suppose that $S_2 (A_2+B_2)$ is a fence. Then $B_2$ is singular.
\end{lemma}
\begin{proof}
By the previous lemmas, $x_1\in A_1$. Suppose by contradiction that $B_2$ is smooth. $A_2$ contains either two or three singular points, while $B_2$ contains either one or two
singular points. We can then see by Lemma \ref{fence_three_two_one} and Lemma \ref{fence_two_two} that the only possibility is that $A_2$ contains three singular points 
and $B_2$ two. In particular this implies that $x_0$ is a non-chain singularity and therefore $e_0\geqslant 2/3$. Let $x\in A_1$ be the fourth singularity.
Suppose that $E_2$ lies over $x$. Then by explicit computation one can see that there exists an integer $n$ such that $\alpha=1/n>0$. By Lemma \ref{small_alpha} we have that
$0<\beta<\alpha\leqslant 1/6$. In particular $n\geqslant 6$ and $0<A_2^2 < 1$. It's also easy to see by an explicit computation that this implies that $A_2 ^2 \leqslant (n-1)/n$,
since different possible singularities on $A_2$ give rise to values of $A_2 ^2$ which differ by multiples of $1/n$. But then by Lemma \ref{fence_tactic} we get that
$\frac{n-1}{n(1-\beta)}\leqslant \frac{n-1}{n}$, which implies that $\beta\leqslant 0$, contradiction. Therefore $E_2$ is extracted from one of the singularities of $E_1$.

If $E_1$ contains two $(2)$ points, then we get again that $\alpha=1/n>0$ for some integer $n$ and the above argument still holds. 
The same happens if $E_1$  does not contain two $(2)$ points, but the index of singularity at $x$ is at most five. Therefore we may assume that the index of $x$ is at least six.
Also, $E_1$ must contain a singularity of index at least four which is not Du Val, for otherwise the hunt would extract $E_2$ from $x$. In particular, $e_0\geqslant 4/5$.
We see then by Lemma \ref{coefficient_boundary} that $e_1>2/3$. This contradicts Lemma \ref{fence_three_two}.
\end{proof}

\begin{lemma}\label{fence_consequence}
If $(S_2, A_2+B_2)$ is a fence. Then 
\begin{enumerate}
\item $x_1\in A_1$.
\item $B_2$ is singular and contains at most one singular point of $S_2$.
\item $\tilde{A}_2 ^2 = -1$.
\item $A_2$ contains at least one non Du Val point.
\item If $\operatorname{char}(k)\neq 2,3$ then $B_2$ has genus at least two.
\end{enumerate}
\end{lemma}
\begin{proof}
Suppose that $(S_2, A_2 + B_2)$ is a fence. By the previous lemmas, $x_1\in A_1$, $B_2$ is singular and contains at most one singular point of $S_2$. The argument
in \cite[Lemma 18.6]{mckernan} then proves that $A_2$ is a $(-1)$ curve. The argument presented there uses deformation theory, and we give here a more elementary proof
using Lemma \ref{fence_tactic}. First of all notice that since $B_2$ is singular $\beta \leqslant 0$. If $\beta = 0$ then $B_2$ is a genus one rational curve
lying in the smooth locus of $S_2$. This implies that $S_2$ is Gorenstein, and since $K_{S_2}\cdot A_2 = -B_2 \cdot A_2 = -1$, we get that $A_2$ is a $(-1)$ curve.
If $\beta \neq 0$, on the other hand, we must have $\beta \leqslant -1/2$. It's easy to see then that $\alpha>\beta$. In fact if $A_2$ contains at most two singularities,
then $\alpha>0$. If $A_2$ contains instead three singularities, then at least two of them are already present on $E_1$. This implies that their indexes are at most two and three
respectively, which yields $\alpha > -(-2 + 1/2 + 2/3 + 1) = -1/6\geqslant \beta$. By Lemma \ref{fence_tactic} and the fact that $\alpha>\beta$, we get that $A_2^2 < 1$. 
Clearly $\tilde{A}_2 ^2 \neq 0$ since otherwise we would get a net and $B_2$ would be a section, contradicting the fact that $B_2$ is singular. Therefore $\tilde{A}_2 ^2 = -1$.

Next, we notice that $A_2$ is in the Du Val locus if and only if $B_2$ has genus one and is in the smooth locus. 
In fact, if $A_2$ is in the Du Val locus, then $(K_{S_2}+B_2)\cdot A_2=0$, which implies that $(K_{S_2}+B_2)\cdot B_2 = 0$ and $B_2$ is in the smooth locus. 
Vice versa, if $B_2$ has genus one and is in the smooth locus, then $S_2$ is Gorenstein and obviously $A_2$ is in the Du Val locus.

If $B_2$ has genus one and lies in the smooth locus, however, $S_2$ is Gorenstein and, in the notation of Lemma \ref{fence_tactic}, $\beta = 0$. 
Since $B_2 ^2$ is an integer, we must have that $1/(1-\alpha)$ is an integer. This implies that $\alpha=0$ and therefore $A_2 ^2=1$. It's now easy to see that this is impossible
since $\tilde{A}_2 ^2 = -1$ and $A_2$ contains at most three singularities, just as above. This proves $(4)$.

Finally, suppose that $B_2$ has genus one. If $B_2$ has a cusp, then consider the Gorenstein log del Pezzo surface $W$ of \cite[Lemma-Definition 12.4]{mckernan}.
Since there are at least two $(-1)$ curve touching the strict transform of $B_2$, the Mordell-Weil group of the corresponding elliptic surface has at least two elements.
By Lemma \ref{cuspgorenstein} we must have that $\operatorname{char}(k)=5$ and $W=S(2A_4)$. In the notation of \cite[Lemma 12.4]{mckernan}, the image of $G$ in $W$ is a smooth
$(-1)$ curve. By the description in Example \ref{sa4char5}, the image of $G$ must pass through both the two $A_4$ singularities. Therefore $M$ must meet $G$ at a smooth point and
contract to an $A_4$ singularity. Since there cannot be any other singularities, $M$ must pass through all the singular points of $A_2$. In particular the image of $G$ will meet
the $A_4$ point created by the contraction of $M$ at the intersection of two exceptional components. This contradicts the description in Example \ref{sa4char5}, contradiction.

Therefore if $B_2$ has genus one, then $B_2$ is nodal. Also, since both $E_2$ and $B_2$ contain exactly one singular point, $\Sigma_2$ must meet $E_2$ at smooth points. By Lemma \ref{multiplicity_two_node}, $b_2 = 1/2$. However by Lemma \ref{three_four_singularities}, $e_1>1/2$, contradiction.
\end{proof}

\begin{lemma}\label{not_a_fence}
If $\operatorname{char}(k)\neq 2,3$ then $(S_2, A_2+B_2)$ is not a fence.
\end{lemma}

\begin{proof}
Suppose it is. By Lemma \ref{fence_consequence} the genus of $B_2$ is at least two. By Lemma \ref{multiplicity_two_node} and Lemma \ref{multiplicity_two_cusp} we have that $b_2<2/3$,
so that the spectral value of $\Delta_1$ is at most one by \cite[Lemma 8.0.7]{mckernan}. Therefore, by Lemma \ref{smallsvalue}, $A_1$ has Du Val or almost Du Val singularities. 
By Lemma \ref{three_four_singularities} we have that $e_1>1/2$ so that configuration $I$ of Lemma \ref{multiplicity_two_node} and Lemma \ref{multiplicity_two_cusp}
does not occur. 

\textbf{Case 1:} Suppose we have a contraction of type $II$ (recall that the nodal case and the cuspidal case are numerically the same in this configuration). 
Then we argue as in \cite[Lemma 18.6]{mckernan}. We have:

\[
 e_1<b_2<(g+1)/(2g+1)\leqslant 2/3
\]

Notice now that $x_1$ is an $A_{g+1}$ point and $\Sigma_2$ passes through the $A_g$ point of $E_2$. In fact, if $x_1$ was an $(3,A_g)$ point, then its spectral value would be at least two.
Recall there is at least one non Du Val point on $A_2$. Therefore, $A_1$ does have a point of spectral value at least one. Since we still chose the Du Val point $x_1$ in the hunt,
we must have that
\[
1/3+a_1/3 \leqslant a_1 (g+1)/(g+2)
\]

Thus $a_1\geqslant (g+2)/(2g+1)$ and $e_1\geqslant a_1 (g+1)/(g+2) = (g+1)/(2g+1)$, contradiction. 

As a consequence, we have shown that $B_2$ must have a cusp since there are no more configurations for Lemma \ref{multiplicity_two_node}. Now we go over the remaining configurations
of Lemma \ref{multiplicity_two_cusp}. 

\textbf{Case 2:} Suppose we have a contraction of type $III$.
Let $r$ be the index of a non Du Val point on $A_2$. By Lemma \ref{coefficient_boundary} and Lemma \ref{three_four_singularities} we get that $b_2 > 2/3 - 1/(3r)$. 
However $b_2\leqslant 3/5$ by Lemma \ref{multiplicity_two_cusp}, hence $r\leqslant 4$. 

\textbf{Case 2a:} Suppose $x_1$ is non Du Val, hence a $(3)$ point.
We see from the configuration and adjunction that $(K_{S_2}+B_2)\cdot B_2 = 2g - 2$ and $B_2^2 = 4g-1$.
By Lemma \ref{fence_tactic} we see that $4g-1 = (2g-1)/(1-\alpha)$, hence $\alpha = 2g / (4g-1)$. Since $\alpha>1/2$, $A_2$ contains only two singular points of $S_2$.
The idea now is to apply Lemma \ref{global_canonical} on $S_2$, since the configurations are greatly simplified by the inequality $b_2 \leqslant 3/5$ and by the shape of $\alpha$. 
If we call $x$ and $y$ the indexes of the singularities in $A_2$, we get $1/x + 1/y = \alpha$ by the definition of $\alpha$.
Let's start with $g=2$, which gives $\alpha = 4/7$. This implies that $x=2$ and $y=14$. Therefore there is an $A_{14}$ point on $A_2$, 
which of course contradicts Lemma \ref{global_canonical}. One can similarly check all cases up to $g=7$ by hand. If $g\geqslant 8$, however, one
concludes immediately by Lemma \ref{global_canonical} since $K_{S_2}^2 = (K_{S_2}\cdot B_2)^2 / (B_2) ^2 \geqslant g+1$.

\textbf{Case 2b:} Suppose then that $x_1$ is Du Val, say an $A_{j+1}$ point. Since there is a non Du Val point on $A_1$, but $x_1$ got extracted anyway,
we must have $j\geqslant 3$. Now, applying adjunction on $B_2$ we get $(K_{S_2}+B_2)\cdot B_2 = -2 + 2g + j/(j+1)$. By the description of configuration $III$ we get that
$B_2 ^2 = 4g+j/(j+1)$. But then
 \[
K_{S_2}^2 \geqslant \frac{4(1+g)^2}{4g+1}\geqslant 4
\]

By Lemma \ref{global_canonical} there must be just two singularities on $A_2$,
either two $(3)$ points or a $(2)$ point and a $(3)$ point. In each case, we get $A_2^2>2/3$. Since $B_2^2>2$, this contradicts Lemma \ref{fence_tactic}.

\textbf{Case 3:} Suppose now that we are in configuration $U$. In particular, $g=2$ and $b_2 = 9/14$. By the configuration, $x_1=(j,2,2)$. The spectral value is at most one so
$j=2$. But then $B_2^2 = 8$, again by the configuration. By Lemma \ref{fence_tactic} we also have that $B_2^2 = 3/(1-\alpha)$, hence $\alpha=5/8$. 
This implies as above that $A_2$ has two singularities, one of which with index eight. But this clearly contradicts the choice of $x_1$ in the hunt.

\textbf{Case 4:} Finally, suppose we are in configuration $V$. In particular, $g=2$ and $b_2 = 7/11$. 

\textbf{Case 4a:} Let $x_1$ be Du Val, say an $A_{j+1}$ point. Just as above in Case 2b, we get that $j\geqslant 3$. By adjunction by the description of the configuration, one gets
$(K_{S_2}+B_2)\cdot B_2 = 2+j/(j+1)$ and $B_2^2 = 9+j/(j+1)$. But then $\alpha = (6j+6)/(10j+9)$. Since $\alpha>1/2$ one immediately sees that $A_2$ has only two singularities on it.
It is easy to see however that the equation $1/x+1/y = (6j+6)/(10j+9)$ does not have a solutions in positive integers.

\textbf{Case 4b:} Suppose that $x_1$ is a $(3)$ point. The same computation as in Case 3 gives $\alpha = 5/8$. One the concludes exactly as in Case 3.

\textbf{Case 4c:} Suppose that $x_1$ is almost Du Val, say of type $(3,A_j)$ with $j\geqslant 1$.
Since the spectral value of $\Delta_1$ is at most one, $A_1$ meets the last $(-2)$ curve.
 Again by adjunction we get that $(K_{S_2}+B_2)\cdot B_2 = 2+2j/(2j+1)$. Also, $B_2^2 = 9+ (2j-1)/(2j+1)$, so that
$\alpha = (12j+5)/(20j+8)$. The equation $1/x+1/y=(12j+5)/(20j+8)$ has no solutions in positive integers.
\end{proof}

\begin{lemma}\label{banana_two_one}
Let $S(X+Y)$ be a smooth banana. Then $X$ and $Y$ can't have contain exactly two and one singular points respectively.
\end{lemma}
\begin{proof}
Suppose they do. Let $f:T\rightarrow S$ extract the adjacent divisor $E$ to $Y$. We get a contraction $\pi$ from Lemma 
\ref{hunt_transformation}. 

\textbf{Case 1:} Suppose $T$ is a net. If $Y$ is a fiber, then $E$ is a section with at most one singular point, hence there
is exactly one singular fiber. However $X$ has two singular points, contradicting the fact that singular fibers contain at most two singular point. 
If $Y$ is a multisection, then $Y\cdot F\geqslant 2$ because $T$ is not smooth. 
But then $(K_S+Y)\cdot f_* F\geqslant 0$. However $K_S+Y$ is negative by adjunction, contradiction. 

\textbf{Case 2:} Suppose $\pi:T\rightarrow S_1$ is a birational contraction. Clearly $\pi$ contracts a curve $\Sigma$ that doesn't touch $Y$. 
If $S_1$ is not smooth, then $Y^2 \geqslant 2$ by Lemma \ref{smoothp1}. On $S$ we have that
$X^2 > Y^2$ by adjunction (and the analogue of Lemma \ref{fence_tactic}). Hence $Y^2 < 2$, contradiction. 
On the other hand, if $S_1$ is smooth, then it's $\mathbb{P}^2$. Since $E\cdot Y=1$, they
must both be lines and $X$ is a smooth conic. However $\Sigma$ must pass through both it's singularities, hence $X$ has a node, contradiction.
\end{proof}

\begin{proof}[Proof of \ref{must_be_banana}]
By Lemma \ref{not_a_fence}, $S_2(A_2+B_2)$ is not a fence. Lemmas \cite[Lemma 18.7 - 18.8]{mckernan} show 
that $A_1$ can't be contracted and $T_2$ is not a net. The only option left in Proposition \ref{hunt} is that $(S_2, A_2+B_2)$ is a smooth banana.
By Lemma \ref{banana_two_one} we may also deduce that $A_1$ contains exactly three singularities.
\end{proof}

Thanks to Lemma \ref{banana_two_one} $S_2$ is one of the smooth bananas described in \cite[Lemma 13.2]{mckernan} and
we can get the classification of \cite[Chapter 19]{mckernan}.
\section{Log del Pezzo surfaces with tigers}\label{sectiontigers}

In this section we classify all log canonical pairs $(S,C)$ such that $S$ is a rank one log del Pezzo surface, $C$ is a reduced integral Weil divisor and $K_S + C$ is anti-nef. 
This classification will allow us to complete the classification of log del Pezzo surfaces of rank one in the next section.

\subsection{$(S,C)$ is divisorially log terminal}

In this subsection we furthermore assume that $K_S+C$ is dlt. In the following, we shall slightly change the scaling convention of Lemma 
\ref{hunt_transformation} to avoid getting coefficients larger than one.

\textbf{Case 1:} If $C$ contains at least three singular points, we run the hunt for $(K_S,aC)$ with $1-\epsilon < a < 1$ and we don't rescale by multiplying by $\lambda$.
Instead, we write $K_{T_1}+\Gamma_1 = f^*(K_S+aC)$, $\Gamma_1' = \Gamma + b_1 E_1$ with $R\cdot \Gamma_1' = 0$, and $\Delta_1 = (\pi_1)_* (\Gamma_1')$.
It's easy to see that if $\pi_1$ is a birational contraction, then $b_1<a$ and the statement of Proposition \ref{hunt} remains true with these conventions.
In fact, multiplication by $\lambda$ was only needed to ensure flushness with respect to the divisor $aC$ in $\Gamma_1'$, but in our setting
this is immediate since $a>1-\epsilon$ (compare with the claim on page 141 of \cite{mckernan}).
Furthermore, since we may pass to the limit $a\rightarrow 1$ by continuity, we may even assume $a=1$.

\textbf{Case 2:} Suppose instead that $C$ contains at most two singular points. Then we run the hunt in the level case for $(S,C)$. Notice that the discussion in Subsection \ref{levelhunt}
shows that $(S_1, C_1+A_1)$ is dlt since $(S,C)$ is dlt.

We start our analysis by extending \cite[Proposition 23.5]{mckernan}.

\begin{proposition}\label{log_terminal_three}
Suppose $\operatorname{char}(k)\neq 2,3$ and that $C$ contains at least three singular points.

If $-(K_S+C)$ is ample then:
\begin{enumerate}

\item If $\operatorname{char}(k)\neq 5$ then $S\setminus C$ has exactly one singular point, a non cyclic singularity, z. If $Z\rightarrow S$ extracts the central divisor $E$ of z,
then $Z$ is a $\mathbb{P}^1$ fibration and $E$ and $C$ are sections. The pair $(S,C)$ is uniquely determined by z, and all non cyclic
singularities z occur in this way for some pair $(S,C)$. These surfaces are classified in \cite[Lemma 23.5.1.1]{mckernan}.

\item If $\operatorname{char}(k)=5$, in addition to the above, $(S,C)$ could also be obtained by resolving the singularity of the unique cuspidal rational curve in the smooth locus of $S(2A_4)$
and taking $C$ to be the last $(-1)$ curve. This is Example \ref{char5}.

\end{enumerate}

If $K_S+C$ is numerically trivial then $S$ is Du Val, and $(S,C)$ is one of five families:
\begin{enumerate}[resume]
\item $S=S(A_1+2A_3)$. $(S,C)$ is given by \cite[Lemma 19.2]{mckernan}, Case 1, $s=3$ and $r=1$.
\item $S=S(3A_2)$. $(S,C)$ is given by \cite[Lemma 19.2]{mckernan}, Case 2, $s=2$ and $r=1$.
\item $K_S^2 = 1$, $C$ is a $(-1)$ curve and $S$ is one of $S(A_1+A_2+A_5)$, $S(2A_1+A_3)$, $S(4A_1)$.
The pairs are obtained from \cite[Lemma 13.5]{mckernan} by blowing up the node of the nodal curve always along the same branch.
\end{enumerate}
\end{proposition}

\begin{proof}
The proof of \cite[Proposition 23.5]{mckernan} goes through with only minor modifications, which we shall point out. 

Notice that $e_1\geqslant 1/2$ by Lemma \ref{coefficient_boundary} since $a=1$ and since there is a singular point on $C$. 
Furthermore, $e_1\geqslant 2/3$ unless all the singularities on $C$ are $A_1$ points.
If follows that $2a+b_1>2$. Hence we can't have a tacnode or a triple point by Lemma \ref{flushmult}.

\textbf{Case 1:} Suppose $x\notin C$. Notice that $\Sigma_1$ passes through a singular point of $C$, for otherwise $(K_{T_1} + C+b_1E_1)\cdot \Sigma_1 > 0$. 

\textbf{Case 1a:} Suppose we have a birational contraction to $S_1$ (notice that our convention has a different indexing with respect to \cite{mckernan}). 
This must be a smooth fence by Proposition \ref{hunt}.
If $C$ contains four singular points, then its corresponding branch in the fence contains three $(2)$ points. The branch corresponding to $E_1$ can't contain
two singular points by Lemma \ref{small_alpha} or one singular point by Lemma \ref{fence_three_two_one}. But then $x$ is a non chain singularity,
and its coefficient is $1/2$ since $(K_{T_1}+C+b_1E_1)\cdot \Sigma_1 = 0$. From this we deduce that $\tilde{E}_1 ^2 =-2$.
The description of the configuration then tells us that the corresponding branch of the fence will be a zero curve. Extracting its singularities we get a net on which $C$ is a section.
Also, there is a smooth section corresponding to a $(-2)$ branch of $x$. Now one proceeds as in the proof of Lemma \ref{fence_net}: there is an irreducible multiple fiber
since there are three singularities on $C$, and this contradicts the existence of a smooth section.

Suppose then that $C$ contains exactly three singular points.  Notice that $\Sigma_1$ can't meet $E_1$ at a singular point because of Lemma \ref{fence_two_two} and 
Lemma \ref{fence_three_two_one}. From this argument it also follows that $x$ is a non chain singularity. We want to apply the results of Section \ref{smooth_fences} with
$X=A_1$ and $Y=C_1$.
Now, since $K_{S_1}+C_1+b_1 A_1$ is anti-nef, in the notation of Lemma \ref{fence_tactic} we get that $b_1\leqslant -(K_{S_1}+C_1)\cdot C_1=\beta$. 
Since $\beta\geqslant b_1\geqslant 1/2\geqslant \alpha>0$ we get a contradiction by Lemma \ref{small_alpha}.

\textbf{Case 1b:} Suppose instead we get a net. Then we get description (1) by the argument in \cite[Proposition 23.5]{mckernan}.

\textbf{Case 2:} Suppose now that $x\in C$. If we get a banana in the hunt, then we get descriptions $(3)$ and $(4)$ by using Lemma \ref{banana_two_one} and 
\cite[Lemma 13.2]{mckernan}.

We can't get a smooth fence in the hunt because of Lemma \ref{fence_two_two} and Lemma \ref{fence_three_two_one}. Hence if we get a fence, it will
be a singular one. Notice that $C_1$ is smooth and $\tilde{C}_1 ^2 = -1$ by the proof of Lemma \ref{fence_consequence} (3). $C_1$ is contained in the Du Val locus if and only if $A_1$ is 
contained in the smooth locus by adjunction. Now we can again use the argument in \cite[Proposition 23.5]{mckernan} to get description $(5)$. A little bit of care is needed in applying
\cite[Proposition 13.4]{mckernan}: if $\operatorname{char}(k)\neq 5$ then one uses Lemma \ref{cuspgorenstein} as usual to replace arguments about triviality of the fundamental group;
if $\operatorname{char}(k)=5$ one has to further discard the case in which $W=S(2A_4)$.
To do so, notice that $C_1$ is a $(-1)$ curve which contains two singular points and such that $K_{S_1}+C_1$ is dlt. This contradicts the explicit description in Example \ref{sa4char5}.

If we get a net, then the proof given \cite[Proposition 23.5]{mckernan} carries through with no modifications. In the case in which $C$ gets contracted, however, there is a difference
if $\operatorname{char}(k)=5$. In fact, if $A_1$ is in the smooth locus, then $S_1$ can also be $S(2A_4)$, leading to description $(2)$. All the rest proceeds in the same way as in
\cite{mckernan}.
\end{proof}

\begin{proposition}\label{log_terminal_two}
Assume $C$ contains exactly two singular points. Then $(\tilde{S},\tilde{C})$ is one of the following
\begin{enumerate}
\item Start with $\mathbb{F}_n$, pick the negative section $E$ and a positive section $C$ disjoint from $E$. Blow up along $C$ or $E$ to create two multiple fibers.
Then keep blowing up while keeping $C$ and $E$ disjoint, $E$ $K$-positive and the $K$-positive curves contractible to klt singularities. 

\item Start with a smooth fence containing one singular point in each branch. These are classified in Lemma \ref{fence_one_one}.
Blow up $X\cap Y$ and then keep blowing up at the intersection of the last exceptional divisor with either branch of the fence, while keeping the $K$-positive curves
contractible to klt singularities. $\tilde{C}$ will be the strict transform of $X$.

\item Stat with $\mathbb{F}_n$, pick the negative section $E$ and a positive section $C$ meeting $E$ exactly once. Blow up along $C$ or $E$ to create one multiple fiber,
while keeping $E$ $K$-positive and the $K$-positive curves contractible to klt singularities.

\item Start with either $\mathbb{F}_2$ or $\mathbb{P}^2$ and pick a banana. Blow up repeatedly above one intersection point to make one of the branches $K$-positive,
while keeping the other branch $K$-negative and the $K$-positive curves contractible to klt singularities. The $K$-negative branch is $\tilde{C}$.

\item Start with either a fence as described in Lemma \ref{fence_one_one}, or by the fence in $\mathbb{F}_n$ given by a fiber $C$ and a positive section $E$.
Then blow up repeatedly over $E$ in order to make $E$ contractile to a klt chain singularity.

\item Start with a dlt pair $(S',C')$, where $C'$ is a smooth rational curve containing either one or two singular points. Blow up one of the singular points (or a smooth point if $C'$
has only one singularity) always along $C'$ to make $C'$ $K$-positive and contractible to a klt singularity. $\tilde{C}$ is the last $(-1)$ curve.
\end{enumerate}
\end{proposition}
\begin{proof}
We run the hunt in the level case for $(S,C)$ and analyze all possible cases.

\textbf{Case 1:} Suppose $x\notin C$. 

\textbf{Case 1a:} Suppose furthermore that $T_1$ is a net. Clearly $C$ is not contractible on $S$ and therefore it's not contractible on $T_1$ either since $x\notin C$.
On the other hand, $E_1$ has negative self intersection because it's exceptional over $S$. Therefore $C$ and $E_1$ are are both multi-sections on $T_1$.
Let $F$ be a general fiber of the net and recall that $(K_{T_1}+C+b_1 E_1)\cdot F = 0$ by definition.
Since $b_1 \geqslant 1/2$, $C$ and $E_1$ are both sections. Running a relative MMP for the morphism $\tilde{T}_1 \rightarrow \mathbb{P}^1$ one gets a Hirzebruch surface $M=\mathbb{F}_n$. 
On $M$ the curves $C$ and $E_1$ are disjoint, $C$ is a section of positive self intersection, and $E_1$ is the section of negative self intersection. This process may be reversed
to obtain $\tilde{S}$. More precisely, one may start from $\mathbb{F}_n$, pick the negative section $E_1$ and a disjoint positive section $C$, and then
blow up either on $E_1$ or $C$ so that to create disjoint two multiple fibers. This is case (1).

\textbf{Case 1b:} If we get a birational contraction instead, then we must have a fence. 
By using Lemma \ref{fence_three_two_one} and Lemma \ref{fence_two_two} we see that
$x$ must be a chain singularity, that $\Sigma_1$ meets $C$ at a singular point, and $E_1$ either at a smooth point 
(if $E_1$ has just one singular point on it) or at a singular point (if $E_1$ has two singular points on it).
Smooth fences with one singular point on each branch are classified in Lemma \ref{fence_one_one}. Then one can obtain $\tilde{S}$ by blowing up
at the intersection of the branches of the fence. This is case (2).

\textbf{Case 2:} Suppose now that $x\in C$. 

\textbf{Case 2a:} Suppose $T_1$ is a net. $C$ can't be a fiber because it's dlt and has only one singularity.
Hence $C$ and $E_1$ are both sections as above, since $(K_{T_1}+C+b_1 E_1)\cdot F=0$. Since $C$ contains only one singular point, there is exactly one multiple fiber. Therefore $E_1$
must also contain exactly one singular point on that multiple fiber. Running a relative MMP over the base we go to $\mathbb{F}_n$. $E_1$ must be the negative section of 
$\mathbb{F}_n$. $C$ is a positive section meeting $E_1$ exactly once in $\mathbb{F}_n$. This gives case $(3)$.

\textbf{Case 2b:} Suppose then that we get a birational contraction down to $S_1$. 
If we go to a banana, then $\Sigma_1$ must meet $C$. Since $(K_{T_1}+C+b_1 E_1)\cdot \Sigma_1=0$, $\Sigma_1$ meets $C$ at a singular point and therefore 
$C_1$ is in the smooth locus of $S_1$. 
By adjunction, $A_1$ also is in the smooth locus. Since $C_1\cdot A_1  =2$ we must have $S_1=\mathbb{F}_2$ or $\mathbb{P}^2$ by Lemma \ref{smoothp1}. To get back to $S$ it suffices
to reverse the process: pick a smooth banana in $\mathbb{F}_2$ or $\mathbb{P}^2$, 
blow up repeatedly one of the intersection points, then contract one of the two branches of the banana. This gives case (4).

If we go to a smooth fence, we have at most one singularity on the branch corresponding to $E_1$ by Lemma \ref{fence_three_two_one}. This is case (5).

Suppose we go to a singular fence. If $A_1$ is in the smooth locus of $S_1$ then $S_1$ is Gorenstein. But then $C_1$ is contracible, since it contains only one singular point, contradiction.
If instead $A_1$ contains a singular point on $S_1$, we must have that $b_1=1/2$ by $(K_{T_1}+(f_0)_* ^{-1} C+b_1 E_1)\cdot \Sigma_1=0$. 
This however implies that in $S$ the singularities on $C$ are
both $A_1$ points, contradiction.

Finally, if $C$ gets contracted we can induct on $(S_1,A_1)$, giving description (6).
\end{proof}

\begin{proposition}\label{log_terminal_one}
Assume $C$ contains exactly one singular point. Then $(\tilde{S},\tilde{C})$ is one of the following:
\begin{enumerate}
\item Start with $\mathbb{F}_n$. Pick the negative section $E$ and a positive section $C$, which is disjoint from $E$. Blow up repeatedly points on a fixed fiber so that $C$ remains
a positive curve, there are $K$-positive curves lying both over $E$ and $C$, the $K$-positive curves are contractible to klt chain singularities and $C$ and $E$ are disjoint.

\item Start with $\mathbb{P}^2$. Pick two lines $C$ and $E$ and blow up at the intersection three times along $E$. Then continue blowing up 
while keeping the $K$-positive curves contractible to klt chain singularities and $C$ and $E$ disjoint.

\item Start with $\mathbb{F}_n$. Pick a positive section $C$ and a fiber $F$. Blow up twice at $F\cap C$ along $F$. Then continue
blowing up while keeping the $K$-positive curves contractible to klt chain singularities and $C$ and $F$ disjoint.

\item ($\mathbb{F}_n, C$), where $n\geqslant 2$ and $C$ is any positive section meeting the negative section exactly once. 

\item Start with $\mathbb{F}_n$. Pick the negative section $E$, a point $p$ and a fiber $C$ not passing through $p$. 
Blow up at $p$ to obtain at most one multiple fiber, while keeping the $K$-positive curves contractible to klt chain singularities.

\item Start with $\mathbb{P}^2$. Pick two lines $C$ and $E$. Pick a point $p\in E\setminus C$, blow up $p$ along $E$ until $E$ becomes
$K$-positive, then keep blowing up at either end of the last exceptional divisor, while keeping the $K$-positive curves contractible to
klt chain singularities.

\item Start with $\mathbb{F}_n$. Pick the negative section $E$, a fiber $F$ and a positive section $C$. 
Blow up at a point in $F\setminus (C \cup E)$ to make $F$ contractible. Continue blowing up while keeping the $K$-positive curves contractible to klt chain singularities.
\end{enumerate}
\end{proposition}
\begin{proof}
We run the hunt in the level case for $(S,C)$ and analyze all possible cases.

\textbf{Case 1:} Suppose $x\notin C$.

\textbf{Case 1a:} Suppose that $T_1$ is a net. With usual arguments as in Lemma \ref{log_terminal_three} and Lemma \ref{log_terminal_two}, we see that $C$ and $E_1$ are both sections.
Since there is one singular point on $C$, there must be one singular point on $E_1$ as well, and they must lie on the same multiple fiber. By running a relative MMP we get 
$\mathbb{F}_n$. $C$ must be a positive section and $E_1$ must be the negative section. This is case $(1)$.

\textbf{Case 1b:} If we get a birational contraction, we must go to a fence. Since $a=1$, $\Sigma_1$ must pass
through the only singular point on $C$. Therefore the strict transform of $C$ is a smooth rational curve in the smooth locus of $S_1$. 
By Lemma \ref{fence_three_two_one}, $x$ is a chain singularity. Therefore $E_1$ contains either one or two singular points of $T_1$.
In the first case, $S_1$ is $\mathbb{P}^2$, and $C_1$ and $A_1$ are lines in $S_1$. This is case $(2)$. 
If $E$ contains two singularities instead, we get case $(3)$ by Lemma \ref{smoothp1}.

\textbf{Case 2:} Suppose now that $x\in C$. 

\textbf{Case 2a:} Suppose $T_1$ is a net. If $C$ is a section, then $T_1$ is smooth. Therefore $T_1=\mathbb{F}_n$ for some $n$, and 
$S$ is just $\overline{\mathbb{F}}_n$. This is case $(4)$.
Suppose then that $C$ is a fiber. Clearly this implies that $E_1$ is a section. Running a relative MMP one gets case (5).

\textbf{Case 2b:} Suppose we get a birational contraction. We can't get a banana because $C$ has only one singular point of $S$ and $a+b_1>1$. 
We can't get a singular fence either, for otherwise $C$ would be a $(-1)$ curve, hence contractible on $S$. 
Notice furthermore that $C$ can't get contracted on $T_1$ for otherwise $C^2 < 0$ in $S$.

Therefore we get a smooth fence. If $A_1$ is in the smooth locus of $S_1$, then $S_1$ is $\mathbb{P}^2$, $C_1$ and $A_1$ are lines,
and we obtain $T_1$ by blowing up a point of $A_1$ multiple times. This is case (6). If $A_1$ is not in the smooth locus of $S_1$ instead, we get case (7).
\end{proof}

We conclude with an almost trivial case.

\begin{lemma}\label{log_terminal_zero}
Assume that $C$ is in the smooth locus of $S$. Then $(\tilde{S}, \tilde{C})$ is one of the following:
\begin{enumerate}
\item $(\mathbb{P}^2, C)$, where $C$ is a line.

\item $(\mathbb{P}_2, C)$, where $C$ is a smooth conic.

\item $(\mathbb{F}_n, C)$, where $C$ is a positive section not meeting the negative section.
\end{enumerate}
\end{lemma}
\begin{proof}
This follows immediately from Lemma \ref{smoothp1}.
\end{proof}

\subsection{$(S,C)$ is log canonical but not divisorially log terminal}

In this subsection we classify pairs $(S,C)$ of a log del Pezzo surface $S$ and a curve $C$ such that $K_S+C$ is anti-nef and log canonical, but not divisorially log terminal.
Throughout this subsection we run the hunt in the level case as in Subsection \ref{levelhunt}.

\begin{lemma}\label{log_canonical_three}
Let $(S,C)$ be a pair of a rank one log del Pezzo $S$ and an irreducible curve $C$ such that $K_S+C$ is anti-nef, log canonical but not dlt.
Suppose also that $C$ contains three singular points of $S$. Then $(\tilde{S},\tilde{C})$ is described as follows. Start with $\mathbb{F}_n$, pick the negative section $E$
and a positive section $C$ such that $C\cap E$ contains at most one point. Now create two dlt multiple fibers of multiplicity two. 
If $C\cap E$ contains exactly one point stop. If $C\cap E=\emptyset$, create a third multiple fiber such that it contains exactly one chain singularity.
\end{lemma}
\begin{proof}
Two of the singular points on $C$ must be $A_1$ points, at which $K_S+C$ is dlt by Lemma \ref{negative_adjunction_2}. For the last point, $p$,
there are two cases by Lemma \ref{lcnotdltsing}:

\begin{enumerate}

\item $p$ is a non chain singularity, with two $(2)$ branches. $C$ meets the opposite end of the third branch, or 
\item $p$ is a $(2,n,2)$ point, and $\tilde{C}$ meets the central curve.

\end{enumerate}

\textbf{Case 1:} Let the hunt extract the central curve of the singularity at $p$.

\textbf{Case 1a:} Suppose we get a net. Clearly $C$ can't be a fiber since it contains three singular points.
Hence $C$ and $E_1$ are multisections, and they are both sections since $K_{T_1}+C+E_1$ is numerically trivial. There are two multiple fibers
with multiplicity two passing through the $(2)$ points. Finally there is a third multiple fiber passing through $C\cap E_1$.
This is precisely one of the two situations described in the statement.

\textbf{Case 1b:} Suppose we get a birational contraction. Clearly $C$ does not get contracted
as it contains three singular points. $K_{S_1}+C+A_1$ is log canonical at singular points and doesn't have triple points because it's level.
$\Sigma_1$ meets $C$ and $E_1$ only at singular points, hence the only possibility is that it passes through $C\cap E_1$, by contractibility 
considerations. If we get a smooth fence, then we are done by Lemma \ref{fence_two_two}. If $C_1 \cap A_1$ is a singular point instead, the second hunt step
extracts one its exceptional components. $\pi_2$ cannot be a birational morphism since neither $C_1$ nor $A_1$ may be contracted, and all
the components of $\Gamma_2$ have coefficient one. Therefore $\pi_2$ must be a net, which implies that the strict transforms of $C_1$ and $A_1$ are fibers.
Therefore $E_2$ is a double section. Since $E_2 ^2\leqslant -2$, we see that if we run an MMP over the base, the strict transform of $E_2$ will have 
self-intersection at most two, contradiction.

\textbf{Case 2:} This case may be thought of as a \say{degenerate} version of Case 1. The first hunt step extracts the middle curve above $p$.
It's easy to see that the contraction of $\Sigma_1$ does not yield a birational contraction. In fact, there are only $A_1$ points lying on $C$ and $E_1$, and there is no configuration
that allows a birational contraction. Suppose therefore that $T_1$ is a net. If $C$ is a fiber, then it has multiplicity two and $E_1$
is a double section, which is impossible since $\tilde{E}_1 ^2\leqslant -2$ (see also \cite[Chapter V, Proposition 2.20]{hartshorneag}). 
Therefore $C$ and $E_1$ are both sections, meeting at a smooth point. This situation is described in the statement.
\end{proof}

\begin{lemma}\label{log_canonical_two_lt}
Suppose $(S,C)$ is a pair of a rank one log del Pezzo $S$ and a curve $C$ such that $K_S+C$ is anti-nef, log canonical but not divisorially log terminal.
If $C$ passes through two singular points, and $(S,C)$ is dlt at one of them, say $p$, then $(\tilde{S},\tilde{C})$ is one of the following:
\begin{enumerate}
\item Pick a pair $(S,E)$ such that $E$ passes through a $(2,n,2)$ point, and $E$ is lc but not dlt at this point (in other words $E$ crosses transversally the $(-n)$ curve).
Suppose that $E$ contains at most one other singular point, and $E$ is dlt along it. These are described in parts $(2)$ and $(3)$ of the statement.
Blow up above the dlt point (or at a smooth point if $E$ has just one singularity), always along $E$ so that it becomes contractible to a klt singularity. $\tilde{C}$ is the final $(-1)$ curve.

\item Start with a dlt pair $(S,E)$ such that $E$ contains two $A_1$ points, and possibly a third singular point. If there are just two singular points,
blow up above a smooth point of $E$ always along $E$, otherwise blow up above the third point, always along $E$. 
$\tilde{C}$ is the last $(-1)$ curve.

\item Start with a fence as described in Lemma \ref{fence_one_one} with a branch, $E$, containing an $A_1$ point. $\tilde{C}$ is the other branch.
Now blow up a smooth point on $E$ to create another $A_1$ point and make $E$ contractible. 
\end{enumerate}
\end{lemma}
\begin{proof}
The configurations for the non dlt point are as in Lemma \ref{log_canonical_three}. We continue to use the same
division in cases.

\textbf{Case 1:} This time we choose to extract the adjacent curve $E_1$ of the non chain singularity on $C$.

\textbf{Case 1a:} Suppose $T_1$ is a net. $C$ and $E_1$ can't be sections since they don't touch on the multiple fiber, but one of the singularities is not dlt for this fiber.
$C$ can't be a fiber either because it's dlt and has only one singular point on it, contradiction.

\textbf{Case 1b:} Hence we get a birational contraction. If $C$ gets contracted, then we are in the same situation as in the hypothesis of the lemma, but with
a log canonical point with a smaller Dynkin diagram, therefore we can apply induction. This is description (1).
If $\Sigma_1$ passes through both the singular points on $C$ and $E_1$, then we would get two smooth
rational curves meeting at a smooth point and at a chain singularity at the opposite ends of it. This, however,
does not happen because of contrability reasons. The only option left is that $\Sigma_1$ meets one of the two $(2)$ ends
of the non chain singularity in $E_1$, and therefore $S_1$ is a smooth fence. However then $E_1$ would have positive self-intersection in $T_1$, contradiction.

\textbf{Case 2:} As above one may show that $T_1$ is not a net. If $C$ gets contracted, 
$E$ is dlt with two $A_1$ points and possibly another dlt point on it. 
We can then we can use Lemma \ref{log_terminal_three} and Lemma \ref{log_terminal_two} to classify such configurations and
get back to the original pair $(S,C)$. This is description (2).

If $\Sigma_1$ meets both $C$ and $E_1$ at singular points, we must go to a banana. It is easy to see however this cannot happen since the singular points on $E$ are $A_1$ points.
The only remaining option is that $\Sigma_1$ passes through an $A_1$ point of $E_1$ and another singular point on $T_1$.
Since we then go to a fence, we see that $\Sigma_1$ contracts to a smooth point on $A_1$ by Lemma \ref{fence_three_two_one}.
We get therefore a fence with one singular point on each branch, one of which is an $A_1$ point. It is possible to get then $S$ by 
starting with such a fence, as classified in Lemma \ref{fence_one_one}, and then blowing up a point in along $E_1$ multiple times, and finally blowing
up once away, to create $\Sigma_1$. Finally one can contract $E_1$. This is case (3).
\end{proof}

\begin{lemma}\label{log_canonical_two_lc}
Suppose $(S,C)$ is a pair of a rank one log del Pezzo $S$ and a curve $C$ such that $K_S+C$ is anti-nef, log canonical but not divisorially log terminal.
If $C$ passes through two singularities, and is not dlt at either, then $(\tilde{S},\tilde{C})$ is one of the following:
\begin{enumerate}
\item Start with $\mathbb{F}_n$, pick a double section $E$ and let the two fibers tangent to $E$ be $F$ and $G$.
Blow up on the intersection $F\cap E$ any number of times, always along $E$. $C$ is the last $(-1)$ curve.
Now blow up at $E\cap G$ twice to separate them and call $H$ the 
$(-1)$ curve. Finally blow up on $H$ at a point not contained in any of the other components, always along $H$. 

\item Start with $\mathbb{F}_n$, pick a double section $E$, a tangent fiber $F$ and a transverse fiber $G$. Blow up on the intersection
$F\cap E$ any number of times along $E$; $\tilde{C}$ is the last $(-1)$ curve. Then blow on one of the points in $G\cap E$, and once at the intersection of the two $(-1)$ curves. 

\item Start from $\mathbb{F}_n$, pick a positive section $E$ and a negative section $F$. Blow up to create two dlt double fibers.
Now blow up any point on $E$ always along $E$, and if the point was singular then make a last blow up away from $E$ to create an $A_1$ point on 
$E$. $\tilde{C}$ is given by the choice of any smooth fiber.
\end{enumerate}
\end{lemma}
\begin{proof}
Let $p$ and $q$ be the two singularities. There are three cases.

\textbf{Case 1:} Suppose first that both of them are non chain singularities and extract the adjacent divisor $E_1$ above $q$.

\textbf{Case 1a:} Suppose $T_1$ is a net. $C$ and $E_1$ can't be both
sections because otherwise they would meet on the singular fiber. 
If $C$ is a fiber instead, it must have multiplicity two because for a general fiber $F$ we have $(K_{T_1}+C+E_1)\cdot F= 0$, so that $F\cdot E_1 = 2 = 2 C\cdot E_1$.
This is description (1).

\textbf{Case 1b:}
Suppose we get a birational contraction instead. 
If $C$ gets contracted, then it contracts to a smooth point since $K_{S_1}+A_1$ is anti-nef and log canonical. This is however not possible since $C$ contains a non chain singular point.
Thus $\Sigma_1 \neq C$. Therefore $\Sigma_1$ can meet $C$ and $E_1$ only at singular points. If it meets both curves 
we must go to a smooth banana since $K_{S_1}+C+A_1$ is anti-nef and log canonical. It's easy to see that this is again impossible since $C$ contains a non chain singular point.
Therefore $\Sigma_1$ only meets $E_1$ and we go to a smooth fence. By Lemma \ref{fence_tactic} we must have a non chain singularity on each branch.
By symmetry $C_1 ^2=A_1 ^2=1$, hence after extracting the adjacent curve $D$ of $p$ at $C_1$, we get that $C_1 ^2=0$, $E_1$ and $D$ are sections.
But then we would have a multiple fiber with two non dlt sections, which is impossible. 

\textbf{Case 2:} Suppose $p$ is a non chain singularity, and $q$ is a $(2,n,2)$ point. Extract the $(-n)$ curve $E_1$.
If $T_1$ is a net, then $C$ must be a double fiber and $E_1$ a double section. This is description (2). If we get a birational contraction,
we must go to a fence, where $C$ has a non chain point and $A_1$ has two $(2)$ points. After extracting the adjacent divisor $D$ of the non chain singularity
we would have a net, with $A_1$ and $D$ being a sections, which is impossible.

\textbf{Case 3:}
Finally, suppose $p$ is a $(2,n,2)$ point and $q$ is a $(2,m,2)$ point. Extract the $(-m)$ curve $E$. If $T$ is a net, we get again description $(2)$.
Suppose therefore that we have a birational contraction to $S_1$. Using the usual arguments we deduce that $C_1+A_1$ is a fence. Now, after
extracting the $(-n)$ curve $D$ from $C_1$, we clearly get a net where $C_1$ is a fiber, $D$ and $E_1$ are sections. This is description (3).
\end{proof}

Before studying the case of just one singular point, we need a preliminary lemma.

\begin{lemma}\label{singular_banana}
Let $S$ be a rank one log del Pezzo surface and let $A$ and $B$ be two smooth rational curves in $S$. Suppose that $A$ and $B$ meet at two distinct points
$p$ and $q$, with $q$ smooth and $p$ singular. Suppose also that $(S, A+B)$ is log canonical. Then $(S, A, B)$ is one of the following:
\begin{enumerate}
\item Start from $\mathbb{P}^2$. Choose three not concurring lines $A$, $B$ and $E$. Blow up above a point in $E\setminus (A\cup B)$ along $E$ so that $E$
becomes $K$-positive. Then contract the $K$-positive curves. 
\item Start from $\mathbb{F}_n$. Choose two fibers $B$ and $E$ and a positive section $A$. Blow up above a point in $E$ that is not contained in $A$, $B$ 
or the negative section. Then contract the $K$-positive curves. 
\end{enumerate}

In particular, the resolution of $p$ consists of at most two curves.
\end{lemma}
\begin{proof}
Without loss of generality, the first hunt step extracts the adjacent divisor to $A$. Clearly $T_1$ cannot be a net because $\Gamma_1$ consists of 
three components of coefficient one. Therefore $\pi_2$ is a birational contraction. Neither $A$ nor $B$ can be contracted, for otherwise $E$ will have negative self-intersection
on $S_1$. Therefore $K_{S_1}+A_1+B_1+C_1$ is numerically trivial and log canonical. Since $A_1$ is in the smooth locus,
we have that $S_1$ is either $\mathbb{F}_n$ or $\mathbb{P}^2$. The lemma then follows. 
\end{proof}

\begin{lemma}\label{log_canonical_one}
Suppose $(S,C)$ is a pair of a rank one log del Pezzo $S$ and a curve $C$ such that $K_S+C$ is anti-nef, log canonical but not divisorially log terminal.
If $C$ passes through only one singular point $p$ of $S$ then $(\tilde{S},\tilde{C})$ is obtained as follows. 
\begin{enumerate}

\item Start with $\mathbb{F}_n$, pick the negative section $E$, and a positive section $C$ meeting $E$ just once and transversally. 
Pick a point in $(C\cup E)\setminus (C\cap E)$ and blow up any number of times to make a divisorially log terminal fiber. Then perform a blow up on an interior
point of the last $(-1)$ curve so that it creates a chain singularity connecting $C$ and $E$, and finally continue blowing up on the $(-1)$ curve
at the nearest point to $C$ and $E$.

\item Start with $\mathbb{F}_n$, pick the negative section $E$, and a positive section $C$ meeting $E$ exactly twice and transversally. 
Blow up at one of intersections once, then blow up at the intersection of the exceptional component and the strict transform of the fiber. Now keep blowing up
at the intersection of the last $(-1)$ curve with the exceptional component joining $C$ and $E$ any number of times.

\item Start with a log canonical pair $(S,E)$, where $S$ is a rank one log del Pezzo surface, $K_S+E\equiv 0$, $E$ has exactly one node, which is a singular point of $S$ as well.
Blow up repeatedly at one of the branches of the node, so that $E$ becomes $K$-positive, and the $K$-positive curves are contractible to klt singularities.
$\tilde{C}$ is the last $(-1)$ curve.

\item Start with a Gorenstein log del Pezzo surface of rank one, with a nodal rational curve $E$ in its smooth locus. Then blow up at the node
of $E$ always along the same branch, finally contract. $C$ is the last $(-1)$ curve.

\item Start with a smooth banana in either $\mathbb{P}^2$
or $\mathbb{F}_2$. Then blow up on one intersection until one curve is negative enough, continue blowing up to make a chain singularity
connecting the two branches. $\tilde{C}$ is the other branch of the banana.

\item Start with one of the surfaces as in Lemma \ref{singular_banana}. Choose a point in either branch, that is not contained in the other. 
Blow up above that point, always along branch of the banana until the branch becomes $K$-positive.

\item Start with $\mathbb{F}_n$, pick the negative section $E$ and a fiber $C$. Then create two dlt double fibers away from $C$.

\item Start from $\mathbb{F}_2$, pick a fiber $E$ and a positive section $C$. Blow up at a smooth point of $E$ so that $E$ is negative and has
an $A_1$ point on it.
\end{enumerate}
\end{lemma}
\begin{proof}
We divide once again our analysis on the type of log canonical singularity at $p$. Since $C$ contains exactly one singular point of $S$, in addition to the cases of the previous
lemmas, $C$ could also be nodal at $p$ (see Lemma \ref{negative_adjunction_2}).

\textbf{Case 1:} Suppose $p$ is a non chain singular point. Then after extracting the divisor $E_1$ adjacent to $C$, the curve $C$ is in the smooth locus of $T_1$. 

\textbf{Case 1a:}
Suppose $T_1$ is a net. $C$ can't be a section as $T_1$ is singular. If $C$ is a fiber, $E_1$ is a section. But then $K_{T_1}+E_1$ must be dlt, contradiction.

\textbf{Case 1b:}
Suppose we get a birational 
contraction to $S_1$. $C$ can't get contracted because $\tilde{C}^2 > -1$.
$\Sigma_1$ can only meet $E_1$ since $C$ is in the smooth locus of $T_1$, so that we get a smooth fence. 
Notice that $C_1$ is in the smooth locus of $S_1$. Therefore $S_1=\overline{\mathbb{F}}_n$ for some $n$ by Lemma \ref{smoothp1}.
However there is no configuration in which $\Sigma_1$ contracts to a $(-n)$ singularity, given the fact that $K_{T_1}+E_1$ is not divisorially log terminal.

\textbf{Case 2:}
Suppose now that $p$ is a node of $C$ and extract one of the two divisors adjacent to $C$. Notice that $K_S+C\equiv 0$.

\textbf{Case 2a:} Suppose $T_1$ is a net and $C$ contains a singular point of $T_1$. $C$ can't be a fiber since $K_{T_1}+C$ is dlt and $C$ contains only one singular point.
If $C$ and $E_1$ are sections, we get description (1).

\textbf{Case 2b}: Suppose $T_1$ is a net and $C$ lies in the smooth locus of $T_1$. If $C$ is a section, we get description $(2)$. If $C$ is a fiber instead, then $E_1$ is a double
section with negative self-intersection, contradiction.

\textbf{Case 2c:} Assume we get a birational transformation. If $C$ contracts, then $(S_1, A_1)$ satisfies again the hypothesis of the lemma, unless $S_1$ is Gorenstein.
This way we get descriptions (3) and (4). If we go to a smooth banana instead we get description (5). If the banana contains a singular points, then we get description
(6) by Lemma \ref{singular_banana}.

\textbf{Case 3:} The last case is the one in which $C$ passes through a $(2,n,2)$ point. Extract the $(-n)$ curve. If $T_1$ is a net, then $C$ must be a fiber, $E_1$ a section.
This is description (7). If we have a birational contraction instead, $C$ can't contract and we must have a fence. Clearly there must be only
one $(2)$ point on $A_1$ since $C_1$ is in the smooth locus of $S_1$, therefore $\Sigma_1$ contracts the other $A_1$ point down to a smooth point.
This is description (8).
\end{proof}

For the case in which $C$ does not contain any singular points of $S$, we simply remark the following.

\begin{lemma}\label{log_canonical_zero}
Suppose $(S,C)$ is a pair of a rank one log del Pezzo $S$ and a curve $C$ such that $K_S+C$ is anti-nef, log canonical but not divisorially log terminal.
If $C$ is in the smooth locus of $S$ then $C$ is nodal and $S$ is Gorenstein.
\end{lemma}
\begin{proof}
Immediate from adjunction.
\end{proof}

Finally, we conclude with a special case that will be useful in Theorem \ref{maintheorem}.

\begin{lemma}\label{almost_log_canonical}
Suppose $(S,C)$ is a pair of a rank one log del Pezzo $S$ and a curve $C$ such that $K_S+C$ is anti-nef and almost log canonical, but not log canonical.
Then $(\tilde{S},\tilde{C})$ is one of the following:
\begin{enumerate}
\item $S=\mathbb{F}_n$. $C$ is a smooth section which touches the negative section only once and of order two.
\item Start from $\mathbb{P}^2$. Choose a line $E$ and a conic $C$ meeting at one point only. Blow up above a point which is not $E\cap C$ until $E$ becomes
$K$-positive.
\item Start from $\mathbb{F}_n$. Choose a positive section $C_1$ disjoint from the negative section, and a positive section $E$ meeting the negative section
only once and transversally, and meeting $C_1$ only once and with order two. Blow up above the intersection of $E$ with the negative section, always along $E$.
\item Start from $\mathbb{F}_n$. Choose $E$ to be the negative section and $C$ a positive section meeting $E$ only once and of order two. Blow up the intersection
point so that there is only one $K$-positive component intersecting $E$ and $C$.
\end{enumerate}
\end{lemma}
\begin{proof}
It is easy to see that $C$ contains only one singular points, $p$, and that $K_S+C$ is numerically trivial.
We go over the three exceptional cases of the definition of almost log canonical.

\textbf{Case 1:} Suppose that there there is only one component $E$ over $p$. Extract $E$ and consider the $K$-negative contraction $\pi$ given by 
Lemma \ref{hunt_transformation}. If $\pi$ is a net, then $E$ is the negative section and $C$ is a positive section meeting $E$ only at one point, with order two.
In particular, $C$ is smooth, and this gives description $(1)$. If instead $\pi$ is a birational contraction, then $\Sigma$ must only meet $E$. 

\textbf{Case 1a:} Suppose that $C$ is smooth. Since
$C_1$ is in the smooth locus of $S_1$, we have that $S_1$ is either $\overline{\mathbb{F}}_n$ or $\mathbb{P}^2$. If $S_1=\mathbb{P}^2$, we get description
$(2)$. Assume then that $S_1=\mathbb{F}_n$. $C_1$ can't be a fiber, for otherwise it would have been contractible on $T_1$. If $A_1$ is a fiber, then
it must lie in the smooth locus, and therefore we get again $\mathbb{P}^2$. Therefore $A_1$ and $C_1$ are both sections on $\mathbb{F}_n$, and we
get description $(3)$.

\textbf{Case 1b:} Suppose that $C$ is nodal. Then $S_1$ is a Gorenstein log del Pezzo surface with a nodal curve in its smooth locus. Then
$K_{S_1}\cdot A_1=-C_1\cdot A_1=-2$ and $C_1$ must lies in the smooth locus by adjunction. In particular $S_1=\overline{\mathbb{F}}_2$, but this cannot happen.

\textbf{Case 2:} Suppose that there are two components $E_1$ and $E_2$ over $p$. Extract $E_1$ and consider the $K$-negative contraction given by 
Lemma \ref{hunt_transformation}. If $\pi$ is a net, then $E_1$ and $C$ must be sections. This is description $(4)$. 
Assume therefore that $\pi$ is birational. $\Sigma$ must pass through $E\cap C$. $C_1^2 = A_1^2 < 1$, and they can't be $(-1)$ curves, for otherwise they would be
contractible. Therefore $C_1$ is a zero-curve, which is a contradiction by adjunction. 
\end{proof}

\section{Classification}\label{sectionclassification}

Let $k$ be an algebraically closed field of characteristic different from two and three. In this section we classify all rank one log del Pezzo surfaces defined over $k$.
We start with a preliminary definition. 

\begin{definition}
A tiger $E$ for $S$ is called exceptional if $E$ does not lie in $S$.
\end{definition}

\begin{theorem}\label{maintheorem}
Let $S$ be a rank one log del Pezzo surface. 

\begin{enumerate}
\item If $S$ has no tigers in $\tilde{S}$ then $S$ is one of the surfaces described in Section \ref{sectionnotigers}.
\item If $S$ has an exceptional tiger in $\tilde{S}$ and $T$ is a net, then $E$ is either a section or a double section. $(T,E)$ is log canonical.
\item If $S$ has an exceptional tiger in $\tilde{S}$ and $T$ is not a net, 
then $S$ is obtained by starting with an almost log canonical pair $(S_1, C)$ such that $K_{S_1}+C$ is anti-nef (these
are described in Section \ref{sectiontigers}), by choosing a point in $\tilde{C}$ and blowing up on that point so that $\tilde{C}$ becomes contractible, and the $K$-positive
curves are contractible to klt singularities.
\item If $S$ has tigers in $\tilde{S}$, but none of which is exceptional, then $(S,C)$ is dlt and is one of the surfaces described in Section \ref{sectiontigers}.
\end{enumerate}

\end{theorem}

\begin{proof}
Part $(1)$ of the statement is clear. Suppose $S$ has an exceptional tiger $E$ in $\tilde{S}$. By definition, this means that there is an effective divisor $\alpha$
such that $K_S+\alpha$ is anti-nef and an extraction $f:T\rightarrow S$ of an exceptional divisor $E$ in $\tilde{S}$ such that $K_T+E+f_* ^{-1}(\alpha) = f^* (K_S+\alpha)$.
Choose $E$ and $\alpha$ such that the coefficient of $E$ for $(S,\emptyset)$ is maximal. 
Let $\pi:T\rightarrow S_1$ be the associated $K_T$-negative contraction as in Lemma \ref{hunt_transformation} and let $C=\pi(E)$. 
If $\pi$ is a net, then $E$ is either a section or a double section and $K_T+E$ is log canonical. This is part $(2)$.
Suppose then that $\pi$ is a birational morphism.
Choose $a>e$ such that $K_T+aE$ is $\pi_1$-trivial. Notice that $K_{S_1}+C$ is anti-nef.
If $a\geqslant 1$, then the pair $(S_1, C)$ is log canonical, since so is $K_T+E$ and $\Sigma$ has negative coefficient. 
Suppose $a<1$. We claim that $(S_1,C)$ is almost log canonical. If that was not the case,
let $c<1$ be the almost log canonical threshold 
(i.e. $e(E; S_1, cC)\leqslant 1$ for all the exceptional divisors $E$ in the minimal resolution and equality holds at least for one of them). 
Clearly $c\geqslant a$ since $K_{S_1}+aC$ is log canonical. Let $E_1$ be such that $e(E_1; S_1, cC)=1$ and let $\tilde{E}_1$ be its strict
transform in $\tilde{S}$. Let $e_1$ be the coefficient of $E_1$ for $(S_1, \emptyset)$. 
Denote by $f(t)=e(\tilde{E}_1; S_1, tC)$. Clearly $f(0)\geqslant 0$ since $E_1$ is exceptional. Also, $f(c)=1$ by definition. Since $f(t)$ is an affine function, we must have $f(a)>a$.
But then $e(\tilde{E}_1; T, aE)>a$. Consider $g(t)=e(\tilde{E}_1; T, tE)$. We have just seen that $g(a)>a$. Also, $g(0)\geqslant 0$. Therefore $g(e)>e$.
However this contradics the maximality of $e$ since $e(\tilde{E}_1; S, \alpha) = e(\tilde{E}_1; T, E+f_* ^{-1}(\alpha))\geqslant e(\tilde{E}_1; S_1, cC)=1$.
Therefore $(S_1, C)$ is log canonical and $K_{S_1}+C$ is anti-nef, so that we can again use Section \ref{sectiontigers}, proving part $(2)$.

Finally, we prove part $(3)$. Let $C$ be a non-exceptional tiger, so that $K_S+C$ is anti-nef. If $(S,C)$ is dlt, then we are done. If not, then there is an exceptional
tiger in the minimal resolution by Lemma \ref{minimalresolution}.
\end{proof}

\subsection{The list}


\begin{definition}[LDP 1]
Choose a smooth rational curve 
$E_1\subseteq \mathbb{P}^1\times \mathbb{P}^1$ which is a triple section for the first projection and has exactly three ramified points $p'$, $q'$, and $r'$.
Suppose that $p'$ and $q'$ are of order two and that $r'$ is of order three. Let $F$, $G$ and $H$ be the corresponding fibers. Blow up above $p'$ three times along $E_1$.
Then blow up twice above $q'$ along $E_1$. Finally, blow up four times above $r'$ along $E_1$.
\end{definition}


\begin{definition}[LDP 2]
Let $A$ and $B$ be two smooth conics meeting to order four in $\mathbb{P}^2$ at a point $c$. Pick a point $b\neq c$ of $B$.
Pick a point $a$ in the intersection of $A$ with the tangent line $M_b$ to $B$ at $b$. Let $b'$ be the other point of the the intersection of $B$
with the line passing through $a$ and $c$. Let $a'$ be a point of the intersection of $A$ with the line through $b$ and $b'$.

Let $\tilde{S}\rightarrow \mathbb{P}^2$ blow up once at $a,a',b'$, twice along $M_b$ at $b$, and five times along $A$ at $c$. 
\end{definition}


\begin{definition}[LDP 3]
Take the cubic $C$ given by $Z^2X=Y^3$ in $\mathbb{P}^2$, the line $L$ given by $Y=0$ and a line $E$ meeting $C$ at $L\cap C$ and two other distinct points $p$ and $q$.
Blow up four times above $[0,0,1]$ along $C$. This gives the minimal resolution of the Gorenstein log del Pezzo $S(A_4)$. Now blow up twice on $p$ along $E$.
This gives the minimal resolution of the Gorenstein log del Pezzo surface $S(A_1+A_5)$. Next, blow up on the cusp of $C$ four times along $C$. 
\end{definition}

\begin{definition}[LDP 4]
Take the cubic $C$ given by $Z^2X=Y^3$ in $\mathbb{P}^2$, the line $L$ given by $Y=0$ and a line $E$ meeting $C$ at two points $p$ and $q$ with order two and one respectively.
Blow up three times above $[0,0,1]$ along $C$. This gives the minimal resolution of the Gorenstein log del Pezzo $S(A_1+A_2)$. Now blow up twice above $p$ along $E$. This
gives the minimal resolution of the Gorenstein log del Pezzo surface $S(3A_2)$. Next, blow up on the cusp of $C$ four times along $C$.
\end{definition}

\begin{definition}[LDP5]
Take the cubic $C$ given by $Z^2X=Y^3$ in $\mathbb{P}^2$, the line $L$ given by $Y=0$ and a line $E$ meeting $C$ at $L\cap C$ and two other distinct points $p$ and $q$. Blow up three times above $[0,0,1]$ along $C$. This gives the minimal resolution of the Gorenstein log del Pezzo $S(A_1+A_2)$. Now blow up twice above
$p$ along $E$. This gives the minimal resolution of the Gorenstein log del Pezzo $S(2A_1+A_3)$. Now blow up at the cusp of $C$ either five or six times.
\end{definition}


\begin{definition}[LDP 6]
Let $B$ and $C$ be two dlt $(-1)$ curves meeting a nodal curve of $S(2A_1+A_3)$ at smooth points, each passing through an $A_1$ point,
and meeting at opposite ends of the $A_3$ point (as in Lemma \ref{nodegorenstein}).

Let $\tilde{S}\rightarrow \tilde{S}(2A_1+A_3)$ blow up once at the intersection of $C$ with the $(-2)$ curve at the $A_3$ singularity, 
twice at $A\cap B$ along $A$ and twice along on the branches of the node of $A$.
\end{definition}


\begin{definition}[LDP 7]
Let $B$ be a $(-1)$ curve in either $S_2 = S(A_1+A_5)$ or $S_2 = S(3A_2)$ passing through two singularities, $B$ dlt. Let $A$ be
a nodal rational curve in the smooth locus of $S_2$.

Let $\tilde{S}\rightarrow \tilde {S}_2$ blow up twice at $A\cap B$ along $B$ and three times on the node of $A$ along the same branch.
\end{definition}

\begin{definition}[LDP 8]
Let $A$ be a nodal curve in the smooth locus of $S(2A_1+A_3)$, $B$ a dlt $(-1)$ curve through $A_1$ and $A_3$.

Let $\tilde{S} \rightarrow \tilde{S}(2A_1+A_3)$ blow up on $A\cap B$ twice along $B$, blow up the node of $A$ four or five times along the same 
branch.
\end{definition}

\begin{definition}[LDP 9]
Let $A$ be a nodal curve in the smooth locus of $S(2A_1+A_3)$, $B$ a dlt $(-1)$ curve through $A_1$ and $A_3$.

Let $\tilde{S} \rightarrow \tilde{S}(2A_1+A_3)$ blow up three times on $A\cap B$ along $B$, blow up four times on the node along the same branch.
\end{definition}


\begin{definition}[LDP 10]
Let $A$ be a nodal curve in the smooth locus of $S(A_1+A_5)$ and $B$ a log terminal $(-1)$ curve.

Let $\tilde{S}\rightarrow \tilde{S}(A_1+A_5)$ blow up at the intersection of $B$ with the $(-2)$ curve in the $A_5$ singularity, blow up 
the node of $A$ twice along one branch and then once along the nearest point of the other branch.
\end{definition}


\begin{definition}[LDP 11]
Let $A$ be a nodal curve in the smooth locus of the Gorenstein log del Pezzo $S_2$ and $B$ a dlt $(-1)$ curve that passes through 
two singular points. Let $\tilde{S}\rightarrow \tilde{S}_2$ blow up $t$ times on the intersection of $B$ with the $(-2)$ curve 
relative to the specified point $p$, always along $B$, then blow up $s$ times at the node of $A$, always along the same branch, 
for $p$,$t$ and $s$ as follows.
\begin{enumerate}
\item $S_2 = S(A_1+A_2)$, $p=A_1$ and $(t,s)=(2,6),(1,6),(1,7)$.
\item $S_2 = S(A_1+A_2)$, $p=A_2$ and $(t,s)=(3,6),(2,6),(1,6),(1,7),(1,8)$.
\item $S_2 = S(A_1+A_5)$, $p=A_1$ and $(t,s)=(1,3)$.
\item $S_2 = S(A_1+A_5)$, $p=A_5$ and $(t,s)=(2,3),(1,3),(1,4)$.
\item $S_2 = S(3A_2)$, $p=A_2$ and $(t,s)=(1,3)$.
\item $S_2 = S(A_2+A_5)$, $p=A_5$ and $(t,s)=(1,2)$.
\item $S_2 = S(A_2+A_5)$, $p=A_2$ and $(t,s)=(1,2)$.
\item $S_2 = S(A_1+2A_3)$, $p=A_3$ and $(t,s)=(1,2)$.
\item $S_2 = S(2A_1+A_3)$, $p=A_3$ and $(t,s)=(1,5),(1,4),(2,4)$.
\item $S_2 = S(2A_1+A_3)$, $p=A_1$ and $(t,s)=(1,4)$.
\end{enumerate}
\end{definition}

\begin{definition}[LDP 12]
Let $A$ be a nodal rational curve in the smooth locus of $S(A_1+A_2)$, and $B$ a dlt $(-1)$ curve passing through the two 
singularities. 

Let $\tilde{S}\rightarrow \tilde{S}(A_1+A_2)$ blow up $t$ times $A\cap B$ along $B$ and $s$ times at the node of $A$, always
along the same branch for $(t,s)=(2,5),(2,6)$, $(3,6),(4,6),(5,6)$, $(2,7),(3,7),(2,8),(2,9)$.
\end{definition}

\begin{definition}[LDP 13]
Let $A$ be a nodal rational curve in the smooth locus of $S(A_1+A_2)$, and $B$ a dlt $(-1)$ curve passing through the two 
singularities. 

Let $\tilde{S}\rightarrow \tilde{S}(A_1+A_2)$ blow up $A\cap B$ twice along $A$, and once near $B$, then blow up the node at $A$
five, six or seven times along the same branch.
\end{definition}

\begin{definition}[LDP 14]
Let $A$ be the nodal rational curve in the smooth locus of $S(A_1+A_2)$, and $B$ the dlt $(-1)$ curve passing through the two 
singularities. 

Let $\tilde{S}\rightarrow \tilde{S}(A_1+A_2)$ blow up $A\cap B$ twice along $B$, and once near $A$, then blow up the node at $A$ six times.
\end{definition}


\begin{definition}[LDP 15]
Let $A$ and $B$ be two positive sections disjoint from the negative section in $\mathbb{F}_2$. Suppose $A$ and $B$ intersect at $p$ and $q$. 
Let $F$ be a fiber, not passing
through $p$ and $q$. Blow up $r+1$ times at $u=A\cap F$ along $F$ and $s+1$ times at $p$ along $B$. 
Now continue in one of the following manners.

If $(s,r)=(3,2)$ then

\begin{enumerate}
\item Blow up above $q$ along $A$ four or five times, or
\item Blow up above $p$ along $A$ three or four times, or
\item Blow up on $u$ along $A$ three times.
\end{enumerate}

If $(s,r)=(4,1)$ then
\begin{enumerate}
\item Blow up above $q$ along $A$ five, six, seven or eight times, or
\item Blow up above $p$ along $A$ $k$ times with $3\leqslant k\leqslant 7$, or
\item Blow up on $u$ along $A$ three or four times.
\end{enumerate}
\end{definition}

\begin{definition}[LDP 16]
Let $A$ and $B$ be two positive sections not intersecting the negative section in f $\mathbb{F}_2$ intersecting at $p$ and $q$. 
Let $F$ be a fiber, not passing
through $p$ and $q$. Blow up $r+1$ times at $u=B\cap F$ along $F$ and $s+1$ times at $p$ along $B$. 
Now continue in one of the following manners.

If $(s,r)=(3,2)$ then blow up  $A\cap F$ along $A$ four times.

If $(s,r)=(4,1)$ then

\begin{enumerate}
\item Blow up $q$ along $A$ six times, or
\item Blow up above $A\cap F$ along $A$ four times, or
\item Blow up on $p$ along $A$ five times.
\end{enumerate}

If $(s,r)=(3,1)$ then
\begin{enumerate}
\item Blow up $q$ along $A$ $k$ times, with $5\leqslant 9$ times, or
\item Blow up $A\cap F$ $k$ times, with $4\leqslant 7$, or
\item Blow up $p$ along $A$ $k$ times, with $4\leqslant k \leqslant 8$.
\end{enumerate}
\end{definition}

\begin{definition}[LDP 17]
Let $A$ and $B$ be two positive sections disjoint from the negative section in $\mathbb{F}_2$ intersecting at $p$ and $q$. 
Let $F$ be a fiber, not passing
through $p$ and $q$. Blow up twice at $u=A\cap F$ along $A$, and five times at $p$ along $B$, twice along $A$ and once away from $A$.
\end{definition}


\begin{definition}[LDP 18]
Let $A$ be a positive section on $\mathbb{F}_n$ and $E$ the negative section. Pick three points on $E$ or $A$ all lying in distinct fibers,
blow them up once, then blow up at the intersection of the $(-1)$ curves and continue blowing up to get log terminal fibers such that
$E$ and the singularities on it form a log terminal non chain singularity. 
Now to define $\tilde{S}$ pick any point on $A$ intersecting a $K$-positive curve, and blow up that point in any fashion
so that $A$ becomes $K$-positive and the $K$-positive curves are contractible to klt singularities. 
\end{definition}

\begin{definition}[LDP 19]
Suppose $\operatorname{char}(k)=5$.
Start with the surface in Definition \ref{sa4char5}. Blow up its cusp three times, then contract the resulting rational curve and two of the exceptional divisors.
Therefore we obtain a log del Pezzo surface with singularities $2A_4+(5)+(3)+(2)$.
\end{definition}


\begin{definition}[LDP 20]
Let $(S,C)$ be one of the surfaces in Proposition \ref{log_terminal_three} (3)-(5), Lemma \ref{log_terminal_two}, Lemma \ref{log_terminal_one}.
Define $\tilde{S}=S_1$.
\end{definition}

\begin{definition}[LDP 21]
Let $(S,C)$ be one of the surfaces in Proposition \ref{log_terminal_three}, Lemma \ref{log_terminal_two}, Lemma \ref{log_terminal_one}, Lemma \ref{log_terminal_zero}, 
Lemma \ref{log_canonical_three}, Lemma \ref{log_canonical_two_lt}, Lemma \ref{log_canonical_two_lc}, Lemma \ref{log_canonical_one}, Lemma \ref{log_canonical_zero}
or Lemma \ref{almost_log_canonical}.
Then define $\tilde{S}$ by a sequence of blow ups on $C$ such that $C$ becomes $K$-positive, and $K$-positive curves are contractible to klt singularities.
\end{definition}

\begin{definition}[LDP 22]
Start from $\mathbb{F}_n$. Let $E$ be either a section or a double section. Blow up so that $E$ is $K$-positive, all $K$-positive curves are contractible and 
$E$ is the boundary of a log canonical pair.
\end{definition}

We are now ready to state the classification theorem.

\begin{theorem}[Classification of rank one log del Pezzo surfaces]\label{total_classification}
Let $S$ be a rank one log del Pezzo surface over an algebraically closed field of characteristic different from two and three.
If $S$ is smooth then $S=\mathbb{P}^2$; if $S$ is Gorenstein, then $S$ is one of the surfaces in Theorem \ref{gorenstein_classification}, 
otherwise it is one of the log del Pezzo surfaces in the families LDP1 to LDP22.
\end{theorem}
\begin{proof}
This is the content of Theorem \ref{maintheorem}.
\end{proof}

\section{Applications}\label{applications}

In this section we highlight some rather immediate applications of Theorem \ref{total_classification}.

\subsection{Liftability to characteristic zero}

One of the main applications is that every rank one log del Pezzo surface defined over an algebraically closed field of characteristic strictly higher than five
admits a log resolution that lifts to characteristic zero over a smooth base.

This answers the question after \cite[Theorem 1.1]{cascini}. Let's start first with the definition of liftability as in \cite{cascini}.

\begin{definition}
Let $X$ be a smooth variety over a perfect field $k$ of characteristic $p>0$, and let $D$ be a simple normal crossing divisor on $X$.
Write $D=\sum_i D_i$, where $D_i$ are the irreducible components of $D$. We say that a pair $(X,D)$ is liftable to characteristic zero over a smooth
base if there exists
\begin{enumerate}
\item A scheme $T$ smooth and separated over Spec $\mathbb{Z}$.
\item A smooth and separated morphism $\mathcal{X}\rightarrow T$.
\item Effective Cartier divisors $\mathcal{D}_1 ..., \mathcal{D}_r$ on $\mathcal{X}$ such that the scheme-theoretic intersection
$\bigcap_i \mathcal{D}_i$ for any subset $J\subseteq \{1,...,r\}$ is smooth over $T$.
\item Amorphism $\alpha:$ Spec $k\rightarrow T$
such that the base changes of the schemes $\mathcal{X}, \mathcal{D}_i$ over $T$ by $\alpha$ are isomorphic to $X,D_i$ respectively.
\end{enumerate}
\end{definition}

\begin{theorem}\label{lift_theorem}
Let $S$ be a rank one log del Pezzo surface over an algebraically closed field of characteristic $p>5$. Let $\pi: \tilde{S}\rightarrow S$ be the minimal resolution.
Then $(\tilde{S},Ex(\pi)))$ lifts to characteristic zero over a smooth base.
\end{theorem}
\begin{proof}
If $S$ is Gorenstein, the result follows by writing the integral Weierstrass model of the corresponding elliptic surface (see \cite{lang3}). 
Suppose therefore that $S$ is not Gorenstein. Notice that in the classification of Theorem \ref{total_classification} each $S$ is determined
by the geometry of the first surface from which we start the construction of $S$. Since these clearly lift to characteristic zero, and the smooth
blow ups also do, we deduce that $S$ has a log resolution that lifts as well. 
\end{proof}

As a byproduct of Section \ref{sectiontigers} we also get the following:

\begin{theorem}\label{lift_theorem2}
Let $S$ be a rank one log del Pezzo surface over an algebraically closed field of characteristic $p>5$ and $C$ a curve in $S$. 
Suppose that $K_S+C$ is anti-nef and log canonical. Let $\pi: \tilde{S} \rightarrow S$ be the minimal resolution. Then $(\tilde{S},\operatorname{Ex}(\pi) + \tilde{C})$ 
lifts to characteristic zero over a smooth base.
\end{theorem}
\begin{proof}
Immediate from the classification in Section \ref{sectiontigers}.
\end{proof}

\subsection{Non liftable examples in low characteristic}\label{not_lift}

In this subsection we will see that in characteristic two, three and five, there are log del Pezzo surfaces of rank one
that do not lift to characteristic zero over a smooth base. Therefore the conclusion of Theorem \ref{lift_theorem} is sharp.
The first such example was shown in \cite[Chapter 9]{mckernan} in characteristic two. The following examples are not
liftable to characteristic zero because they do not satisfy the Bogomolov bound of \cite[Chapter 9]{mckernan} (see the proof of \cite[Theorem 1.3]{cascini}).

We start with an example in characteristic three due to Fabio Bernasconi \cite{fabio}.

\begin{example}[characteristic 3]\label{char3}
Pick $C$ to be $y=x^3$ in $\mathbb{P}^1 \times \mathbb{P}^1$. Choose three points on $C$ and blow up three times each along $C$.
This gives a log del Pezzo surface with singularities $4(3)+3A_2$.
\end{example}

The following example in characteristic three is new.

\begin{example}[characteristic 3]
Pick $C$ to be $y=x^3$ in $\mathbb{P}^2$. Choose three flex lines $L_1, L_2, L_3$ such that their intersection with $C$ lie on a line $L_4$.
Blow up three times two flex points, blow up twice the third one,always along $C$ and then blow up $L_1\cap L_2\cap L_3$. Finally blow up the cusp
of $C$.
This gives a log del Pezzo surface with singularities $2(3)+E_6+(2,3)$.
\end{example}

Finally let's conclude with an example in characteristic five.

\begin{example}[characteristic 5]\label{char5}
Start with the surface in Definition \ref{sa4char5}. Let $C$ be the cuspidal rational curve contained in its smooth locus. 
Blow up the cusp of $C$ three times, then contract $C$ and two of the exceptional divisors.
We obtain a log del Pezzo surface with singularities $2A_4+(5)+(3)+(2)$.
\end{example}

\subsection{Kodaira vanishing theorem}

Kodaira vanishing theorem is known to fail for log del Pezzo surfaces in characteristic two and three. However, a consequence of Theorem \ref{lift_theorem}
is that it holds for rank one log del Pezzo surfaces in characteristic strictly higher than five. More precisely:

\begin{theorem}\label{kodairavanishing}
Let $k$ be an algebraically closed field of characteristic $\operatorname{char}(k)>5$.
Let $S$ be a rank one log del Pezzo surface defined over $k$ and let $A$ be an ample Cartier divisor on $S$. Then $H^i(S, K_S+A)=0$ for all $i>1$.
\end{theorem}
\begin{proof}
See \cite[Lemma 6.1]{cascini}.
\end{proof}

\appendix
\section{Surface singularities}\label{surfacesings}

In this Appendix we give a detailed description of klt, lc and dlt surface singularities. We also provide techniques that are useful in computations. The main references for the discussion
are \cite[Chapter 3]{flips} and \cite{mckernan}. 
In what follows, unless specified otherwise, $S$ indicates the germ (in the Zariski topology) of a normal surface at a point $p\in S$. We start with some definitions.

\begin{definition}
Let $p\in S$ be a singular point and let $\pi: \tilde{S}\rightarrow S$ be the minimal resolution. The point $p$ is a chain singularity if the dual graph of the minimal resolutions above $p$ is a chain.
If the dual graph is not a chain, we will say that $p$ is a non chain singularity.

If $p$ is a chain singularity, we say that $p$ has type $(-E_1^2, -E_2^2,\cdots,-E_n^2)$, where $E_1, \cdots, E_n$ are the exceptional curves over $p$. 
\end{definition}

\begin{definition}
A point $p\in S$ is called Du Val if the coefficient of any exceptional irreducible divisor of the minimal resolution over $p$ is zero.
 By adjunction, this is equivalent to saying that for every such divisor $E$ is a smooth rational curve and that we have $E^2 = -2$.
\end{definition}

\begin{remark}
In the case of surfaces, Du Val singularities are the same as klt Gorenstein singularities.
\end{remark}

\begin{definition}
We denote the singularity $(2, \cdots, 2)$ by $A_j$, where $j$ is the number of components of the dual graph. 
\end{definition}

\begin{definition}
Let $p\in S$ be a singular point and let $\Gamma$ be the dual graph of its minimal resolution. The index $\Delta(\Gamma)$ is the absolute value of the determinant of the intersection
matrix of $\Gamma$.
\end{definition}

\begin{lemma}\label{kltsing}
Let $p\in S$ be a klt singularity. Then every curve of the dual graph of its minimal resolution is a smooth rational curve. 
If $p$ is a non chain singularity, then $p$ has a fork with three chain branches. 
The indexes $(\Delta_1,\Delta_2,\Delta_3)$ of the branches are one of the following: $(2,2,n)$, $(2,3,3)$, $(2,3,4)$ or $(2,3,5)$.
We will refer to the fork vertex as the central curve of $p$.
\end{lemma}

There is an analogous description for lc singularities, for which we refer to \cite[Chapter 3, page 58]{flips}.

\begin{lemma}\label{dltsing}
Let $p\in S$ and let $C$ be an irreducible and reduced curve through $p$. If $(S,C)$ is dlt at $p$ then either:
\begin{enumerate}
\item $p$ is smooth and $C$ has simple normal crossings at $p$ or
\item $p$ is singular, in which case it must be a chain singularity, $C$ is smooth at $p$ and touches normally one of the ends of the chain. 
\end{enumerate}
\end{lemma}

\begin{lemma}\label{lcnotdltsing}
Let $p\in S$ and let $C$ be an irreducible and reduced curve through $p$. If $(S,C)$ is lc but not dlt at $p$ then either:
\begin{enumerate}
\item $p$ is smooth and $C$ has a simple node at $p$ or
\item $p$ is a chain singularity and $C$ touches both ends of the singularity normally or
\item $p$ is a chain singularity with exactly three components and $C$ touches the central component normally or
\item $p$ is a non chain singularity, in which case two of the branches have index two, and $C$ touches the opposite end of the third branch normally.
\end{enumerate}
\end{lemma}

Once we know the dual graph of the minimal resolution over a klt singularity $p$ and all the self-intersections $E_i^2$, we can compute the coefficients of the $E_i$. In fact
it's sufficient to solve the $n$ equations $(K_{\tilde{S}} + \sum_i e_i E_i)\cdot E_i = 0$ by using adjunction. The following lemmas can be useful in speeding up computations, especially
when $p$ is a non chain singularity.

\begin{lemma}
Let $\Gamma$ be the dual graph of the minimal resolution at $p$ and let $v$ be a vertex of $\Gamma$. Let $v_1, \cdots, v_s$ be the vertices adjacent to $v$. Then
\[
  \Delta(\Gamma) = n\cdot \Delta(\Gamma - v) - \sum_i \Delta(\Gamma-v-v_i)
\]
\end{lemma}

\begin{lemma} The coefficients of the exceptional divisors $E_j$ over $p$ are given by
\[
1-e_i = \frac{1}{\Delta(\Gamma)} \sum_k \Delta(\Gamma-(\text{path from $v_j$ to $v_k$}))\cdot c_k
\]
where $c_k$ is given by $c_k = 2- E_k (\sum_{l\neq k} E_l)$. It is useful to notice that $c_k=0$ if and only if $v_k$ has exactly two neighbors, $c_k=1$ if and only if it has one neighbor and
$c_k=-1$ if and only if it has three neighbors. 
\end{lemma}

\begin{definition}
Let $p$ be a klt cyclic singularity $(-E_1^2, \cdots, -E_n^2)$. We say that $p$ has spectral value $k$ if the coefficient of $E_1$ has the form $k/r$ where $r$ is the index
of the singularity at $p$.
\end{definition}

\begin{lemma}\label{smallsvalue}
If $\beta=(j,\alpha)$ then the difference of the spectral value of $\beta$ and $\alpha$ is $(j-2)r$, where $r$ is the index of $\alpha$. In particular:
\begin{enumerate}
\item $\alpha$ has spectral value 0 if and only if $\alpha= A_j$.
\item $\alpha$ has spectral value 1 if and only if $\alpha=(A_j, 3)$ or $(3)$.
\item $\alpha$ has spectral value 2 if and only if $\alpha=(4)$ or $\alpha=(3,2)$.
\end{enumerate}
\end{lemma}

\begin{definition}
We say that a point $p\in S$ of spectral value 1 is almost Du Val.
\end{definition}

While the notion of spectral value is important for the \say{hunt}, the main relevance in the present discussion comes from the following lemmas.

\begin{lemma}\label{coefficient_boundary}
Suppose that $p\in S$ is a klt chain singularity and that $D$ is an irreducible divisor through $p$, touching $E_1$ normally (i.e. $(S,D)$ is dlt at $p$). Then
\[
e(E_1, K_S+\lambda D)= (k/r)+ \lambda(r-1-k)/r = \lambda(r-1)/r + (1-\lambda)(k/r)
\]
\end{lemma}

\begin{lemma}\label{nonchainsing}
Let $p$ be a klt non chain singularity with branches $\beta_1, \beta_2, \beta_3$. Assume that the central curve is a $(-l)$ curve. Then the coefficient of $e$ of $p$ is the same
as the coefficient of the central curve, and it is a rational number of the form $k/(k+1)$ for some positive integer $k$.

Moreover, if $\beta_1=\beta_2=(2)$ and $\alpha=(l,\beta_3)$ then $k$ is the spectral value of $\alpha$.
\end{lemma}

For the convenience of the reader, we list below all klt singularities with small coefficient. This can be done with a straightforward computation and we refer to \cite[Proposition 10.1]{mckernan}.

\begin{proposition}\label{singlist}
Let $p\in S$ be a klt singularity with coefficient $0<e<3/5$. Then the possibilities for $p$ are as follows.
\begin{enumerate}
\item $e<1/2$: $(3,A_j)$. $e=(j+1)/(2j+1)$.
\item $e=1/2$:
\begin{enumerate}
\item $(4)$
\item $(3,A_j, 3)$
\item $(2,3,2)$
\item $p$ is a non chain singularity, with center $(2)$ and branches $(2)$, $(2)$ and $(A_j, 3)$, with the central curve and the unique $(-3)$ curve meeting opposite ends of the $A_j$ chain.
This is the only non chain singularity with $e<2/3$.
\end{enumerate}
\item $1/2<e<3/5$:
\begin{enumerate}
\item $(2,3,A_j)$ with $2\leqslant j\leqslant 4$. $e=(2j+2)/(3j+5)$.
\item $(4,2)$. $e=4/7$.
\end{enumerate}
\end{enumerate}
\end{proposition}

\subsection{Adjunction}

A recurring topic in our discussion is the fact that \say{negativity} of the canonical divisor should control the behavior of singularities. In the following we collect some useful lemmas
that deal with the case in which $K_S+\Delta$ is non-positive and $\Delta$ is an integral divisor. We start first by recalling the results in \cite[Appendix L]{mckernan}

\begin{lemma}\label{adjunction1}
Let $C\subseteq S$ be a reduced irreducible curve on a $\mathbb{Q}$-factorial projective surface. If $(K_S+C)\cdot C < 0$ then $C$ is rational.
\end{lemma}
\begin{proof}
See \cite[Lemma L.2.2]{mckernan}
\end{proof}

\begin{lemma}\label{adjunction2}
Let $S$ be a $\mathbb{Q}$-factorial projective surface, and $C$, $D\subseteq S$ two curves, with $C$ integral. If $(K_S+C+D)\cdot C\leqslant 0$, and $C\cap D$ contains at least two points,
then $K_S+C+D$ is log canonical, $C\cap D$ consists exactly of two points, and $(K_S+C+D)\cdot C=0$.
\end{lemma}
\begin{proof}
See \cite[Corollary L.2.3]{mckernan}.
\end{proof}

\begin{lemma}\label{negative_adjunction_2}
Let $S$ be a log del Pezzo surface and $C$ an irreducible curve in $S$. Suppose that $(S,C)$ is log canonical but not divisorially log terminal.
\begin{enumerate}
\item If $K_S+C$ is anti-ample, then $C$ contains at most two singular points. If furthermore $C$ contains exactly two singular points, then $K_X+C$ is dlt at one of them.
\item If $K_S+C$ is numerically trivial, $C$ contains at most three singular points. If furthermore $C$ contains exactly three singular points, then two of them are $A_1$ points and
$(S,C)$ is dlt along them, and $(S,C)$ is not dlt at the third point.
\item If $K_S+C$ is anti-nef and $C$ contains at least two singular points, then $C$ is smooth.
\end{enumerate}
\end{lemma}
\begin{proof}
Let start by proving $(1)$. Suppose by contradiction that $C$ contains at least three singular points
$p$, $q$, $r$. We may assume that $K_S+C$ is not dlt at $p$.
By inversion of adjunction, the coefficient of $p$ in $\text{Diff}_C (0)$ is at least one. Since the coefficient of any singular point is at least $1/2$, however, we have that
$(K_S+C)\cdot C = -2+\operatorname{deg}(\text{Diff}_C (0))\geqslant 0$, contradiction. Therefore $C$ contains at most two singular points. Suppose now
that $C$ contains exactly two singular points. If $(S,C)$ is not dlt at either point, then both of their coefficients in $\operatorname{Diff}_C (0)$ are one. This implies again that 
$(K_S+C)\cdot C = -2+\operatorname{deg}(\text{Diff}_C (0)) = 0$, contradiction.

This proves $(1)$. The same type of reasoning gives $(2)$. For $(3)$ notice that if $C$ is not smooth then by adjunction $(K_S+C)\cdot C\geqslant 0$, 
with equality exactly when $C$ has one node and contains no singularities of $S$ outside the node. 
\end{proof}

\section{Gorenstein log del Pezzo surfaces}\label{gorenstein}

Recall that a normal projective Cohen-Macaulay variety $X$ is called Gorenstein if the canonical divisor $K_X$ is Cartier. 
In this appendix we classify Gorenstein log del Pezzo surfaces of Picard number one
over algebraically closed fields of characteristic different from two and three. Their singularities were first classified
over the complex numbers by Furushima \cite{furushima} (see also \cite{zhang1} and \cite{zhang2}). 
The analogous result in positive characteristic can easily be derived from existing papers, but we were unable to 
find it explicitly stated in the literature. In the following treatment, we take the approach of Ye \cite{yegorenstein}, who reduces the study of Gorenstein log del Pezzo surfaces to the study
of extremal rational elliptic surfaces. These were classified over the complex numbers by Miranda and Persson in \cite{mirandapersson}, and over algebraically closed fields
of positive characteristic by Lang in \cite{lang1} and \cite{lang2}.

A complete classification of log del Pezzo surfaces of index two (i.e. $K_X$ is not Cartier, but $2K_X$ is) was obtained in \cite{nakayama} by different methods.

\subsection{Extremal rational elliptic surfaces}

\begin{definition}
An elliptic surface is a smooth relatively minimal surface $X$ over a curve $C$, such that the general fiber is a smooth curve of genus one. 
\end{definition}

Let $X$ be a smooth projective variety and let $f:X\rightarrow C$ be a fibration such that the general fiber is irreducible. 
Over the complex numbers we can deduce that the general fiber is in fact smooth by generic smoothness. 
In positive characteristic, however, generic smoothness no longer holds and some care is needed. The following lemma, largely taken from \cite{miranda}, 
shows that we recover smoothness of the general fiber in a special case. 

\begin{lemma}\label{elliptic_fibration}
Let $f: X\rightarrow \mathbb{P}^1$ be a smooth relatively minimal surface with section such the general fiber is irreducible and of arithmetic genus one.
Suppose also that $X$ is rational, the image of the section is a $(-1)$ curve and $\operatorname{char}(k)\neq 2,3$.
Then $f:X\rightarrow \mathbb{P}^1$ is obtained by resolving the rational map induced by a pencil of cubic curves in $\mathbb{P}^2$ whose general member is smooth.
In particular $X$ is an elliptic surface.
\end{lemma}
\begin{proof}
For the general fiber $F$ we have that $K_X \cdot F = 0$ by adjunction (and by Kodaira's classification of singular fibers \cite[Table I.4.1]{miranda}).
Hence $K_X \equiv nF$ for some integer $n$. Consider the image
$S$ of the section of $f$, which is a $(-1)$ curve by assumption. Again by adjunction we get that $K_X \cdot S = -1$, hence $K_X \equiv -F$. It follows that $-K_X$ is nef and that
every rational curve has self intersection at least $-2$. Let $g:X \rightarrow M$ be a blowdown to a Mori fiber space $M$. From the above considerations, $M$ can only be 
$\mathbb{F}_0, \mathbb{F}_2$ or $\mathbb{P}^2$.  In each of these cases $X$ dominates $\mathbb{P}^2$ and therefore we get a contraction $h: X\rightarrow \mathbb{P}^2$.
Pushing forward $|F|$ to $\mathbb{P}^2$ we obtain a pencil of curves that are numerically equivalent to the push forward of elements of $|-K_X|$.
It follows then that $f_* |F|$ is a pencil in $|-K_{\mathbb{P}^2}|$, and it is therefore a pencil of cubics. 
Finally note that the general member of a pencil of cubics in characteristic different from two or three is smooth by \cite[Lemma I.5.2]{miranda} and the comment after it. 
\end{proof}

\begin{definition}
Let $f: X\rightarrow C$ be an elliptic surface with section $\sigma$. The section $\sigma$ naturally gives all the fibers the structure of elliptic curves. The set of all sections of $f$ is then a group,
where the multiplication is done fiber by fiber and $\sigma$ is the identity element. This group is called the Mordell-Weil group of $X$ and is denoted by $\operatorname{MW}(X)$.
\end{definition}

\begin{definition}
Let $f:X\rightarrow C$ be a rational elliptic surface with section $\sigma$. We say that $X$ is extremal if the Mordell-Weil group of $X$ is finite and 
$\operatorname{NS}(X)_{\mathbb{Q}}$ is generated by the classes of $\sigma(C)$ and of the vertical components. 
\end{definition}

\begin{theorem}\label{elliptic_classification}
The classification of the singular fibers of extremal rational elliptic surfaces over algebraically closed fields with $\operatorname{char}(k)\neq 2,3,5$
is the same as over $\mathbb{C}$. The configurations are listed in the following table using Kodaira's notation:

\begin{center}
\begin{tabular}{ c c c c}
 Singular fibers & $|\operatorname{MW}(X)|$ & Singular fibers & $|\operatorname{MW}(X)|$\\ 
 $II,II^*$ & $1$ & $I_1 ^*, I_4, I_1$ & $2$\\  
 $III, III^*$ & $2$ &  $I_2 ^*, I_2, I_2$ & $4$\\
 $IV,IV^*$ & $3$ &   $I_9, I_1, I_1, I_1$ & $3$\\
 $I_0 ^*, I_0 ^*$ & $4$ &  $I_8, I_2, I_1, I_1$ & $4$\\
 $II^*, I_1, I_1$ & $1$ &  $I_6, I_3, I_2, I_1$ & $6$\\
 $III^*, I_2, I_1$ & $2$ &  $I_5, I_5, I_1, I_1$ & $5$\\
 $IV^*, I_3, I_1$ & $3$ &  $I_4, I_4, I_2, I_2$ & $8$\\
 $I_4 ^*, I_1, I_1$ & $2$ &  $I_3, I_3, I_3, I_3$ & $9$
\end{tabular}
\end{center}

If $\operatorname{char}(k)=5$ the classification is the same, except that in the above table one replaces $I_5$, $I_5$, $I_1$, $I_1$ with $I_5$, $I_5$, $II$. Furthermore, if 
$\operatorname{char}(k)\neq 2,3$, for every configuration of possible singular fibers in the above table there is a unique extremal rational elliptic surface with section with that
configuration of singular fibers, except for the case $(I_0^* , I_0^*)$. In this case there are infinitely many extremal rational elliptic surfaces with that configuration of fibers.
\end{theorem}
\begin{proof}
See \cite[Theorem 2.1]{lang1}, \cite[Theorem 4.1]{lang2}, \cite[Theorem 4.1]{mirandapersson} and \cite[Theorem VIII.1.5]{miranda}.
\end{proof}

\subsection{Reduction to elliptic surfaces}\label{reduction}

We recall here the relation between Gorenstein log del Pezzo surfaces and extremal rational elliptic surfaces. Throughout this discussion we assume $\operatorname{char}(k)\neq 2,3$.
Let $V$ be a Gorenstein log del Pezzo surface of rank one such that $K_V ^2 = 1$. Let $U$ be its minimal resolution, with
exceptional locus $D$, consisting of eight $(-2)$ curves by Lemma \ref{sum_ten}. 
The general member of the pencil $|-K_U|$ is reduced and irreducible by \cite[Chapter III, Th\'{e}or\`{e}me 1, (b')]{demazure} 
(see the remark in the introduction to Chapter IV), and of arithmetic genus one by adjunction. This pencil has only one base point $p$ by \cite[Chapter III, Proposition 2]{demazure}.
After blowing it up, we are in the hypothesis of Lemma \ref{elliptic_fibration}, where we may take the exceptional curve of the blow up as a section. 
Therefore $X=\operatorname{Bl}_p (U)$ is an elliptic surface. 
Notice that $p$ is not contained in the support of $D$ because $K_U \cdot D = 0$. Hence the strict transform of $D$ in $X$ is contained in the union of the 
singular fibers and does not meet the section. By Lemma \ref{sum_ten} we have that $\rho(U) = 10-K_U ^2 = 9$, and therefore $\rho(X)=10$. By the Shioda-Tate formula 
(\cite[Corollary 6.13]{shioda}), however, we have that $\rho(X) = \sharp D + \operatorname{rank}(\operatorname{MW}(X)) + 2$, 
where $\sharp D$ indicates the number of irreducible components of the support of $D$.
It follows then that the Mordell-Weil group of $X$ is finite and therefore $X$ is an extremal rational elliptic surface. 

From this discussion we see that we may obtain every rank one Gorenstein log del Pezzo surface $V$ with $K_V ^2=1$ by starting with an extremal rational elliptic surface of Picard number
ten, selecting a section, contracting all the $(-2)$ curves not meeting the section and then blowing down the section. 

Suppose now that $V$ is a Gorenstein log del Pezzo surface such that $2\leqslant K_V ^2 \leqslant 8$ and let $U$ be its minimal resolution. 
If $K_V ^2 =8$ then $V$ is obviously isomorphic to $\overline{\mathbb{F}}_2$, so we further assume $K_V^2 \leqslant 7$. Let $C$ be any $(-1)$ curve on $U$
(such a curve exists because the Picard number of $U$ is at least three).
Every irreducible divisor $A$ in the linear system $|-K_U|$ meets $C$ exactly once, and by taking $A$ to be general, we may assume that there are no other $(-1)$ curves passing 
through $p=A\cap C$. 
Blowing up $p$ and then contracting all the $(-2)$ curves, we get a rank one log del Pezzo surface $V'$ such that $K_{V'}^2 = K_V ^2 -1$. By repeating this process we may reduce
our analysis to the case when $K_V^2 =1$, which is the content of Lemma \cite[Lemma 3.2]{yegorenstein}.

We summarize our conclusions in the following result:

\begin{theorem}\label{reductiontheorem}
Let $V$ be a rank one Gorenstein log del Pezzo surface with $K_V ^2 \leqslant 7$ and let $U\rightarrow V$ be the minimal resolution. If $\operatorname{char}(k)\neq 2,3$,
there exists an extremal rational elliptic surface $Y$ and a morphism $f:Y\rightarrow U$ that is a composition of blow downs of some $(-1)$-curves.
\end{theorem}
\begin{proof}
Immediate from the previous description. For more detail see \cite[Theorem 3.4]{yegorenstein}.
\end{proof}

Thanks to Theorem \ref{reductiontheorem} and Theorem \ref{elliptic_classification}, one can now classify all Gorenstein log del Pezzo surfaces. In fact, the only $(-1)$ curves
in an extremal rational elliptic surface $W$ are the members of the Mordell-Weil group. In order to get the minimal resolution of
any Gorenstein log del Pezzo surface it is therefore enough to consider an extremal rational elliptic surface $X$, contract a section, and then keep contracting the $(-1)$ curves
that are formed in the process along the singular fibers. The result is the following:

\begin{theorem}\label{gorenstein_classification}
Let $V$ be a singular rank one Gorenstein log del Pezzo surface over an algebraically closed field such that $\operatorname{char}(k)\neq 2,3$.
The singularity types on $V$ are listed in the following table. Furthermore, $V$ is uniquely determined by its singularities,
with the exception of the cases $E_8, A_1+E_7, A_2+E_6$, which have two classes of isomorphism each, and the case $2D_4$, which has infinitely many
classes of isomorphism.

\begin{center}
\begin{tabular}{ c c c c c c c }
$A_1$ & $A_1 + A_2$ & $A_4$ & $2A_1 + A_3$ & $D_5$ & $A_1 + A_5$\\
$3A_2$ & $E_6$ & $3A_1+ D_4$ & $A_7$ & $A_1+D_6$ & $E_7$\\
$A_1+2A_3$ & $A_2+A_5$ & $D_8$ & $2A_1+D_6$ & $E_8$ & $A_1+E_7$\\
$A_1+A_7$ & $2A_4$ & $A_8$ & $A_1+A_2+A_5$ & $A_2+E_6$& $A_3+D_5$\\
$4A_2$ & $2A_1+2A_3$ & $2D_4$\\
\end{tabular}
\end{center}

\end{theorem}

\begin{proof}
The same proof as in \cite[Theorem 1.2]{yegorenstein} applies, using Theorem \ref{elliptic_classification}.
\end{proof}

\begin{notation}
We will denote any Gorenstein log del Pezzo surface by the corresponding singularity type.
For example, $S(A_1)$ denotes the unique Gorenstein log del Pezzo surface with only one singularity, of type $A_1$. Obviously, $S(A_1)$ is isomorphic to $\overline{\mathbb{F}}_2$.
\end{notation}

\begin{lemma}\label{minusonecurve}
Let $\operatorname{char}(k)\neq 2,3$ and let $S$ be a rank one Gorenstein log del Pezzo surface. Suppose that $\tilde{S}$ is obtained by contracting $(-1)$ curves
of the extremal rational elliptic surface $X$ as in Theorem \ref{reductiontheorem}. Then the $(-1)$ curves of $\tilde{S}$ are either images of elements of $\operatorname{MW}(X)$ or
images of the rational curves contained in the fibers of $X$.
\end{lemma}

\begin{proof}
Obvious by Theorem \ref{reductiontheorem} and the discussion after it.
\end{proof}

\begin{lemma}\label{cuspgorenstein}
Let $\operatorname{char}(k)\neq 2,3$ and let $S$ be a rank one log del Pezzo surface. Suppose that there is a rational cuspidal curve $C$ in the smooth locus of $S$. 
Then $S$ is Gorenstein. Furthermore:

\begin{enumerate}
\item If $\operatorname{char}(k)\neq 5$, then $S$ is obtained by blowing down $(-1)$ curves on the extremal rational elliptic surface with singular fibers $II$, $II^*$ and is 
one of the following: $S(E_8)$, $S(E_7)$, $S(E_6)$, $S(D_5)$, $S(A_4)$ or $S(A_1+A_2)$. On $S$ there is one and only one $(-1)$ curve. 

\item If $\operatorname{char}(k)=5$, then there is one more case: $S$ is obtained by blowing down one of the sections of the extremal rational elliptic surface with
singular fibers $I_5$, $I_5$, $II$. In this case $S=S(2A_4)$, which is fully described in Example \ref{sa4char5}.
\end{enumerate}

\end{lemma}

\begin{proof}
Notice that $C\equiv -K_S$ by adjunction, so $S$ is Gorenstein.
Every rank one Gorenstein log Del Pezzo surface is obtained by blowing down $(-1)$ curves in an extremal rational elliptic surface $X$ by Theorem \ref{reductiontheorem}.
Consider the strict transform $\tilde{C}$ of $C$ in $X$ and let $F$ be a general fiber. If $\tilde{C}\cdot F>1$ then by looking at the description of the contractions we see that
there is a $(-1)$ curve $D$ in $S$ such that $C\cdot D>1$, contradicting the fact that $K_S\cdot D=-1$. If $\tilde{C}\cdot F=1$ then $\tilde{C}$ is a section of the fibration
$X\rightarrow \mathbb{P}^1$, contradicting the fact that $\tilde{C}$ is cuspidal. Therefore we see that $\tilde{C}$ is a fiber of $X$. This concludes the proof by
Theorem \ref{elliptic_classification} and Lemma \ref{minusonecurve}.
\end{proof}

\begin{lemma}\label{nodegorenstein}
Let $\operatorname{char}(k)\neq 2,3$ and let $S$ be a rank one log del Pezzo surface. Suppose that:
\begin{enumerate}
\item There is a rational nodal curve $A$ in its smooth locus.
\item There are two rational curves $C$ and $D$ such that $K_S\cdot C = K_S\cdot D =-1$.
\item $K_S^2 \geqslant 4$. 
\end{enumerate}

Then $S$ is Gorenstein and $S=S(2A_1+A_3)$. Furthermore, $C\cap D\cap A=\emptyset$, 
$C$ and $D$ each pass through one of the $A_1$ points and meet at opposite ends of the $A_3$ point.
Finally, $C$ and $D$ are the only two $(-1)$ curves on $S$.
\end{lemma}

\begin{proof}
We have again that $A\equiv -K_S$ by adjunction, so $S$ is Gorenstein.
Notice that $(K_S+C+D)\cdot A \leqslant -2$, hence $C$ and $D$ are smooth by adjunction. Let $X$ be an extremal rational elliptic surface that dominates $\tilde{S}$ as in
Theorem \ref{reductiontheorem}. By Lemma \ref{minusonecurve} we have that $|\operatorname{MW}(X)|\geqslant 2$. Also, by the same reasoning as in Lemma \ref{cuspgorenstein},
we have that the strict transform of $A$ in $X$ is a fiber. By looking at the classification in Theorem \ref{elliptic_classification}, we see that starting with an
extremal rational elliptic surface $X$ with non trivial Mordell-Weil group and a fiber of type $I_1$, the only Gorenstein 
log del Pezzo surface with $K_S^2 \geqslant 4$ that we obtain is $S(2A_1+A_3)$. This case is achieved by taking $X$ to be the extremal rational elliptic surface with singular fibers
$III^*$, $I_2$, $I_1$. Now one can get the result by either looking at the description of the elliptic surface, or by following the argument in \cite[Lemma 3.9.2]{mckernan}.
\end{proof}

\subsection{Special examples}

Here we study more in detail some specific Gorenstein log del Pezzo surfaces that turn out to be especially important in the present paper. 

\begin{example} \label{sa4char5}
Let $\operatorname{char}(k)=5$ and consider $S=S(2A_4)$. This surface is obtained by taking $W$ to be the unique extremal rational elliptic surface
with singular fibers of type $I_5$, $I_5$, $II$, contracting one of the five sections and then contracting all the $(-2)$ curves. 
It is clear from this description that there is a unique rational cuspidal curve $D$ in its smooth locus. We would like to present in the following a slightly more explicit approach.

Consider the following four points in $\mathbb{P}^2_k$: $a=[-1,1,1]$, $b=[-1,-1,1]$, $c=[1,-1,1]$ and $d=[1,1,1]$. Let $L_{ab}$, $L_{ac}$, $L_{ad}$, $L_{bc}$, $L_{bd}$ and $L_{cd}$ be the
six lines between them. Consider the cubics $C_0 = L_{ad} + L_{ac} + L_{bc}$ and $C_\infty=L_{ab} + L_{bd} + L_{cd}$. The equation of $C_0$ is 
$(Y^2 - Z^2)(X+Y)=0$ and the equation of $C_\infty$ is $(X^2-Z^2)(Y-X)=0$. Consider the pencil of cubics $C_t$ generated by $C_0$ and $C_\infty$, namely the curves
$C_t$ whose equations are

\[
(Y^2 - Z^2)(X+Y)+ t(X^2-Z^2)(Y-X) = 0
\]

Clearly the base locus of this pencil is given by the points $a$, $b$, $c$, $d$ counted with multiplicity two, and by the point $f=[0,0,1]$, counted with multiplicity one.
We claim that the curve $C_2$ is a rational cuspidal curve, with cusp at $[1,3,0]$. First, it's immediate to check that $C_2$ passes through $[1,3,0]$.
Now we study $C_2$ in the chart $X\neq 1$, with local coordinates $y=Y/X$ and $z=Z/X$. Its equation is 
\[
 (y^2 -z^2)(y+1) + 2(1-z^2)(y-1) = 0
\]

The Jacobian and the Hessian of $C_2$ in these coordinates are 
\[
  J=(3y^2 + 2z^2 + 2y + 2, -yz+2z)
\]
and

\[
   H=
   \left[ {\begin{array}{cc}
   y+2 & 4z \\
   4z & -y+2 \\
  \end{array} } \right]
\]
respectively. Evaluating $J$ and $H$ in $(3,0)$ shows that $C_2$ has a cusp.

Now we resolve the base locus of the pencil at $a$, $b$, $c$, $d$. Since the multiplicity of the base locus at these points is two, we need two blow ups at each point. 
We denote by $E_1 ^a$ the first blow up over $a$,
and we use the analogous notation for the other points. We also slightly abuse notation by denoting the strict transforms of the lines and of the cubics considered so far with their original names.
Consider the blow up $\tilde{S} \rightarrow \mathbb{P}^2$ on $E_1 ^a \cap L_{ab}$, $E_1 ^b \cap L_{bc}$, $E_1 ^c \cap L_{cd}$ and $E_1 ^d \cap L_{ad}$ and call $E_2^a$, $E_2^b$,
$E_2^c$ and $E_2^d$ the respective exceptional divisors.
Notice that the curves $E_1^a$, $L_{ad}$, $L_{bc}$ and $E_1 ^c$ form a chain of four $(-2)$ curves. Similarly the curves $E_1^d$, $L_{cd}$, $L_{ab}$ and $E_1^b$ form another chain of $(-2)$ curves. Our surface $S$ is therefore obtained by contracting these two chains to two points, $p$ and $q$ respectively. 
The strict transform of $C_2$ in $S$ is contained in the smooth locus and is a member of
$|-K_S|$ by adjunction. The strict transforms of $L_{ac}$ and $L_{bd}$ are the only two $(-1)$ curves in $|-K_S|$. There are exactly four other $(-1)$ curves on $S$
by Lemma \ref{minusonecurve}: the images of $E_2 ^a$, $E_2 ^b$, $E_2 ^c$ and $E_2 ^d$.

Next we study the geometry of these curves with respect to the singular points on $S$. The curve $L_{ac}$ is nodal, with a node at $p$. The two branches of the node touch opposite
curves of the $A_4$ chain singularity at $p$. Similarly $L_{bd}$ is nodal with a node at $q$. Notice that $L_{ac}\cap L_{bd}$ is the only point in the base locus of $|-K_S|$, which is
the image of $[0,0,1]$. For the remaining $(-1)$ curves we notice that they are smooth, pass through both $p$ and $q$, are dlt at one but not lc at the other. Furthermore, given any
two such curves, $p$ and $q$ are the only points where they intersect. 

\end{example}

\begin{example} Let $\operatorname{char}(k)\neq 2,3,5$ and $S=S(2A_4)$. Most of the analysis in Example \ref{sa4char5} carries through. The only difference is that
$S$ contains exactly two singular rational curves in its smooth locus, both with nodal singularities.
\end{example}

\begin{example} Let $\operatorname{char}(k)\neq 2,3$ and $S=S(A_1+A_2)$. Then $S$ is obtained as follows. Consider the flex cubic $C$ given by $XZ^2 = Y^3$ in $\mathbb{P}_k ^2$
and the line $L$ given by $Y=0$. Blow up three times over the intersection $L\cap C$ to separate them, then contract the three $(-2)$ curves obtained in the process. Notice
that the strict transform of $C$ in $S$ is a cuspidal rational curve contained in the smooth locus. The image of the last exceptional curve created by blowing up at $L\cap C$ is a
$(-1)$ curve. This is also the unique $(-1)$ curve, for example by Lemma \ref{cuspgorenstein}.
\end{example}

\bibliography{biblio2}

\begin{thebibliography}{CTW17}

\bibitem[Bel09]{belousov}
Grigory Belousov.
\newblock The maximal number of singular points on log del {P}ezzo surfaces.
\newblock {\em J. Math. Sci. Univ. Tokyo}, 16(2):231--238, 2009.

\bibitem[Ber17]{fabio}
Fabio Bernasconi.
\newblock Kawamata-viehweg vanishing fails for log del pezzo surfaces in
  characteristic 3, 2017.

\bibitem[CTW17]{cascini}
Paolo Cascini, Hiromu Tanaka, and Jakub Witaszek.
\newblock On log del {P}ezzo surfaces in large characteristic.
\newblock {\em Compos. Math.}, 153(4):820--850, 2017.

\bibitem[Dem80]{demazure}
M~Demazure.
\newblock Surfaces de del pezzo—iii—positions presque generales.
\newblock {\em S{\'e}minaire sur les Singularit{\'e}s des Surfaces}, pages
  36--49, 1980.

\bibitem[Fur86]{furushima}
Mikio Furushima.
\newblock Singular del {P}ezzo surfaces and analytic compactifications of
  {$3$}-dimensional complex affine space {${\bf C}^3$}.
\newblock {\em Nagoya Math. J.}, 104:1--28, 1986.

\bibitem[Har77]{hartshorneag}
Robin Hartshorne.
\newblock {\em Algebraic geometry}.
\newblock Springer-Verlag, New York-Heidelberg, 1977.
\newblock Graduate Texts in Mathematics, No. 52.

\bibitem[HX15]{haconxu}
Christopher~D. Hacon and Chenyang Xu.
\newblock On the three dimensional minimal model program in positive
  characteristic.
\newblock {\em J. Amer. Math. Soc.}, 28(3):711--744, 2015.

\bibitem[JLR12]{lang3}
Tyler~J. Jarvis, William~E. Lang, and Jeremy~R. Ricks.
\newblock Integral models of extremal rational elliptic surfaces.
\newblock {\em Comm. Algebra}, 40(10):3867--3883, 2012.

\bibitem[KM99]{mckernan}
Se\'{a}n Keel and James McKernan.
\newblock Rational curves on quasi-projective surfaces.
\newblock {\em Mem. Amer. Math. Soc.}, 140(669):viii+153, 1999.

\bibitem[Kol97]{flips}
J\'{a}nos Koll\'{a}r.
\newblock Singularities of pairs.
\newblock In {\em Algebraic geometry---{S}anta {C}ruz 1995}, volume~62 of {\em
  Proc. Sympos. Pure Math.}, pages 221--287. Amer. Math. Soc., Providence, RI,
  1997.

\bibitem[Lan91]{lang1}
William~E. Lang.
\newblock Extremal rational elliptic surfaces in characteristic {$p$}. {I}.
  {B}eauville surfaces.
\newblock {\em Math. Z.}, 207(3):429--437, 1991.

\bibitem[Lan94]{lang2}
William~E. Lang.
\newblock Extremal rational elliptic surfaces in characteristic {$p$}. {II}.
  {S}urfaces with three or fewer singular fibres.
\newblock {\em Ark. Mat.}, 32(2):423--448, 1994.

\bibitem[Mir89]{miranda}
Rick Miranda.
\newblock {\em The basic theory of elliptic surfaces}.
\newblock Dottorato di Ricerca in Matematica. [Doctorate in Mathematical
  Research]. ETS Editrice, Pisa, 1989.

\bibitem[MP86]{mirandapersson}
Rick Miranda and Ulf Persson.
\newblock On extremal rational elliptic surfaces.
\newblock {\em Math. Z.}, 193(4):537--558, 1986.

\bibitem[MZ88]{zhang1}
M.~Miyanishi and D.-Q. Zhang.
\newblock Gorenstein log del {P}ezzo surfaces of rank one.
\newblock {\em J. Algebra}, 118(1):63--84, 1988.

\bibitem[MZ93]{zhang2}
M.~Miyanishi and D.-Q. Zhang.
\newblock Gorenstein log del {P}ezzo surfaces. {II}.
\newblock {\em J. Algebra}, 156(1):183--193, 1993.

\bibitem[Nak07]{nakayama}
Noboru Nakayama.
\newblock Classification of log del {P}ezzo surfaces of index two.
\newblock {\em J. Math. Sci. Univ. Tokyo}, 14(3):293--498, 2007.

\bibitem[SS10]{shioda}
Matthias Sch\"{u}tt and Tetsuji Shioda.
\newblock Elliptic surfaces.
\newblock In {\em Algebraic geometry in {E}ast {A}sia---{S}eoul 2008},
  volume~60 of {\em Adv. Stud. Pure Math.}, pages 51--160. Math. Soc. Japan,
  Tokyo, 2010.

\bibitem[Tot19]{totaro}
Burt Totaro.
\newblock The failure of {K}odaira vanishing for {F}ano varieties, and terminal
  singularities that are not {C}ohen-{M}acaulay.
\newblock {\em J. Algebraic Geom.}, 28(4):751--771, 2019.

\bibitem[Ye02]{yegorenstein}
Qiang Ye.
\newblock On {G}orenstein log del {P}ezzo surfaces.
\newblock {\em Japan. J. Math. (N.S.)}, 28(1):87--136, 2002.

\end{thebibliography}
\bibliographystyle{alpha}

\end{document}